\documentclass[reqno]{amsart}

\usepackage{amssymb,amsmath}
\usepackage{mathrsfs}

\usepackage[thinlines]{easytable}
\usepackage{longtable}
\usepackage{array}
\usepackage{accents}
\usepackage{relsize}
\usepackage{diagbox}

\usepackage{a4wide}
\usepackage{graphicx}
\usepackage{hyperref}
\usepackage[mathscr]{euscript}

\usepackage{mathtools}

\usepackage{subfigure}
\usepackage{calc}
\usepackage{tikz}

\usepackage{subfigure}

\usetikzlibrary{cd,arrows,matrix}
\usepackage{newtxtext}
\usepackage{newtxmath}
\usepackage{anyfontsize}
\usepackage{bm}

\usepackage[english]{babel}

\usepackage[makeroom]{cancel}

\hypersetup{colorlinks=true,linkcolor=blue}

\usepackage{enumitem}

\usepackage{multicol}

\numberwithin{equation}{section}

\theoremstyle{plain}
\newtheorem{theorem}{Theorem}[section]
\newtheorem{lemma}[theorem]{Lemma}
\newtheorem{proposition}[theorem]{Proposition}
\newtheorem{corollary}[theorem]{Corollary}

\newtheorem{definition/proposition}[theorem]{Definition/Proposition}
\newtheorem{theorem/definition}[theorem]{Theorem/Definition}
\newtheorem*{bigtheorem*}{Main Theorem}

\theoremstyle{definition}
\newtheorem{definition}[theorem]{Definition}
\newtheorem{example}[theorem]{Example}
\newtheorem{example/definition}[theorem]{Example/Definition}

\newtheorem{conjecture}[theorem]{Conjecture}
\newtheorem*{conjecture*}{Conjecture}

\theoremstyle{remark}
\newtheorem{remark}[theorem]{Remark}

\makeatletter
\newcommand{\setword}[2]{%
  \phantomsection
  #1\def\@currentlabel{\unexpanded{#1}}\label{#2}%
}
\makeatother

\newcommand{\field}{\mathbf{k}}

\newcommand{\Spec}{\mathrm{Spec}}

\newcommand{\cO}{\mathcal{O}}

\newcommand{\isomorphic}{\cong}

\newcommand{\rank}{\mathrm{rk}}

\newcommand{\fm}{\mathfrak{m}}

\newcommand{\Hilb}{\mathrm{Hilb}}

\newcommand{\Gr}{\mathrm{Gr}}

\newcommand{\an}{\mathrm{an}}

\newcommand{\cP}{\mathcal{P}}

\newcommand{\alg}{\mathrm{alg}}

\newcommand{\fw}{\mathfrak{w}}

\newcommand{\ord}{\mathrm{ord}}

\newcommand{\GR}{\mathbf{Gr}}

\newcommand{\fg}{\mathfrak{g}}

\newcommand{\Spr}{\mathrm{Spr}}

\newcommand{\mH}{\mathrm{H}}

\newcommand{\cH}{\mathcal{H}}

\newcommand{\mot}{\mathrm{mot}}

\newcommand{\pt}{\mathrm{pt}}

\newcommand{\Gaps}{\mathrm{Gaps}}

\newcommand{\cI}{\mathcal{I}}

\DeclarePairedDelimiter\floor{\lfloor}{\rfloor}

\newcommand{\Adm}{\mathcal{A}dm}

\newcommand{\Young}{\mathcal{Y}}

\newcommand{\id}{\mathrm{id}}

\newcommand{\GC}{\mathcal{GC}}

\newcommand{\cG}{\mathcal{G}}

\newcommand{\cR}{\mathcal{R}}

\newcommand{\AJ}{\mathrm{AJ}}

\newcommand{\im}{\mathrm{Im}}

\newcommand{\Hom}{\mathrm{Hom}}

\newcommand{\Gen}{\mathrm{Gen}}

\newcommand{\Poly}{\mathrm{Poly}}

\newcommand{\cJ}{\mathcal{J}}

\newcommand{\SM}{\mathrm{SM}}

\newcommand{\transpose}{\top}

\newcommand{\Ext}{\mathrm{Ext}}

\newcommand{\RHom}{\mathrm{RHom}}

\newcommand{\dinv}{\mathrm{dinv}}

\newcommand{\codinv}{\mathrm{co}\text{-}\mathrm{dinv}}

\newcommand{\basevector}{\mathfrak{v}}

\newcommand{\uformn}{\mathbf{u}}
\newcommand{\vformn}{\mathbf{v}}

\newcommand{\armlength}{\mathfrak{a}}
\newcommand{\leglength}{\mathfrak{l}}

\newcommand{\moveindex}{\mathbf{p}}

\newcommand{\permutation}{\mathbf{s}}

\newcommand{\permutationindex}{\mathbf{i}}

\newcommand{\correctionterm}{\varepsilon}

\newcommand{\Weight}{\mathrm{W}}

\newcommand{\specialrank}{{\mathrm{r}_0}}

\newcommand{\length}{\mathrm{\ell}}

\newcommand{\minval}{\mathrm{e}}

\newcommand{\constant}{\mathbf{c}}

\newcommand{\LHS}{\mathrm{LHS}}
\newcommand{\RHS}{\mathrm{RHS}}

\newcommand{\cD}{\mathcal{D}}

\def\acts{\curvearrowright}

\newcommand{\red}{\mathrm{red}}

\newcommand{\cC}{\mathcal{C}}

\newcommand{\cF}{\mathcal{F}}

\newcommand{\diag}{\mathrm{diag}}

\begin{document}

\title[Purity of generalized affine Springer fibers from generic planar curve singularities]{Purity of generalized affine Springer fibers from generic planar curve singularities}

\author{Taiwang Deng}
\email{dengtaiw@bimsa.cn}
\address{Beijing Institute of Mathematical Sciences and Applications\\
Huairou District, 101408, Beijing, China}

\author{Tao Su}
\email{sutao@bimsa.cn}
\address{Beijing Institute of Mathematical Sciences and Applications\\
Huairou District, 101408, Beijing, China}

\date{}



\begin{abstract}
We prove the cohomological purity of punctual Hilbert schemes of points on generic irreducible planar curve singularities, by constructing an explicit affine paving. Via their identification with generalized $GL_N$-affine Springer fibers attached to the direct sum of the adjoint and standard representations, this establishes a new case of the purity conjecture for generalized affine Springer fibers. The combinatorics of the paving - cell indices and dimensions - are controlled by $(dn,dm)$-Dyck paths extending results of Gorsky-Mazin-Oblomkov on compactified Jacobians. As a byproduct, we also give a simpler proof of their bijection between admissible $(dn,dm)$-invariant subsets and $(dn,dm)$-Dyck paths.
\end{abstract}

\maketitle

\tableofcontents

\section{Introduction}

Affine Springer fibers were introduced by Kazhdan–Lusztig \cite{KL88}, where the basic properties were studied. Goresky–Kottwitz–MacPherson \cite{GKM04} later used them to study orbital integrals on reductive groups over local fields, relating these integrals to point counts of affine Springer fibers. They also formulated the \textbf{purity conjecture} for affine Springer fibers and, under this assumption, established the function field analogue of the fundamental lemma \cite{LS87} in the unramified case - a combinatorial identity appearing in \cite{Lab79}, \cite{Lan83}, and \cite{LS87} which describes the relations between orbital integrals for the unit of the Hecke algebras on different groups. They further confirmed the conjecture in the equivalued case \cite{GKM06}.
Via the Grothendieck-Lefschetz trace formula, purity ensures that orbital integrals can be expressed in terms of the cohomology of affine Springer fibers.
The general purity conjecture, however, remains open. 
It was precisely this obstacle that Ng\^{o} overcame in his proof of the fundamental lemma \cite{Ngo10}, by passing to the global setting of the Hitchin fibration and applying the decomposition and support theorems.

As a natural extension, Goresky-Kottwitz-MacPherson \cite{GKM06} introduced \emph{generalized affine Springer fibers (GASF)} by replacing the adjoint representation with an arbitrary finite-dimensional representation of a reductive group. 
More recently, these varieties became prominent as natural geometric sources for realizing representations \cite{HKW23,GK23} of (quantized) Coulomb branch algebras \cite{BFN18}.
It is expected that the purity conjecture extends to this broader setting:

\begin{conjecture}\label{conj:purity_conjecture}
Under mild conditions, the (co)homology of a generalized affine Springer fiber is pure.
\end{conjecture}

Partial progress has been made in various settings (see~\cite{Che14,GSV24,Che24,GMO25}), but the general case remains wide open. 
In particular, in~\cite{GSV24} the authors study the GASF for $G=GL_n^{\times r}$ and the representation $V = \Hom(\field^n,\field^n)^{\oplus r} \oplus \field^n$. 
They show that in this case the GASF corresponds to the compositional Hilbert scheme of quasi-homogeneous planar curve singularities ($\{y^n-x^m=0\}$)~\cite[Prop.~8.9]{GSV24}, and prove that it admits an affine paving~\cite[Prop.~8.12]{GSV24}.
Finally, we expect the generalized purity conjecture to play a role in the geometry of the relative fundamental lemma, as in the classical case.

\subsection*{Main result}

In this article we prove Conjecture~\ref{conj:purity_conjecture} for a new family of generalized affine Springer fibers arising from \emph{generic irreducible planar curve singularities}
$(C,0)$.
By \cite[Thm.~3.3]{GK23}, these fibers are naturally identified with the (punctual) Hilbert schemes $\Hilb^{[\bullet]}(C,0)$.

Our first main result is:

\begin{theorem}[Cor.~\ref{cor:affine_paving_for_Hilbert_scheme}, Thm.~\ref{thm:affine_paving_for_Hilbert_scheme_over_generic_curve_singularity}]\label{thm:main_result_1}
Let $(C,0)$ be a \textbf{generic} irreducible planar curve singularity: 
\[
x = t^{dn},\quad y = t^{dm} + \lambda t^{dm+1} + \text{ h.o.t.}\footnote{`h.o.t.' stands for `higher $t$-order terms'.},\quad \gcd(n,m)=1,~~\lambda \in \field^{\times}.
\]
Then 
$\Hilb^{[\bullet]}(C,0) \isomorphic \Spr_{\basevector}$ admits an explicit affine paving:
\[
\Hilb^{[\tau]}(C,0) = \sqcup_{\Delta\in\Adm(dn,dm):~\tau-e(\Delta)\in\Delta} H_{\Delta}^{[\tau]},\quad H_{\Delta}^{[\tau]} \isomorphic \mathbb{A}^{\dim\Delta - |\Gaps(\tau - e(\Delta))|}.
\]
In particular, its (co)homology is pure.
\end{theorem}
The relevant terminology will be recalled in the next section. 
The construction of our paving extends the approach of \cite{GMO25}, which established affine pavings and purity for (local) compactified Jacobians of $(C,0)$ as above. 
In this setting, the correspondence with equivalued affine Springer fibers in \cite{GKM06} holds precisely when $d = 1$, i.e., when the curve singularity is quasi-homogeneous (see Remark \ref{rem:equivalued_planar_curve_singularity_is_quasi-homogeneous}).

Our second main result gives a new proof of the bijection between $\Adm(dn,dm)$, the set of admissible $(dn,dm)$-invariant subsets, and $\Young(dn,dm)$, the set of $(dn,dm)$-Dyck paths, originally established in \cite{GMO25,GMV20}. See Theorem~\ref{thm:admissible_invariant_subsets_vs_Dyck_paths} for a precise reformulation and a refinement.
Our constructions of the bijection $\Psi_d:\Young(dn,dm) \to \Adm(dn,dm)$ (Section~\ref{subsubsec:Dyck_paths_vs_admissible_invariant_subsets}, Proposition~\ref{prop:from_Dyck_paths_to_admissible_invariant_subsets}) and its inverse $\Phi_d:\Adm(dn,dm) \to \Young(dm,dm)$ (Proposition~\ref{prop:from_admissible_invariant_subsets_to_Dyck_paths}) are much simpler than the originals, and may be of independent interest.
While a priori it is not clear that the two constructions coincide, we prove that $\Psi_d$ is compatible with the sweep map (Proposition~\ref{prop:sweep_map_vs_bijection}), thereby establishing the equivalence (Remark~\ref{rem:compare_with_the_old_bijection}).

Our third result describes the combinatorics of the paving - cell indices and dimensions - in terms of $(dn,dm)$-Dyck paths, generalizing \cite{GMO25} on compactified Jacobians $\overline{J}(C,0)$. See Proposition~\ref{prop:combinatorics_of_affine_paving_for_Hilbert_scheme}.

Finally, our original motivation comes from the Oblomkov-Rasmussen-Shende (ORS) conjecture \cite{OS12,ORS18} and the Cherednik-Danilenko-Philipp (CDP) conjecture \cite{CD16,CP18}, which connect enumerate geometry of planar curve singularities with knot invariants and double affine Hecke algebras (DAHA).
Concretely, let $L(q,t,a)$ be the \textbf{Hilbert $L$-function} of $(C,0)$ (Definition~\ref{def:Hilbert_L-function}), the generating function for the cohomology of $\Hilb^{[\bullet]}(C,0)$. 
Then, the ORS conjecture states the following:
\begin{conjecture}[see {\cite[Conj.~2]{ORS18}} or Conjecture~\ref{conj:ORS_conjecture}]
Up to a suitable normalization, $L(q,t,a)$ is the generating function for the triply graded HOMFLY-PT homology of the singularity link of $(C,0)$.
\end{conjecture}


For generic $(C,0)$, our affine paving allows a computation of $L(q,t,0)$.
Moreover, $L(q,t,0)$ encodes the \textbf{perverse filtration} on the cohomology of $\overline{J}(C,0)$ \cite{MS13,MY14}. 
As a consequence, we compute the perverse filtration for generic $(C,0)$ (Cor.~\ref{cor:perverse_filtration_for_generic_planar_curve_singularities}).

On the other hand, let $\mH^{\mot}(q,t,a)$ denote the \textbf{normalized motivic superpolynomial} of $(C,0)$ (Definition~\ref{def:normalized_motivic_superpolynomial}), originating from DAHA representations. 
An abridged version of the Cherednik-Danilenko-Philipp conjecture asserts:

\begin{conjecture}[\cite{CD16,CP18}]
$L(q,t,a) = \mH^{\mot}(q,t,a)$.
\end{conjecture}

In the setting of a generic irreducible planar curve singularity $(C,0)$, our affine paving reduces the $(q,t)$-version of this conjecture to the following purely combinatorial statement:

\begin{conjecture}[Conj.~\ref{conj:combinatorial_two-variable_CDP_conjecture_for_generic_curve_singularity}]
\begin{equation}
 (1-t)\sum_{D\in\Young(dn,dm)}\sum_{z\in \overline{D}} q^{|D| - \dinv(D) + \hat{\epsilon}_D(z)} t^{|D| + \hat{\rank}_D(z)}
 =
 \sum_{D\in\Young(dn,dm)} q^{\codinv(D)} t^{\codinv(D) + \delta - |D|}.
 \end{equation}
\end{conjecture}
Finally, in view of \cite{CGHM24}, the lowest $a$-degree part of the ORS conjecture is likewise equivalent to this combinatorial conjecture.
See Section~\ref{subsec:connection_to_ORS_and_CDP} for more details.

\subsection*{Acknowledgements}
\addtocontents{toc}{\protect\setcounter{tocdepth}{1}}

This project was initiated following a talk by Ivan Cherednik at ICBS 2024 (BIMSA).
We are grateful to Eugene Gorsky for pointing out the relevance of \cite{CGHM24} to our work, and to Vivek Shende and Penghui Li for helpful discussions.
The first author gratefully acknowledges the hospitality of the Department of Mathematics at the National University of Singapore during a visit, and especially thanks Professor Lei Zhang for the invitation. He was supported by BJNSF No. 1244042 and NSFC No. 12401013.

\addtocontents{toc}{\protect\setcounter{tocdepth}{2}}

\section{Preliminaries}

\subsection{Planar curve singularities}

Let $\field=\mathbb{C}$ be the base field.
Consider a reduced irreducible planar curve singularity $(C,0)$, which after an analytic coordinate change, admits a parametrization of the form (see \cite[I.Cor.~3.8]{GLS07}):
\begin{equation}\label{eqn:puiseux_parametrization}
x(t) = t^{dn},\quad y(t) = t^{dm} + \text{ h.o.t. },\quad m > n,~~\gcd(n,m) = 1.
\end{equation}
Here, `h.o.t.' stands for `higher order terms'.

The local defining equation of $(C,0)$ can be written as $f(x,y)=0$, with (see \cite[Prop.~I.3.4]{GLS07})
\[
f(x,y):= \prod_{j=0}^{dn-1}(y - y(\zeta_{dn}^jx^{\frac{1}{dn}}))\in \field\{x\}[y],\quad \zeta_{dn}:= e^{\frac{2\pi\sqrt{-1}}{dn}},
\]
which is monic and irreducible in $y$ of degree $dn$.
The relevant reductive group will be $G=GL_{dn}$.

Let $R=\hat{\cO}_{C,0}=\field[[x(t),y(t)]]$ be the completed local ring, with normalization $\tilde{R}=\field[[t]]$ and quotient field $E=\field((t))$.
Set $\cO = \cO_F =\field[[x=t^{dn}]]$ with fraction field $F=\field((x))$.
Then $R$ is a free $\cO$-module of rank $dn$:
\[
R = \oplus_{j=0}^{dn-1}\cO \cdot y^j.
\]

Let $\gamma\in \mathfrak{gl}_{dn}(F)$ denote the matrix of multiplication by $y$ relative to the basis $(1,y,\dots,y^{dn-1})$. Explicitly,
\begin{equation}
\gamma = \left(\begin{array}{cccc}
0 & \hdots & 0 & a_0\\
1 & \ddots & \vdots & a_1\\
 & \ddots & 0 & \vdots\\
  & & 1 & a_{dn-1}
\end{array}\right),\quad f(x,y) = y^{dn} - a_{dn-1}y^{dn-1} - \cdots - a_0,\quad a_k(x) \in x\field[[x]].
\end{equation}

\begin{definition}[{\cite[Thm.~1.1]{GMO25}}]\label{def:generic_planar_curve_singularities}
The planar curve singularity $(C,0)$ is called \textbf{generic}\footnote{See Remark \ref{rem:finite_determinancy_of_planar_curve_singularities} for a justification.} if it admits a parametrization:
\[
x(t) = t^{dn},\quad y(t) = t^{dm} + \lambda t^{dm+1} + \text{ h.o.t. },\quad \lambda\in\field^{\times}.
\]
\end{definition}
\noindent{}\textbf{Note}: For $d = 1$, up to a re-parametrization, this is the quasi-homogeneous singularity $(\{y^n = x^m\},0)$.

\subsection{Hilbert schemes of points}

We briefly recall the notion of generalized affine Springer fibers (GASF).  
As above, let $\cO=\field[[x]]$ and $F=\field((x))$.  
For a reductive group $G$ over $\field$ and a finite-dimensional representation $V$ of $G$, the affine Grassmannian is $\GR_G := G(F)/G(\cO)$.

\begin{definition}
For $\basevector\in V(F)$, the \emph{generalized affine Springer fiber (GASF)} associated with $(\basevector,G,V)$ is the closed sub-ind-scheme of $\GR_G$ defined by
\[
\Spr_\basevector(A) := 
\{g\in G(A((t))) \mid g^{-1}\cdot \basevector \in V(A[[t]])\}/G(A[[t]]),\quad A\in \field\text{-alg}.
\]
\end{definition}

When $V=\fg$ is the adjoint representation, $\Spr_\basevector$ reduces to the usual affine Springer fiber (ASF).

\begin{definition}
The (punctual) Hilbert scheme of $\tau$ points on $(C,0)$ is the moduli space
\[
\Hilb^{[\tau]}(C,0)(A) := \{\text{ideals of colength $\tau$ in } A[[x,y]]/(f)\},\quad A\in\field\text{-alg}.
\]
The full Hilbert scheme is
\[
\Hilb^{[\bullet]}(C,0) := \sqcup_{\tau\geq 0} \Hilb^{[\tau]}(C,0).
\]
\end{definition}

\begin{lemma}[{\cite[Thm.~3.3]{GK23}}]
Let $G=GL_{dn}$ and $V:=\fg\oplus\field^{dn}$, the direct sum of the adjoint and standard representations.  
Set $\basevector:=(\gamma^{\transpose},e_{dn})$, with $e_{dn}=(0,\dots,0,1)^{\transpose}\in\field^{dn}$. Then
\[
\Spr_\basevector \cong \Hilb^{[\bullet]}(C,0).
\]
\end{lemma}

This explains the terminology in Theorem \ref{thm:main_result_1}, except for the affine paving.  
\textbf{From now on}, by abuse of notation, we identify a $\field$-scheme $X$ with its set of $\field$-points $X(\field)$.

\subsection{Semigroup and semimodules of valuations}

Recall that 
\[
\delta = \delta(R):=\dim \tilde{R}/R < \infty \quad \text{(\textbf{$\delta$-invariant})}.
\] 
Since $R$ is Gorenstein,
\[
c(R):= \min\{c\in\mathbb{Z}\mid t^c\tilde{R}\subset R\} = 2\delta \quad \text{(\textbf{conductor})}.
\]

Let $\ord_t:E\to \mathbb{Z}\sqcup\{\infty\}$ be the $t$-adic valuation.  

\begin{definition}
For a nonzero finitely generated $R$-submodule $J\subset E$, set
\[
\ord_t(J):=\{\ord_t(z): z\in J\setminus\{0\}\}\subset\mathbb{Z}.
\]
In particular,
\begin{equation}
\Gamma:=\ord_t(R)\subset\mathbb{N}
\end{equation}
is a \textbf{semigroup} of valuations, and $\ord_t(J)$ is a $\Gamma$-\textbf{semimodule}\footnote{Here a $\Gamma$-semimodule $\Delta\subset\mathbb{Z}$ means: $a\in\Gamma$, $b\in\Delta$ $\Rightarrow$ $a+b\in\Delta$.}.
\end{definition}

For any nonempty $\Gamma$-semimodule $S\subset\mathbb{Z}$, we have $\min S\in\mathbb{Z}$ $\Leftrightarrow$ $S\neq \mathbb{Z}$, since $s\in S$ implies
\[
s+2\delta+\mathbb{N}\subset s+\Gamma\subset S.
\]
In particular, this is the case for $J\subset E$ as above:
Say, $J = Rz_1+\cdots + Rz_r$ with $z_i \in E^{\times}$. Then
\[
\ord_t(z)\geq \min_i\ord_t(z_i) > -\infty,
\]
for any $z=\sum_i a_i z_i\in J$, $a_i\in R$.

\begin{definition}
For a nonempty $\Gamma$-semimodule $\Delta\subsetneq\mathbb{Z}$, define the \textbf{conductor}
\[
c(\Delta):=\min\{c\in\mathbb{Z}\mid c+\mathbb{N}\subset\Delta\}.
\]
Clearly $\min\Delta\leq c(\Delta)\leq \min\Delta+2\delta$.  
If $\Delta\subset\mathbb{N}$, define the \textbf{$\delta$-invariant}
\[
\delta(\Delta):=|\mathbb{N}\setminus\Delta|.
\]
Finally, $\Delta$ is called \textbf{$0$-normalized} if $\min\Delta=0$.
\end{definition}

By definition, for a nonzero finitely generated $R$-submodule $J\subset E$, we have
\[
c(\ord_t(J))=c(J):=\min\{c\in\mathbb{Z}\mid t^c\tilde{R}\subset J\}.
\]
If $J\subset\tilde{R}$, then $\delta(\ord_t(J))=\delta(J):=\dim \tilde{R}/J$.  
In particular $\delta(\Gamma)=\delta$ and $c(\Gamma)=c(R)=2\delta$.

\begin{remark}\label{rem:finite_determinancy_of_planar_curve_singularities}
In the parametrization \eqref{eqn:puiseux_parametrization} for $(C,0)$, by a further formal coordinate change $(x,y)\mapsto(x,y+h(x,y))$ with $h\in \langle x,y\rangle^2\subset\field[[x,y]]$, we may assume
\[
y(t)=t^{dm}+\sum_{\substack{j\geq 1 \\ dm+j\notin \Gamma}} \lambda_j t^{dm+j},\quad \lambda_j\in\field.
\]
In particular, $\lambda_j\neq 0\Rightarrow dm+j\leq 2\delta-1$. This does not affect $R$.

For fixed $\delta$, among parametrizations
\[
x(t)=t^{dn},\quad y(t)=t^{dm}+\sum_{1\leq j\leq 2\delta-1}\lambda_j t^{dm+j},
\]
the \emph{generic} singularities (those with $\lambda=\lambda_1\neq 0$) form a Zariski open dense subset.  
This justifies Definition \ref{def:generic_planar_curve_singularities}.
\end{remark}

\begin{lemma}\label{lem:estimates_for_delta-invariant_and_conductor}
For any nonempty $\Gamma$-semimodule $\Delta\subset\mathbb{N}$:
\[
c(\Delta)\leq \delta+\delta(\Delta).
\]
If $\Delta$ is $0$-normalized then
\[
\delta(\Delta)\leq \delta,
\]
with equality $\delta(\Delta)=\delta$ iff $\Delta=\Gamma$ iff $c(\Delta)=\delta+\delta(\Delta)$.
\end{lemma}

\begin{proof}
Since $c(\Delta)-1\notin\Delta$ and $\Delta$ is a $\Gamma$-semimodule, we have
\[
a\in [0,c(\Delta)-1]\cap\Gamma \ \Rightarrow\ c(\Delta)-1-a\notin\Delta.
\]
Hence
\[
\delta(\Delta)=|[0,c(\Delta)-1]\setminus\Delta|\geq |[0,c(\Delta)-1]\cap\Gamma|\geq c(\Delta)-\delta,
\]
which gives $c(\Delta)\leq \delta+\delta(\Delta)$.

If $\Delta$ is $0$-normalized, then $\Gamma\subset\Delta$ and $\delta(\Delta)\leq |\mathbb{N}\setminus\Gamma|=\delta$. Equality holds iff $\Delta=\Gamma$.  
Finally, the condition $c(\Delta)=\delta+\delta(\Delta)$ forces $c(\Delta)=2\delta$ and $\delta(\Delta)=\delta$, hence $\Delta=\Gamma$. The converse is immediate.
\end{proof}

\begin{definition}\label{def:semimodule_for_compactified_Jacobian}
For any $0$-normalized $\Gamma$-semimodule $\Delta\subset\mathbb{N}$, set
\[
e(\Delta):=\delta-\delta(\Delta)\in[0,\delta],\qquad 
\overline{\Delta}:=e(\Delta)+\Delta\subset e(\Delta)+\mathbb{N}.
\]
\end{definition}

In particular, $\min\overline{\Delta} = e(\Delta)$, and
\begin{equation}
\delta(\overline{\Delta}) = |\mathbb{N} - \overline{\Delta}| = e(\Delta) + |(e(\Delta)+\mathbb{N}) - (e(\Delta)+\Delta)| = e(\Delta) + \delta(\Delta) = \delta;\quad
c(\overline{\Delta}) = c(\Delta) + e(\Delta) \leq 2\delta.
\end{equation}

Conversely, any $\Gamma$-semimodule $\overline{\Delta}\subset\mathbb{N}$ with $\delta(\overline{\Delta})=\delta$ arises as $\overline{\Delta}=e(\Delta)+\Delta$ for a unique $0$-normalized $\Delta$, with $e(\Delta)=\min\overline{\Delta}$.

\begin{definition}\label{def:dn-generators}
For any bounded below $dn$-invariant subset $\Delta \subset \mathbb{Z}$, define $A(\Delta)$ as the set of \textbf{$dn$-generators} of $\Delta$. That is,
\[
\Delta = \sqcup_{a\in A(\Delta)} a + dn\mathbb{N}.
\]
\end{definition}

\subsection{(Local) compactified Jacobians}

Let $0\neq J\subset E$ be a finitely generated $R$-submodule. We record two basic properties:

\begin{enumerate}[wide,labelwidth=!,labelindent=0pt,itemindent=!,label=(\roman*)]
\item
$J$ is a \emph{torsion-free coherent $R$-module of (generic) rank $1$}.

Indeed, set $S:=R\setminus\{0\}$ so that $E=S^{-1}R$. Since $0\neq J\subset E$, we have
\[
0\neq J \subset S^{-1}J \subset E.
\]
But $S^{-1}J$ is a nonzero $E$-vector space, hence equal to $E$. 
Done.

\item
\emph{$J$ is a free $\cO_F$-module of rank $dn$}. 

Write $J=t^{-N}J_0$ for some $N\geq 0$ and $J_0\subset\tilde{R}$ a nonzero $R$-submodule.  
Since $\cO_F$ is a PID, $J_0$ is a free $\cO_F$-module of finite rank. Moreover,
\[
\delta(J_0) = \delta(\ord_t(J_0)) = \dim_{\field}\tilde{R}/J_0 = |\mathbb{N}-\ord_t(J_0)|\leq c(J_0) < \infty,
\]
so $\tilde{R}/J_0$ is a finitely generated torsion $\cO_F$-module. 
Tensoring with $F$ gives $F\otimes_{\cO_F}\tilde{R}/J_0=0$. 
Applying the exact functor $F\otimes_{\cO_F}-$ to 
\[
0\to J_0\to \tilde{R}\to \tilde{R}/J_0\to 0
\]
yields $F\otimes_{\cO_F}J_0 \isomorphic F\otimes_{\cO_F}\tilde{R}=\oplus_{j=0}^{dn-1}F\cdot t^j$.  
Hence $J_0$, and therefore $J$, is a free $\cO_F$-module of rank $dn$.
\end{enumerate}

\begin{definition}
The \textbf{(local) compactified Jacobian} of $(C,0)$ is
\[
\overline{J}(C,0):=\{ J\subset E \mid J \text{ an $R$-submodule with } J\subset t^{-N}\tilde{R},~ \dim t^{-N}\tilde{R}/J = \dim t^{-N}\tilde{R}/R,~N\gg 0\}.
\]
\end{definition}

\noindent
The terminology is explained as follows:
\begin{itemize}[wide,labelwidth=!,labelindent=0pt,itemindent=!]
\item
For $J\in\overline{J}(C,0)$ and $J\subset t^{-N}\tilde{R}$, the Noetherian property ensures that $J$ is finitely generated over $R$.

\item
By (i) above, $J$ can be viewed as a torsion-free rank $1$ coherent sheaf on $\Spec R$, of \emph{degree $0$}, where
\[
\deg(J):=\dim t^{-N}\tilde{R}/J - \dim t^{-N}\tilde{R}/R=0,
\]
independently of $N\gg 0$ (see Lemma \ref{lem:degree_of_a_lattice}).
\end{itemize}

\begin{lemma}\label{lem:compactified_Jacobian}
For $J\in\overline{J}(C,0)$,
\[
t^{2\delta}\tilde{R} \subset J \subset \tilde{R},\qquad \dim \tilde{R}/J=\delta.
\]
Equivalently,
\[
\overline{J}(C,0) = \{ J\subset\tilde{R} \mid J \text{ an $R$-submodule with } \dim\tilde{R}/J=\delta\}.
\]
\end{lemma}

\begin{proof}
Let $v_m:=\min\ord_t(J)$. Choose $z\in J$ with $\ord_t(z)=v_m$ and set $J_0:=z^{-1}J$. 
Then $\ord_t(J_0)=\ord_t(J)-v_m$ is a $0$-normalized $\Gamma$-semimodule, so $J_0\subset\tilde{R}$.  
By Lemma \ref{lem:estimates_for_delta-invariant_and_conductor}, $\delta(J_0)\leq \delta$.

Suppose $J\subset t^{-N}\tilde{R}$ with $N\gg 0$. Then
\[
\dim(t^{-N}\tilde{R}/J) = |(-N + \mathbb{N}) \setminus (v_m+\ord_t(J_0))| = |(-N-v_m+\mathbb{N})\setminus \ord_t(J_0)| = N+v_m + \delta(J_0) \leq N + v_m + \delta.
\]
But by assumption this dimension equals $\dim(t^{-N}\tilde{R}/R)=N+\delta$, hence $v_m\geq 0$. Thus $J\subset\tilde{R}$.  

Now, $\ord_t(J)\subset\mathbb{N}$ is a nonzero $\Gamma$-semimodule with $\delta(\ord_t(J))=\delta(J)=\delta$. By Lemma \ref{lem:estimates_for_delta-invariant_and_conductor}, $c(J)\leq 2\delta$, i.e. $t^{2\delta}\tilde{R}\subset J$.
\end{proof}

Recall that $E = K(R) = \oplus_{j=0}^{dn-1} F y^j$. 
For $G = GL_{dn}$, the map 
\[
[g] = gG(\cO) \mapsto (1,y,\cdots,y^{dn-1})g\cO^{dn}
\] 
induces a natural identification
\[
\GR_G \isomorphic \{\Lambda: \Lambda \text{ is an } \cO\text{-lattice in $E$, i.e., finitely generated (projective) $\cO$-submodule s.t. $F\otimes_{\cO}\Lambda = E$} \}.
\]
Moreover, the determinant $\det:G=GL_{dn} \rightarrow GL_1 = \mathbb{G}_m$ induces
\begin{equation}
\det:\GR_G \rightarrow \GR_{\mathbb{G}_m} = F^{\times}/\cO^{\times} \isomorphic \mathbb{Z}~~(\text{non-reduced}): [g] \mapsto [\det g] \isomorphic \ord_x\det(g).
\end{equation}
For clarity, \textbf{set} 
\begin{equation}
\deg([g]):= \ord_x\det(g) = \frac{1}{dn}\ord_t\det(g) \in \mathbb{Z}.
\end{equation}

\begin{lemma}\label{lem:degree_of_a_lattice}
Let $[g]\in\GR_G$ with associated lattice $\Lambda=(1,y,\dots,y^{dn-1})g\cO^{dn}$. Then
\[
\deg(\Lambda):=\deg([g])=\dim t^{-N}\tilde{R}/\Lambda - \dim t^{-N}\tilde{R}/R,\qquad N\gg 0.
\]
Moreover, $\deg(t\Lambda)=\deg(\Lambda)+1$.
\end{lemma}

\begin{proof}
Choose $N\gg 0$ such that $\Lambda,R\subset t^{-N}\tilde{R}$. 
Write $t^{-N}\tilde{R} = \oplus_{i=0}^{dn-1}\cO\cdot z_i$ and $(1,y,\cdots,y^{dn-1}) = (z_0,z_1,\cdots,z_{dn-1})h$, $h\in \fg(\cO)$.
Then $(1,y,\cdots,y^{dn-1})g = (z_0,z_1,\cdots,z_{dn-1})hg$ is an $\cO$-basis of $\Lambda$.
As $\cO$ is a PID, by the structure theorem for finitely generated $\cO$-modules, we have
\[
\dim t^{-N}\tilde{R}/R = \ord_x\det(h);\quad \dim t^{-N}\tilde{R}/\Lambda = \ord_x\det(hg) = \ord_x\det(h) + \ord_x\det(g).
\]
The first claim follows.  

For the second, note $t\Lambda \subset t^{1-N}\tilde{R} \subset t^{-N}\tilde{R}$, and $\ord_t(t\Lambda) = 1 + \ord_t(\Lambda)$, so
\[
\dim t^{-N}\tilde{R}/t\Lambda = |(-N+\mathbb{N})\setminus \ord_t(t\Lambda)| = |(-N-1 + \mathbb{N})\setminus \ord_t\Lambda| = 1 + |(-N + \mathbb{N})\setminus \ord_t\Lambda| = 1 + \dim t^{-N}\tilde{R}/\Lambda.
\]
Combining with the first claim yields $\deg(t\Lambda)=\deg(\Lambda)+1$.
\end{proof}

\begin{lemma}
Let $\Spr_{\gamma}$ be the affine Springer fiber associated to $(\gamma,G=GL_{dn},\fg)$.
Then,
\begin{equation}
\mathbb{Z}\times \overline{J}(C,0) \isomorphic \Spr_{\gamma}: (v,J)\mapsto t^vJ,
\end{equation}
as (reduced) ind-varieties.
\end{lemma}
\noindent\textbf{Note.} Since $\gamma$ is $G(F)$-conjugate to $\gamma^\transpose$, also $\Spr_\gamma\isomorphic \Spr_{\gamma^\transpose}$.

\begin{proof}
Let $\Lambda\subset E$ be an $\cO$-lattice. Then $\Lambda$ is an $R$-module iff $y\Lambda\subset\Lambda$, i.e.
\[
g^{-1}\cdot\gamma\in\fg(\cO)\quad\text{for } \Lambda=(1,y,\dots,y^{dn-1})g\cO^{dn},
\]
which is precisely the condition $[g]\in\Spr_\gamma$.  
The identification with $\mathbb{Z}\times\overline{J}(C,0)$ now follows from Lemma \ref{lem:degree_of_a_lattice}. In particular,
\[
\overline{J}(C,0)=\{\Lambda\in\Spr_\gamma:\deg(\Lambda)=0\}.
\]
\end{proof}

\begin{remark}\label{rem:equivalued_planar_curve_singularity_is_quasi-homogeneous}
Goresky–Kottwitz–MacPherson \cite{GKM06} proved purity of affine Springer fibers $\Spr_{\basevector}$ for regular integral equivalued semisimple $\basevector\in\fg(F)$.  
For $G=GL_{dn}$, this means $\basevector$ is $G(\overline{F})$-conjugate to 
\[
\diag(a_1x^s,\dots,a_{dn}x^s) + \text{h.o.t.},
\]
with $s\in\mathbb{Q}_{\geq 0}$, $a_i\in\field$ distinct.  
Consequently, our $\gamma\in G(F)$ is regular integral equivalued semisimple iff $d = 1$, i.e. $(C,0)$ is quasi-homogeneous: $x=t^n$, $y=t^m$, $\gcd(n,m)=1$.
\end{remark}

\subsection{The (local) Abel--Jacobi map}

For $\tau \geq 0$, recall that the Hilbert scheme of $\tau$ points on $(C,0)$ is
\[
\Hilb^{[\tau]}(C,0) := \{I\subset R~|~ I \text{ is an ideal with } \dim R/I = \tau\}.
\]
In particular, $\Hilb^{[0]}(C,0) = \{\pt\}$.

\begin{lemma}[{\cite[Thm.~3]{PS92}}]\label{lem:the_closed_embedding_from_Hilbert_scheme_to_local_compactified_Jacobian}
There is a closed embedding
\[
\phi_{\tau}:\Hilb^{[\tau]}(C,0) \hookrightarrow \overline{J}(C,0), 
\quad I \mapsto t^{-\tau}I.
\]
Moreover, for $J \in \overline{J}(C,0)$, one has $I := t^{\tau}J \in \Hilb^{[\tau]}(C,0)$ 
iff $t^{\tau}J \subset R$. In particular, $\phi_{\tau}$ is an isomorphism for $\tau \geq 2\delta$.
\end{lemma}

\begin{definition}
For a nonzero finitely generated $R$-submodule $J\subset E$, define the \textbf{dual $R$-module}
\begin{equation}
J^{-1}:=\{z\in E:~zJ\subset R\} \subset E.
\end{equation}
\end{definition}


It is immediate that $J^{-1}$ is again a finitely generated $R$-module. 
Some basic properties are:
\begin{enumerate}[wide,labelwidth=!,labelindent=0pt,itemindent=!,label=(\alph*)]
\item
As a $R$-module, $J^{-1} \xrightarrow[]{\isomorphic} \Hom_R(J,R): z\mapsto z$.
The inverse $\Hom_R(J,R)\xrightarrow[]{\isomorphic} J^{-1}:h\mapsto r_h$ is defined as follows: $\forall h\in\Hom_R(J,R)$, we get a commutative diagram of $R$-modules
\[
\begin{tikzcd}[row sep = 1pc, column sep = 2pc]
J\arrow[r,"h"]\arrow[d,hookrightarrow] & R\arrow[d,hookrightarrow]\\
E=S^{-1}J\arrow[r,"{S^{-1}h}"] & S^{-1}R = E
\end{tikzcd}.
\]
Take any nonzero $z\in J\subset E$, then by $E$-linearity, $(S^{-1}h)(az) = a h(z) = az\cdot (z^{-1}h(z))$, $\forall a\in E$.
Moreover, for any other nonzero $w\in J\subset E$, we have $w = az$ for some $a\in E^{\times}$. It follows that $w^{-1}h(w) = z^{-1}a^{-1}ah(z) = z^{-1}h(z)$.
So, $r_h:=z^{-1}h(z)$ is independent of $z\in J\setminus \{0\}$, and $r_hw = h(w) \in R$, $\forall w\in J\setminus\{0\}$. In other words, $r_h\in J^{-1}$.

\item
$R^{-1} = R$. 


\item
Since $R$ is Gorenstein, $(J^{-1})^{-1} = J$ \cite[Lem.~1.1]{Har86}. 

\item
If $J_1 \subset J_2$, then $J_2^{-1} \subset J_1^{-1}$.  
\end{enumerate}

\begin{lemma}
For any $J \in \overline{J}(C,0)$:
\begin{enumerate}[wide,labelwidth=!,labelindent=0pt,itemindent=!]
\item $t^{2\delta}\tilde{R} \subset J^{-1} \subset \tilde{R}$ and $\dim \tilde{R}/J^{-1} = \delta$.  
In particular, we get an involution
\[
\iota: \overline{J}(C,0) \xrightarrow{\;\cong\;} \overline{J}(C,0): J \mapsto J^{-1}.
\]

\item For $a \in [0,2\delta-1]$,
\[
a \in \ord_t(J) 
\;\;\Longleftrightarrow\;\; 
2\delta-1-a \notin \ord_t(J^{-1}).
\]
In particular,
\[
\ord_t(J^{-1}) = [0,2\delta-1] \setminus (2\delta-1-\ord_t(J)) \;\sqcup\; [2\delta,\infty).
\]
\end{enumerate}
\end{lemma}

\begin{proof}
\noindent{}$(1)$. We first claim that $\tilde{R}^{-1} = t^{2\delta}\tilde{R}$. Indeed $z \in \tilde{R}^{-1}$ iff $z\tilde{R} \subset R$, forcing $\ord_t(z) \geq 2\delta$. Conversely, $t^{2\delta}\tilde{R}\subset R$. 

Now by Lemma \ref{lem:compactified_Jacobian}, we have $t^{2\delta}\tilde{R}\subset J \subset \tilde{R}$.  
Taking duals gives
\[
\tilde{R}^{-1} = t^{2\delta}\tilde{R} \subset J^{-1} \subset (t^{2\delta}\tilde{R})^{-1} = \tilde{R}.
\]

To compute $\dim \tilde{R}/J^{-1}$, apply $\RHom_R(-,R)$ to the short exact sequence
\[
0 \to J \to \tilde{R} \to \tilde{R}/J \to 0.
\]
Since $R$ is Gorenstein of dimension $1$, we have \cite[Prop.~4.9, Prop.~4.13]{Har67}:
\begin{enumerate}[wide,labelwidth=!,labelindent=0pt,itemindent=!,label=(\Alph*)]
\item
$\Ext^1_R(-,R)$ is a duality on finite-length $R$-modules, preserving length.

\item
$\Ext^i_R(M,R)=0$ for $i\neq 1$ if $M$ has finite length.

\item
$\Ext^i_R(M,R)=0$ for all $i>0$ if $M$ is finitely generated torsion-free \cite[Lem.~1.1]{Har86}.

\item
If $R$ is a local $\field$-algebra with residue field $\field$, then $\ell_R(M) = \dim_{\field} M$.
\end{enumerate}
Using these, we obtain an exact sequence
\[
\Hom_R(\tilde{R}/J,R) = 0 \rightarrow \Hom_R(\tilde{R},R) \rightarrow \Hom_R(J,R) \rightarrow \Ext_R^1(\tilde{R}/J,R) \rightarrow \Ext_R^1(\tilde{R},R) = 0,
\]
where $\Hom_R(\tilde{R},R) = \tilde{R}^{-1} = t^{2\delta}\tilde{R}$ and $\Hom_R(J,R) = J^{-1}$.  
Hence
\[
J^{-1}/t^{2\delta}\tilde{R} \;\cong\; \Ext^1_R(\tilde{R}/J,R).
\]
By (A), the latter has length $\dim_\field \tilde{R}/J = \delta$.  
Thus $\dim \tilde{R}/J^{-1} = 2\delta - \dim J^{-1}/t^{2\delta}\tilde{R} = \delta$.

\smallskip
\noindent (2)  
By (1), $t^{2\delta}\tilde{R}\subset J^{-1}\subset \tilde{R}$.  
For $a \in [0,2\delta-1]$, observe that if $z\in J$ and $w\in J^{-1}$ then $zw \in R$, so 
\[
\ord_t(J)+\ord_t(J^{-1}) \subset \Gamma.
\]
Since $2\delta-1 \notin \Gamma$, we deduce
\[
a \in \ord_t(J) \;\Rightarrow\; 2\delta-1-a \notin \ord_t(J^{-1}).
\]
But both sides contain exactly $\delta$ elements in $[0,2\delta-1]$, so the implication is an equivalence.  
The alternative description of $\ord_t(J^{-1})$ follows.  
\end{proof}

\begin{definition}
The composition
\begin{equation}
\AJ_{\tau}:= \iota \circ \phi_{\tau}: 
\Hilb^{[\tau]}(C,0) \hookrightarrow \overline{J}(C,0), 
\quad I \mapsto t^{\tau}I^{-1},
\end{equation}
is called the ($\tau$-th) \textbf{local Abel--Jacobi map}.
\end{definition}

\begin{corollary}\label{cor:local_Abel-Jacobian_map}
For $J\in \overline{J}(C,0)$,
\begin{equation}\label{eqn:image_of_local_Abel-Jacobian_map}
J\in \im(\AJ_{\tau}) 
\quad \Longleftrightarrow \quad 
t^{\tau}\in J.
\end{equation}
In particular, $t^{\tau}\in J$ automatically if $\tau \geq c(J)$.
\end{corollary}

\begin{proof}
By Lemma~\ref{lem:the_closed_embedding_from_Hilbert_scheme_to_local_compactified_Jacobian} and property (d) above,
\[
J\in \im(\AJ_{\tau}) 
\;\;\Longleftrightarrow\;\; 
J^{-1}\in \im(\phi_{\tau}) 
\;\;\Longleftrightarrow\;\;
t^{\tau}J^{-1}\subset R
\;\;\Longleftrightarrow\;\;
t^{\tau}\in J.
\]
The last claim is immediate since $c(J)$ is the conductor of $\ord_t(J)$.  
\end{proof}

\section{Affine paving and purity}

Let $(C,0)$ be a \textbf{generic} irreducible planar curve singularity as in Definition \ref{def:generic_planar_curve_singularities}.
The smooth case is trivial, so we \textbf{assume} $(C,0)$ is singular.  
By Remark \ref{rem:finite_determinancy_of_planar_curve_singularities}, we may take $(C,0)$ to be parametrized by
\begin{equation}
x(t) = t^{dn}, \quad 
y(t) = t^{dm} + \lambda t^{dm+1} + \sum_{\substack{j\geq 2 \\ dm+j \notin \Gamma}} \lambda_j t^{dm+j},
\quad (n,m)=1, \; m>n>1, \; \lambda \in \field^\times, \; \lambda_j \in \field.
\end{equation}
Thus $(C,0)$ has two \textbf{Puiseux pairs}, $(n,m)$ and $(d,dm+1)$, and its \textbf{singularity knot} $K_{(C,0)}$ is the $(d,dmn+1)$-cable of the $(n,m)$-torus knot.  
The semigroup is generated by $dn, dm, dmn+1$, i.e.
\begin{equation}
\Gamma = \bigsqcup_{0\leq i\leq d-1,\;0\leq j\leq n-1} \bigl(\hat{a}_{i,j} + dn\mathbb{N}\bigr), 
\quad \hat{a}_{i,j} := jdm + (dmn+1)i.
\end{equation}

Recall that a $\field$-variety $X$ admits an \textbf{affine paving} if there exists an increasing finite closed filtration
\[
\emptyset = X_0 \subset X_1 \subset \cdots \subset X_r = X
\]
such that each difference $X_{i+1}\setminus X_i$ is isomorphic to an affine space $\mathbb{A}^k$.

\subsection{The affine paving for the local compactified Jacobian}

We now recall Gorsky–Mazin–Oblomkov's affine paving for the local compactified Jacobian $\overline{J}(C,0)$ associated to $(C,0)$.

\subsubsection{Cell decomposition via semimodules of valuations}

\begin{definition}
A \textbf{standard module} for $(C,0)$ is an $R$-submodule $M \subset \tilde{R}$ with $\min \ord_t(M) = 0$. 
In particular, $\ord_t(M)\subset \mathbb{N}$ is a $0$-normalized $\Gamma$-semimodule.
\end{definition}

\begin{definition}
For a $0$-normalized $\Gamma$-semimodule $\Delta \subset \mathbb{N}$, define
\[
J_{\Delta} := \{\,M \subset \tilde{R} \;\mid\; M \text{ is a standard $R$-submodule with } \ord_t(M) = \Delta \,\}.\quad (\text{possibly empty})
\]
\end{definition}

\begin{lemma}
There is a canonical locally closed embedding
\[
\iota_{\Delta}: J_{\Delta} \hookrightarrow \overline{J}(C,0), \quad M \mapsto t^{e(\Delta)}M.
\]
Moreover,
\[
\overline{J}(C,0) = \bigsqcup_{\Delta} \iota_{\Delta}(J_{\Delta}).
\]
\end{lemma}

\begin{proof}
Immediate from Lemma \ref{lem:compactified_Jacobian} and Definition \ref{def:semimodule_for_compactified_Jacobian}.
\end{proof}

\subsubsection{Alternative description via Schubert cells}

By Lemma \ref{lem:compactified_Jacobian}, there is a natural identification (as reduced schemes):
\[
\overline{J}(C,0) \cong 
\{\Lambda \in \GR(\delta,\tilde{R}/t^{2\delta}\tilde{R}) : x\Lambda \subset \Lambda,\, y\Lambda \subset \Lambda\}
\subset \GR(\delta,\tilde{R}/t^{2\delta}\tilde{R}),
\]
where $\GR(\delta,\tilde{R}/t^{2\delta}\tilde{R})$ is the Grassmannian of $\delta$-dimensional subspaces of 
$\tilde{R}/t^{2\delta}\tilde{R} \cong \bigoplus_{j=0}^{2\delta-1}\field\cdot t^j$.

Define the canonical flag
\[
0 = V_0 \subset V_1 \subset \cdots \subset V_{2\delta} = \tilde{R}/t^{2\delta}\tilde{R},\quad V_i:= t^{2\delta-i}\tilde{R}/t^{2\delta}\tilde{R} \isomorphic \oplus_{j=2\delta-i}^{2\delta-1}\field\cdot t^j.
\]

For integers $\delta \geq a_1\geq \cdots \geq a_{\delta} \geq 0$, the \textbf{Schubert cell}
\begin{equation}
W_{a_1,\cdots,a_{\delta}}:= \{\Lambda\in \GR(\delta,\tilde{R}/t^{2\delta}\tilde{R}):~\dim(\Lambda\cap V_j) = i \text{ for }\delta+i-a_i \leq j < \delta+i+1 - a_{i+1}\}m
\end{equation}
gives the Schubert cell decomposition \cite[p.106]{PS92}
\[
\GR(\delta,\tilde{R}/t^{2\delta}\tilde{R}) = \sqcup_{\delta \geq a_1\geq \cdots \geq a_{\delta} \geq 0} W_{a_1,\cdots,a_{\delta}},\quad W_{a_1,\cdots,a_{\delta}} \isomorphic \mathbb{A}^{\delta^2 - \sum_{i=1}^{\delta} a_i}.
\]
Moreover, $W_{b_1,\dots,b_\delta} \subset \overline{W}_{a_1,\dots,a_\delta}$ iff $b_i \geq a_i$ for all $i$, which defines an affine paving of the Grassmannian.

In valuation terms:
\[
W_{a_1,\dots,a_\delta} = 
\{\Lambda \in \GR(\delta,\tilde{R}/t^{2\delta}\tilde{R}) : \ord_t(\Lambda)\cap [0,2\delta-1] = \{\delta+a_i-i : 1\leq i \leq \delta\}\}.
\]
Hence, for any $0$-normalized $\Gamma$-semimodule $\Delta \subset \mathbb{N}$,
\[
\iota_{\Delta}(J_{\Delta}) \cong \overline{J}(C,0)\cap W_{a_1,\dots,a_\delta},
\]
with $a_i$ determined by
\[
\{\delta + a_1 - 1 > \cdots > \delta + a_\delta - \delta\} 
= \overline{\Delta}\cap [0,2\delta-1], \quad \overline{\Delta} := e(\Delta) + \Delta.
\]

\subsubsection{Admissibility and affine paving}

It remains to determine $J_{\Delta}$, in particular, when $J_{\Delta}$ is nonempty.

\begin{definition}\label{def:admissibility}
Fix any $(dn,dm)$-invariant subset $\Delta\subset \mathbb{N}$. 
\begin{enumerate}[wide,labelwidth=!,labelindent=0pt,itemindent=!]
\item 
For $\ell \geq 1$ and $0\leq i \leq d-1$, say $i$ is \textbf{$\ell$-admissible}
\footnote{In \cite[Def.3.7]{GMO25}, this corresponds to $\ell$ not being $i$-suspicious.} 
if there exists $a\in A(\Delta)$ with $a \equiv i \pmod d$ and $a+dm+\ell \in \Delta$.  
We write ``admissible'' for $\ell=1$.

\item
$\Delta$ is \textbf{admissible} if every $i \in \{0,\dots,d-1\}$ is admissible.
\end{enumerate}
Finally, \textbf{denote} by $\Adm(dn,dm)$ the set of all admissible $0$-normalized $\Gamma$-semimodules.
\end{definition}

For a more concrete description of $\Adm(dn,dm)$, see Section \ref{subsec:cell_indices_and_dimensions_via_Dyck_paths}.

\begin{definition}
For $\Delta \in \Adm(dn,dm)$ and $a \in \Delta$, set
\[
\Gaps(a) = \Gaps_{\Delta}(a) := [a,\infty)\setminus \Delta \subset \mathbb{N}.
\]
Define
\begin{equation*}
\dim\Delta := \sum_{a\in A(\Delta)} |\Gaps(a)| - \sum_{a\in A(\Delta)} |\Gaps(a+dm)|.
\end{equation*}
\end{definition}

\begin{lemma}[{\cite[Prop.3.23]{GMO25}}]\label{lem:affine_paving_for_local_compactified_Jacobian_over_generic_curve_singularity}
For any $0$-normalized $\Gamma$-semimodule $\Delta \subset \mathbb{N}$,
\[
J_\Delta \neq \emptyset \quad \Longleftrightarrow \quad \Delta \in \Adm(dn,dm).
\]
In this case,
\[
J_\Delta \cong \mathbb{A}^{\dim \Delta}.
\]
\end{lemma}

\begin{corollary}
For a generic irreducible planar curve singularity $(C,0)$, the compactified Jacobian $\overline{J}(C,0)$ admits an affine paving
\[
\overline{J}(C,0) = \bigsqcup_{\Delta \in \Adm(dn,dm)} \iota_\Delta(J_\Delta),
\]
with the filtration induced from the Schubert cell paving of $\GR(\delta,\tilde{R}/t^{2\delta}\tilde{R})$.
\end{corollary}

\subsection{The affine paving for the Hilbert scheme of points}

We now arrive at our main goal: the affine paving of the Hilbert scheme 
$\Hilb^{[\tau]}(C,0)$, when $(C,0)$ is a \textbf{generic} irreducible planar curve singularity.

\begin{definition}
For any $0$-normalized $\Gamma$-semimodule $\Delta\subset \mathbb{N}$, define
\begin{eqnarray}
H_{\Delta}^{[\tau]} 
&:=& \{\, I\in \Hilb^{[\tau]}(C,0) \;\big|\; 
\AJ_{\tau}(I)\in \iota_{\Delta}(J_{\Delta}), \text{ i.e. } 
\ord_t(t^{\tau-e(\Delta)}I^{-1}) = \Delta \,\}\nonumber\\
&\isomorphic& \{\, M\in J_{\Delta} \;\big|\; t^{\tau-e(\Delta)}\in M \,\}.
\end{eqnarray}
The second description follows from Corollary~\ref{cor:local_Abel-Jacobian_map} with 
$J=\AJ_{\tau}(I) = t^{e(\Delta)}M$. 
Equivalently, via the Abel--Jacobian map $\AJ_{\tau}:\Hilb^{[\tau]}(C,0) \hookrightarrow \overline{J}(C,0)$, 
\[
H_{\Delta}^{[\tau]} = \Hilb^{[\tau]}(C,0)\cap \iota_{\Delta}(J_{\Delta}).
\]
\end{definition}

This gives a decomposition into locally closed subvarieties:
\begin{equation}
\Hilb^{[\tau]}(C,0) = \bigsqcup_{\Delta} H_{\Delta}^{[\tau]}.
\end{equation}

\begin{remark}
For any $M\in J_{\Delta}$, the condition $t^{\tau - e(\Delta)} \in M$ is automatic whenever 
$\tau - e(\Delta) \geq c(M) = c(\Delta)$. Thus,
\begin{equation}\label{eqn:image_of_local_Abel-Jacobian_map-2}
H_{\Delta}^{[\tau]} \isomorphic J_{\Delta}, 
\quad \forall \tau \geq c(\Delta) + e(\Delta) = c(\overline{\Delta}).
\end{equation}
\end{remark}

\begin{theorem}\label{thm:affine_paving_for_Hilbert_scheme_over_generic_curve_singularity}
We have $H_{\Delta}^{[\tau]}\neq \emptyset$ if and only if 
$\Delta\in \Adm(dn,dm)$ and $\tau_0 := \tau - e(\Delta)\in \Delta$. 
In this case,
\begin{equation}
H_{\Delta}^{[\tau]} \isomorphic \mathbb{A}^{\dim\Delta - |\Gaps(\tau_0)|}.
\end{equation}
\end{theorem}

\begin{proof}
The ``only if'' part is clear. 
It remains to show that for any $\Delta\in \Adm(dn,dm)$ with $\tau_0=\tau-e(\Delta)\in \Delta$, 
\[
H_{\Delta}^{[\tau]} \isomorphic \mathbb{A}^{\dim\Delta - |\Gaps(\tau_0)|}.
\]
See Example~\ref{ex:affine_paving_for_Hilbert_scheme_of_a_generic_planar_curve_singularity} for an illustration.

\medskip\noindent
\textbf{Step 1. Preparations.} 
Fix $\Delta\in \Adm(dn,dm)$. 
Let $A(\Delta) = \{\hat{b}_{i,j}\mid 0\leq i\leq d-1,\,0\leq j\leq n-1\}$ be the set of $dn$-generators of $\Delta$ 
with $\hat{b}_{i,j}\equiv jdm+i \pmod{dn}$. 
For any standard $R$-module $M\subset \tilde{R}$ with $\ord_t(M)=\Delta$, we may write uniquely
\begin{equation}
M = \bigoplus_{a\in A(\Delta)}\cO_F\cdot g_a 
= \bigoplus_{0\leq i\leq d-1,\,0\leq j\leq n-1}\cO_F\cdot \beta_{i,j}, 
\quad \beta_{i,j}=g_{\hat{b}_{i,j}},
\end{equation}
where for each $a\in\Delta$ there exists a unique element
\begin{equation}
g_a = t^a + \sum_{\ell\in [a,\infty)\setminus \Delta} g_{a;\ell-a}t^{\ell}, 
\qquad g_{a;\ell-a}\in \field.
\end{equation}
For convenience, set
\[
g_{a;0}:=1, 
\qquad g_{a;\ell-a}:=0 \quad \text{for all }\ell\in (a,\infty)\cap\Delta,\; a\in\Delta.
\]

\medskip
\noindent
Conversely, given elements $g_a$ for $a\in A(\Delta)$ as above, define
\[
M := \bigoplus_{a\in A(\Delta)} \cO_F\cdot g_a.
\]
Then $M$ is a standard $R$-module if and only if $yM\subset M$.  
Concretely, for each $a\in A(\Delta)$, viewing the coefficients $g_{a;j-a}$ with $j\in \Gaps(a)$ as \textbf{variables}, there exists a unique element
\begin{equation}
s_{a+dm} = \sum_{\ell\in \Gaps(a+dm)} s_{a+dm;\,\ell-a-dm}\, t^{\ell}
\end{equation}
such that:
\begin{itemize}
\item 
each $s_{a+dm;\,\ell-a-dm}$ is a polynomial in the variables $g_{a';\ell'}$, $a'\in A(\Delta), \ell'>0$;

\item 
$yg_a - s_{a+dm}\in \oplus_{a\in A(\Delta)} \cO_F g_a$.
\end{itemize}
Then,
\begin{equation}
M\in J_{\Delta}
\;\;\Longleftrightarrow\;\;
yM\subset M
\;\;\Longleftrightarrow\;\;
s_{a+dm;\,\ell-a-dm}=0,\quad 
\forall\,\ell\in \Gaps(a+dm),\; a\in A(\Delta).
\end{equation}
In this case, by Corollary \ref{cor:local_Abel-Jacobian_map}, 
\[
\iota_{\Delta}(M)\in H_{\Delta}^{[\tau]}
\quad\Longleftrightarrow\quad
g_{\tau_0}=t^{\tau_0}.
\]

\medskip
\noindent
To prepare for explicit calculations, we now determine the formulas for $s_{a+dm}$ ($a\in A(\Delta)$) and for $g_b$ ($b\in\Delta$).  
Since
\[
\hat{b}_{i,j}+dm \in \Delta, 
\qquad \hat{b}_{i,j}+dm \equiv \hat{b}_{i,j+1}\pmod{dn}, 
\]
for all $j\in\mathbb{Z}/n$, we may write
\[
\hat{b}_{i,j}+dm = \hat{b}_{i,j+1} + dn\alpha_{i,j},\qquad \alpha_{i,j}\geq 0.
\]
For $\ell\geq 1$, set
\begin{equation}
g_{i,j;\ell}^- := g_{\hat{b}_{i,j};\ell} - g_{\hat{b}_{i,j+1};\ell}, 
\qquad
R_{i,j;\ell} := \lambda_2 g_{\hat{b}_{i,j};\ell-2} + \cdots + \lambda_{\ell}g_{\hat{b}_{i,j};0},
\end{equation}
with the convention $R_{i,j;1}=0$.
Then, 
\begin{equation}
yg_{\hat{b}_{i,j}} - x^{\alpha_{i,j}}g_{\hat{b}_{i,j+1}} = \sum_{\ell\geq 1} F_{i,j;\ell} t^{\hat{b}_{i,j}+dm+\ell},\quad 
F_{i,j;\ell} := g_{i,j;\ell}^- + \lambda g_{\hat{b}_{i,j};\ell-1} + R_{i,j;\ell}.
\end{equation}

It follows that
\begin{eqnarray*}
s_{\hat{b}_{i,j}+dm} &=&  \sum_{\ell\geq 1} F_{i,j;\ell} t^{\hat{b}_{i,j}+dm+\ell} - \sum_{\ell_0\geq 1:~\hat{b}_{i,j}+dm+\ell_0\in\Delta}F_{i,j;\ell_0}g_{\hat{b}_{i,j}+dm+\ell_0} \\
&=& \sum_{\ell\geq 1:~\hat{b}_{i,j}+dm+\ell\notin\Delta} (F_{i,j;\ell} - \sum_{1\leq \ell_0 < \ell:~\hat{b}_{i,j}+dm+\ell_0\in\Delta}F_{i,j;\ell_0}g_{\hat{b}_{i,j}+dm+\ell_0;\ell-\ell_0}) t^{\hat{b}_{i,j}+dm+\ell}.
\end{eqnarray*}
Hence,
\begin{equation}\label{eqn:s-equation}
s_{\hat{b}_{i,j}+dm;\ell} = F_{i,j;\ell} - \sum_{1\leq \ell_0 < \ell:~\hat{b}_{i,j}+dm+\ell_0\in\Delta}F_{i,j;\ell_0}g_{\hat{b}_{i,j}+dm+\ell_0;\ell-\ell_0}.
\end{equation}

\medskip
\noindent
Finally, the coefficients $g_{b;\ell}$ ($b\in \Delta$, $b+\ell\in \Gaps(b)$) are polynomials in the basic variables 
$g_{a;j}$ ($a\in A(\Delta)$, $a+j\in\Gaps(a)$), determined inductively:
\begin{itemize}
\item 
If $b\geq c(\Delta)$, then $g_b=t^b$ and hence $g_{b;\ell}=0$.

\item 
Otherwise, suppose $g_{b'}$ is known for all $b'>b$ in $\Delta$.  
Write $b=a+\alpha dn$ with $a\in A(\Delta)$, $\alpha\geq 0$.  
Since $b=\ord_t(x^\alpha g_a)$, we have
\begin{eqnarray}\label{eqn:inductive_g-formula}
&&x^{\alpha}g_a = t^b + \sum_{j\geq 1:~a+j\notin \Delta} g_{a;j}t^{b+j}. ~~\Rightarrow~~ g_{a+\alpha dn} = x^{\alpha}g_a - \sum_{j\geq 1:~a+j\notin\Delta, a+\alpha dn+j\in\Delta} g_{a;j}g_{a+\alpha dn+j}. ~~\Rightarrow \nonumber\\
&&g_{a+\alpha dn;\ell} = g_{a;\ell} -  \sum_{1\leq j < \ell:~a+j\notin\Delta, a+\alpha dn+j\in\Delta} g_{a;j}g_{a+\alpha dn+j;\ell-j},~~(a+\alpha dn + \ell \notin \Delta).
\end{eqnarray}
\end{itemize}
By induction, this determines all $g_b$.

\medskip\noindent
\textbf{Step~2.}
We recall some details from the proof of Lemma~\ref{lem:affine_paving_for_local_compactified_Jacobian_over_generic_curve_singularity}, cf.~\cite[Prop.~3.23]{GMO25}.

\medskip\noindent
Define
\[
\deg_+ g_{a;\ell}:= \ell \quad (a\in A(\Delta)), 
\qquad \deg_+(\lambda_{\ell}):= \ell.
\]
It follows that $g_{b;\ell}$ ($b\in \Delta$) and $s_{a+dm;\ell}$ are homogeneous of degree $\deg_+ = \ell$.  
Set
\begin{equation}
\Gen_{\ell} := \oplus_{a\in A(\Delta):~a + \ell\in\Gaps(a)}\field\cdot g_{a;\ell}.
\end{equation}

\medskip\noindent
For $i\in \mathbb{Z}/d$, $\ell\geq 1$, \textbf{define}
\begin{equation}\label{eqn:index_set_hat-i-ell}
\hat{I}(i;\ell):= \{j\in\mathbb{Z}/n:~\hat{b}_{i,j} + \ell \notin \Delta\};\quad N_{i;\ell}:= |\hat{I}(i;\ell)|; \quad \Gen_{i;\ell}:= \oplus_{j\in\hat{I}(i;\ell)} \field\cdot g_{\hat{b}_{i,j};\ell}.
\end{equation}
Then $\Gen_{\ell} = \oplus_{i\in\mathbb{Z}/d}\Gen_{i;\ell}$.
If $\hat{I}(i;\ell)\neq \emptyset$, \textbf{set}
\begin{equation}\label{eqn:basis_vector_g_i,ell-plus}
g_{i;\ell}^+ := \sum_{j\in\hat{I}(i;\ell)} g_{\hat{b}_{i,j};\ell};\quad \Gen_{i;\ell}^-:= \{\sum_{j\in\hat{I}(i;\ell)} *_\ell g_{\hat{b}_{i,j};\ell}\in \Gen_{i;\ell}:~\sum *_{\ell} = 0\}.
\end{equation}
Thus, $\Gen_{i;\ell} = \Gen_{i;\ell}^-\oplus \field\cdot g_{i;\ell}^+$.

\medskip\noindent
\emph{New basis for $\Gen_{i;\ell}^-$.}
Note that if $\hat{b}_{i,j}+dm+\ell\notin\Delta$, then also $\hat{b}_{i,j}+\ell,\hat{b}_{i,j+1}+\ell\notin\Delta$, i.e.~$j,j+1\in\hat{I}(i;\ell)$.  
For $i\in \mathbb{Z}/d = \{0,1,\cdots,d-1\}$, $\ell > 0$:
\begin{enumerate}[wide,labelwidth=!,labelindent=0pt,itemindent=!,label=(\roman*)]
\item
If $i$ is $\ell$-admissible (Definition \ref{def:admissibility}), define 
\begin{equation}\label{eqn:index_set_I-i,ell_1}
I(i;\ell):= \{j\in\mathbb{Z}/n:~\hat{b}_{i,j} + dm + \ell \notin \Delta\}\subsetneq \mathbb{Z}/n.
\end{equation}
Then $\{g_{i,j;\ell}^- : j\in I(i;\ell)\}$ is linearly independent, and $I(i;\ell)\subsetneq \hat{I}(i;\ell)$.

\item
Otherwise, $\hat{b}_{i,j}+dm+\ell\notin\Delta$ for all $j\in\mathbb{Z}/n$. So, $\hat{I}(i;\ell)=\mathbb{Z}/n$.  
In this case, set
\begin{equation}\label{eqn:index_set_I-i,ell_2}
I(i;\ell):=\{0,\cdots,n-2\} \subsetneq \hat{I}(i;\ell).
\end{equation}
\end{enumerate}

In either case, we may complete $\{g_{i,j;\ell}^- : j\in I(i;\ell)\}$ to a basis of $\Gen_{i;\ell}^-$ by adding elements of the form $g_{\hat{b}_{i,j};\ell}-g_{\hat{b}_{i,k};\ell}$ with $j,k\in\hat{I}(i;\ell)$.  
For example, fix $j_0\in\hat{I}(i;\ell)\setminus I(i;\ell)$, and let $\overline{I}(i;\ell):=\hat{I}(i;\ell)\setminus(I(i;\ell)\cup\{j_0\})$.  
Then
\[
\Gen_{i;\ell}^- 
= \bigoplus_{j\in I(i;\ell)} \field\cdot g_{i,j;\ell}^- 
  \;\oplus\;
  \bigoplus_{j\in \overline{I}(i;\ell)} \field\cdot h_{i,j;\ell}^-, 
\quad 
h_{i,j;\ell}^- := g_{\hat{b}_{i,j};\ell}-g_{\hat{b}_{i,j_0};\ell}.
\]

\medskip\noindent
Next, set
\[
\Lambda:= \field[\lambda^{\pm 1}, \lambda_j, j\geq 2, dm+j\notin\Gamma].
\]
Define a \textbf{partial order} on 
\begin{equation}\label{eqn:new_variables_for_equations-affine_paving_for_Hilbert_scheme}
X:= \{g_{i;\ell}^+; g_{i,j;\ell}^-, j\in I(i;\ell); h_{i,j;\ell}^-, j\in \overline{I}(i;\ell)\}
\end{equation} 
by
\begin{equation}
g_{\bullet,\bullet;\ell}^-, h_{\bullet,\bullet;\ell}^- < g_{\bullet;\ell}^+ < g_{\bullet,\bullet;\ell+1}^-, h_{\bullet,\bullet;\ell+1}^-;\quad g_{d-\ell;\ell}^+ > g_{d-\ell+1;\ell}^+ > \cdots > g_{d-\ell-1;\ell}^+.
\end{equation}

\medskip\noindent
Recall that
\[
J_{\Delta}=\{s_{a+dm;\ell}=0 : a\in A(\Delta),\, a+dm+\ell\notin\Delta\}.
\]
If $i$ is not $\ell$-admissible (so necessarily $\ell\geq 2$), then we replace the equation $s_{\hat{b}_{i,n-1}+dm;\ell}=0$ by
\[
\sum_{j\in\mathbb{Z}/n} s_{\hat{b}_{i,j}+dm;\ell}=0.
\]

\medskip\noindent
For brevity, for $g\in X$ write $\Poly(<g)$ for any polynomial in $\Lambda[g'\in X : g'<g]$.  
Similarly, for $k\geq 1$, write $\Poly(\leq k)$ for any polynomial in $\Lambda[g\in X : \deg_+\leq k]$.  
Then, by definition and induction,
\begin{equation}\label{eqn:asymptotic_g-formula}
F_{i,j;\ell} = g_{i,j;\ell}^- + \lambda g_{\hat{b}_{i,j};\ell-1} + \Poly(\leq \ell-2), 
\qquad
g_{a+\alpha dn;\ell} = g_{a;\ell} + \Poly(\leq \ell-1).
\end{equation}

\medskip
Now, $J_{\Delta}\subset \mathbb{A}^{G(\Delta)} = \{(g)_{g\in X}\}$ is defined by 
$E(\Delta) = \sum_{a\in A(\Delta)} |\Gaps(a+dm)|$ new equations as follows. 
We proceed by induction on $\ell$, starting with the initial case $\ell=1$.
\begin{enumerate}[wide,labelwidth=!,labelindent=0pt,itemindent=!]
\item
If $\hat{b}_{i,j}+dm+\ell\notin \Delta$ such that either $j\neq n-1$ or $i$ is $\ell$-admissible 
(Definition \ref{def:admissibility}), then $j\in I(i;\ell)$. 
By (\ref{eqn:s-equation}), the corresponding equation is
\begin{equation}
0 = s_{\hat{b}_{i,j}+dm;\ell} = g_{i,j;\ell}^- + \Poly(\leq \ell-1).
\quad\Longleftrightarrow\quad g_{i,j;\ell}^- = \Poly(\leq \ell-1).
\end{equation}

\item
If $\hat{b}_{i,n-1}+dm+\ell\notin \Delta$ and $i$ is not $\ell$-admissible, then necessarily $\ell\geq 2$ 
since $\Delta$ is admissible. 
By \cite[Lem.~3.14]{GMO25}, $\hat{I}(i;\ell-1)\neq \emptyset$.
By \cite[Cor.~3.12]{GMO25}, $i\neq d-\ell$. 
Hence, $g_{i;\ell-1}^+ > g_{i+1;\ell-1}^+$.

\begin{enumerate}[wide,labelwidth=!,labelindent=0pt,itemindent=!,label=(2.\arabic*)]
\item
If $\ell\geq 3$, then by (\ref{eqn:s-equation}), the new equation 
$0 = \sum_{j\in\mathbb{Z}/n} s_{\hat{b}_{i,j}+dm;\ell}$ becomes
\begin{eqnarray*}
0 &=& \lambda g_{i;\ell-1}^+ - \sum_{j:~\hat{b}_{i,j}+dm+1\in\Delta}F_{i,j;1}\,g_{\hat{b}_{i,j}+dm+1;\ell-1} \\
&&- \sum_{j:~\hat{b}_{i,j}+dm+\ell-1\in\Delta}F_{i,j;\ell-1}\,g_{\hat{b}_{i,j}+dm+\ell-1;1} + \Poly(\leq \ell-2).
\end{eqnarray*}
Using $F_{i,j;\ell-1}=g_{i,j;\ell-1}^-+\lambda g_{\hat{b}_{i,j};\ell-2}+ \Poly(\leq \ell-3)$, this simplifies to
\begin{eqnarray*}
0 &=& \lambda g_{i;\ell-1}^+ - \sum_{j:~\hat{b}_{i,j}+dm+1\in\Delta}F_{i,j;1}\,g_{\hat{b}_{i,j}+dm+1;\ell-1}\\
&&- \sum_{j:~\hat{b}_{i,j}+dm+\ell-1\in\Delta}g_{i,j;\ell-1}^-\,g_{\hat{b}_{i,j}+dm+\ell-1;1} + \Poly(\leq \ell-2).
\end{eqnarray*}

\begin{enumerate}[wide,labelwidth=!,labelindent=0pt,itemindent=!,label=(2.1.\arabic*)]
\item
We simplify further the expression $\sum_{j:~\hat{b}_{i,j}+dm+1\in\Delta}F_{i,j;1}\,g_{\hat{b}_{i,j}+dm+1;\ell-1}$.

\noindent
Since $\Delta$ is admissible, $I(i;1) \subsetneq \mathbb{Z}/n$, hence $I(i;1)^c\neq \emptyset$. 
\begin{itemize}[wide,labelwidth=!,labelindent=0pt,itemindent=!]
\item
For $j\in I(i;1)^c$, write
\[
\hat{b}_{i,j}+dm+1 = \hat{b}_{i+1,j+1+j_1} + \gamma_{i,j}dn,\quad \gamma_{i,j}\geq 0,
\]
where $j_1$ depends only on $i$\footnote{Precisely, $j_1=0$ if $0\leq i\leq d-2$; if $i=d-1$, 
$j_1$ is uniquely determined by $j_1m\equiv 1 \mod n$.}. 
Then by (\ref{eqn:inductive_g-formula}),
\[
g_{\hat{b}_{i,j}+dm+1;\ell-1} = g_{\hat{b}_{i+1,j+1+j_1};\ell-1} + \Poly(\leq \ell-2).
\]
Since $\hat{b}_{i,j}+dm+\ell \notin \Delta$, we conclude $j+1+j_1\in \hat{I}(i+1;\ell-1)$, 
so $g_{i+1;\ell-1}^+\in X$.

\item
For $j\in I(i;1)$, we use (1): $0=s_{\hat{b}_{i,j}+dm;1}=F_{i,j;1}=g_{i,j;1}^-+\lambda$, 
with $g_{i,j;1}^-\in X$. Thus
\[
\sum_{j:~\hat{b}_{i,j}+dm+1\in\Delta}F_{i,j;1} = \sum_{j\in\mathbb{Z}/n}F_{i,j;1} = n\lambda.
\]
\end{itemize}
As a consequence, we obtain
\begin{eqnarray*}
&&\sum_{j:~\hat{b}_{i,j}+dm+1\in\Delta}F_{i,j;1}g_{\hat{b}_{i,j}+dm+1;\ell-1} = \sum_{j:~\hat{b}_{i,j}+dm+1\in\Delta}F_{i,j;1}g_{\hat{b}_{i+1,j+1+j_1};\ell-1} + \Poly(\leq \ell-2)\\
&=&\sum_{j:~\hat{b}_{i,j}+dm+1\in\Delta}F_{i,j;1}(g_{\hat{b}_{i+1,j+1+j_1};\ell-1} - \frac{g_{i+1;\ell-1}^+}{N_{i+1;\ell-1}}) 
+ \frac{n\lambda}{N_{i+1;\ell-1}}g_{i+1;\ell-1}^+ + \Poly(\leq \ell-2)\\
&=& \frac{n\lambda}{N_{i+1;\ell-1}}g_{i+1;\ell-1}^+ + \Poly(<g_{d-\ell;\ell-1}^+).\quad (g_{\hat{b}_{i+1,j+1+j_1};\ell-1} - \frac{g_{i+1;\ell-1}^+}{N_{i+1;\ell-1}}\in \Gen_{i+1;\ell-1}^-)
\end{eqnarray*}

\item
We simplify further the expression $\sum_{j:~\hat{b}_{i,j}+dm+\ell-1\in\Delta}g_{i,j;\ell-1}^-\,g_{\hat{b}_{i,j}+dm+\ell-1;1}$.

\noindent
If $\exists j$ s.t. $\hat{b}_{i,j}+dm+\ell-1\in \Delta$, then $i$ is ($\ell-1$)-admissible, and $I(i;\ell-1)^c = \{j\in\mathbb{Z}/n:~\hat{b}_{i,j}+dm+\ell-1\in \Delta\}\neq \emptyset$. In this case:
\begin{itemize}[wide,labelwidth=!,labelindent=0pt,itemindent=!]
\item
$\forall j\in I(i;\ell-1)^c$, we may write $\hat{b}_{i,j}+dm+\ell-1 = \hat{b}_{i+\ell-1,j+1+j_{\ell-1}} + \gamma_{i,j}^{(\ell-1)}dn$, $\gamma_{i,j}^{(\ell-1)}\geq 0$, where $j_{\ell-1}\in\mathbb{Z}/n$ depends only on $i$\footnote{Precisely, $j_{\ell-1}\in \mathbb{Z}/n$ is uniquely characterized by $j_{\ell-1}m  \equiv \floor*{\frac{i + \ell-1}{d}} \mod n$.}. Then by (\ref{eqn:inductive_g-formula}),
$g_{\hat{b}_{i,j}+dm+\ell-1;1} = g_{\hat{b}_{i+\ell-1,j+1+j_{\ell-1}};1}$.

\item
$\forall j\in I(i;\ell-1)$, i.e., $\hat{b}_{i,j}+dm+\ell-1 \notin \Delta$, as $i$ is ($\ell-1$)-admissible,
by (1), we have $g_{i,j;\ell-1}^- = \Poly(\leq \ell-2)$.

\item
$\forall j\in \hat{I}(i;\ell-1)^c$, i.e., $\hat{b}_{i,j}+\ell -1 \in \Delta$, we may write $\hat{b}_{i,j}+\ell-1 = \hat{b}_{i+\ell-1,j+j_{\ell-1}} + \tilde{\gamma}_{i,j}^{(\ell-1)}dn$, $ \tilde{\gamma}_{i,j}^{(\ell-1)}\geq 0$. 
It follows that $\hat{b}_{i,j}+dm+\ell =  \hat{b}_{i+\ell-1,j+j_{\ell-1}} + \tilde{\gamma}_{i,j}^{(\ell-1)}dn + dm + 1\notin \Delta$, $\Rightarrow$ $\hat{b}_{i+\ell-1,j+j_{\ell-1}} + dm + 1\notin \Delta$. Then $0 = s_{\hat{b}_{i+\ell-1,j+j_{\ell-1}}+dm;1} = F_{i+\ell-1,j+j_{\ell-1};1} = g_{i+\ell-1,j+j_{\ell-1};1}^- + \lambda$. 
Thus, 
\[
\sum_{j\in\hat{I}(i;\ell-1)} g_{i+\ell-1,j+j_{\ell-1};1}^- = - \sum_{j\in \hat{I}(i;\ell-1)^c}g_{i+\ell-1,j+j_{\ell-1};1}^- = (n - N_{i;\ell-1})\lambda.
\]
\end{itemize}
As a consequence, we obtain
\begin{eqnarray*}
&& \sum_{j:~\hat{b}_{i,j}+dm+\ell-1\in\Delta}g_{i,j;\ell-1}^-g_{\hat{b}_{i,j}+dm+\ell-1;1} = \sum_{j\in\mathbb{Z}/n}g_{i,j;\ell-1}^-g_{\hat{b}_{i+\ell-1,j+1+j_{\ell-1}};1} + \Poly(\leq \ell-2)\\
&=& -\sum_{j\in\hat{I}(i;\ell-1)} g_{i+\ell-1,j+j_{\ell-1};1}^- g_{\hat{b}_{i,j};\ell-1} + \Poly(\leq \ell-2)\\
&=& -\sum_{j\in\hat{I}(i;\ell-1)} g_{i+\ell-1,j+j_{\ell-1};1}^- (g_{\hat{b}_{i,j};\ell-1} - \frac{g_{i;\ell-1}^+}{N_{i;\ell-1}}) - \frac{n-N_{i;\ell-1}}{N_{i;\ell-1}}\lambda g_{i;\ell-1}^+  + \Poly(\leq \ell-2)\\
&=& - \frac{n-N_{i;\ell-1}}{N_{i;\ell-1}}\lambda g_{i;\ell-1}^+ + \Poly(<g_{d-\ell;\ell-1}^+).\quad (g_{\hat{b}_{i,j};\ell-1} - \frac{g_{i;\ell-1}^+}{N_{i;\ell-1}} \in \Gen_{i;\ell-1}^-)
\end{eqnarray*}
\end{enumerate}

\medskip\noindent
Altogether, the new equation $0 = \sum_{j\in\mathbb{Z}/n}s_{\hat{b}_{i,j}+dm;\ell}$ becomes:
\begin{eqnarray}\label{eqn:summing_equation-ell_at_least_3}
0 &=& \sum_{j\in\mathbb{Z}/n}s_{\hat{b}_{i,j}+dm;\ell} = \lambda g_{i;\ell-1}^+ - \frac{n\lambda}{N_{i+1;\ell-1}}g_{i+1;\ell-1}^+ + \frac{n-N_{i;\ell-1}}{N_{i;\ell-1}}\lambda g_{i;\ell-1}^+ +  \Poly(<g_{d-\ell;\ell-1}^+)\nonumber\\
&=& \frac{n\lambda}{N_{i;\ell-1}}g_{i;\ell-1}^+ - \frac{n\lambda}{N_{i+1;\ell-1}}g_{i+1;\ell-1}^+ + \Poly(<g_{d-\ell;\ell-1}^+).
\end{eqnarray}
In particular, it takes the form 
\[
g_{i;\ell-1}^+ = \Poly(<g_{i;\ell-1}^+).
\]

\item
If $\ell = 2$, then by (\ref{eqn:s-equation}),
\[
0 =  \sum_{j\in\mathbb{Z}/n} s_{\hat{b}_{i,j}+dm;2} = \lambda g_{i;1}^+ - \sum_{j:~\hat{b}_{i,j}+dm+1\in\Delta}F_{i,j;1}g_{\hat{b}_{i,j}+dm+1;1} + \Poly(\leq 0).
\]
By the same reasoning as for $\ell\geq 3$, we have
\[
\sum_{j:~\hat{b}_{i,j}+dm+1\in\Delta}F_{i,j;1}g_{\hat{b}_{i,j}+dm+1;1} = \frac{n\lambda}{N_{i+1;1}}g_{i+1;1}^+ + \Poly(<g_{d-2;1}^+).
\]
Altogether, the new equation $0 = \sum_{j\in\mathbb{Z}/n}s_{\hat{b}_{i,j}+dm;\ell}$ becomes:
\begin{equation}\label{eqn:summing_equation-ell_equals_2}
0 =  \sum_{j\in\mathbb{Z}/n} s_{\hat{b}_{i,j}+dm;2} = \lambda g_{i;1}^+ - \frac{n\lambda}{N_{i+1;1}}g_{i+1;1}^+ + \Poly(<g_{d-2;1}^+).
\end{equation}
In particular, it takes the form 
\[
g_{i;1}^+ = \Poly(<g_{i;1}^+).
\]
\end{enumerate}
\end{enumerate}

Finally, \textbf{define} 
\begin{eqnarray}\label{eqn:dependent_variables_for_equations-affine_paving_for_Hilbert_scheme}
Y &:=& \{g_{i,j;\ell}^-:~\hat{b}_{i,j}+dm+\ell \notin\Delta \text{ s.t. $j\neq n-1$ or $i$ is $\ell$-admissible}\}\nonumber\\
&&\sqcup \{g_{i;\ell}^+:~\text{$i$ is not ($\ell+1$)-admissible}\}.
\end{eqnarray}
Then $Y\subset X$, with $|Y| = E(\Delta)$. By above, 
\[
J_{\Delta} \cong \mathbb{A}_{X\setminus Y}^{\dim\Delta} \subset \mathbb{A}_X^{G(\Delta)},
\]
where the dependent variables $Y$ are determined inductively in the form
\[
g = \Poly(<g), \quad g\in Y.
\]

\medskip\noindent
\textbf{Step~3.}
We show that $H_{\Delta}^{[\tau]} \isomorphic \mathbb{A}^{\dim\Delta - \Gaps(\tau_0)}$.

Since $\tau_0\in\Delta$, we may write $\tau_0 = \hat{b}_{u,v} + \alpha dn$ with $\alpha\geq 0$.
By \textbf{Steps} $1$ and $2$, we have 
\[
H_{\Delta}^{[\tau]} \isomorphic \{g_{\tau_0} = t^{\tau_0}\} \subset J_{\Delta}\isomorphic \mathbb{A}_{X\setminus Y}^{\dim\Delta} \subset \mathbb{A}_X^{G(\Delta)}.
\]
That is, the extra defining equations for $H_{\Delta}^{[\tau]}$ are: 
$g_{\tau_0;\ell} = 0$ for all $\ell \in \Gaps(\tau_0) - \tau_0$.

Fix any $\ell\geq 1$ s.t. $\tau_0 + \ell =   \hat{b}_{u,v} + \alpha dn + \ell \notin \Delta$. 
It follows that $\hat{b}_{u,v} + \ell \notin \Delta$, i.e., $v\in \hat{I}(u;\ell)$. Hence, $g_{\hat{b}_{u,v};\ell}^+\in X$.
By (\ref{eqn:inductive_g-formula}) and (\ref{eqn:asymptotic_g-formula}), the equation $0 = g_{\tau_0;\ell} = g_{\hat{b}_{u,v}+\alpha dn;\ell}$ becomes
\begin{equation}\label{eqn:extra_equation_for_cell_decomposition_of_Hilbert_scheme}
0 = g_{\hat{b}_{u,v}+\alpha dn;\ell} = g_{\hat{b}_{u,v};\ell} + \Poly(\leq \ell-1) = \frac{g_{u;\ell}^+}{N_{u;\ell}} + \Poly(< g_{d-\ell-1;\ell}^+).\quad (g_{\hat{b}_{u,v};\ell} -  \frac{g_{u;\ell}^+}{N_{u;\ell}} \in \Gen_{u;\ell}^-)
\end{equation}
In particular, it takes the form 
\[
g_{u;\ell}^+ = \Poly(< g_{d-\ell-1;\ell}^+).
\]

Inductively on $\ell$, we replace these equations by equivalent ones as follows:
\begin{itemize}[wide,labelwidth=!,labelindent=0pt,itemindent=!]
\item
If $g_{u;\ell}^+ \notin Y$ (i.e., $u$ is $(\ell+1)$-admissible), then we keep the equation $g_{u;\ell}^+ = P(< g_{d-\ell-1;\ell}^+)$ unchanged and define $u_{\ell} := u$.

\item
If $g_{u;\ell}^+\in Y$ (i.e., $u$ is not ($\ell+1$)-admissible), then by \cite[Cor.3.12]{GMO25}, $u\neq d-\ell-1\in \mathbb{Z}/d$. Moreover, by \textbf{Step} $2$, in the description of $J_{\Delta}\isomorphic \mathbb{A}_{X\setminus Y}^{\dim \Delta}\subset \mathbb{A}_X^{G(\Delta)}$, we already have an equation of the form:
\[
g_{u;\ell}^+ = C g_{u+1;\ell}^+ + \Poly(<g_{d-\ell-1;\ell}^+),
\]
where $C$ is a nonzero constant depending only on $u$ and $\ell$.
Hence, equation (\ref{eqn:extra_equation_for_cell_decomposition_of_Hilbert_scheme}) is equivalent to a new equation of the form
\[
g_{u+1;\ell}^+ =  \Poly(<g_{d-\ell-1;\ell}^+).
\]
If $g_{u+1;\ell}^+ \notin Y$, then we stop and define $u_{\ell}:= u+1$.
Otherwise, $u+1$ is not $(\ell+1)$-admissible. Then by the same argument, $u+1\neq d-\ell-1$, and the equation $g_{u+1;\ell}^+ =  \Poly(<g_{d-\ell-1;\ell}^+)$ is further equivalent to a new equation of the form
\[
g_{u+2;\ell}^+ =  \Poly(<g_{d-\ell-1;\ell}^+).
\] 
Repeating this procedure finitely many times, we eventually obtain some $u_{\ell} \in \mathbb{Z}/d$ such that (\ref{eqn:extra_equation_for_cell_decomposition_of_Hilbert_scheme}) is equivalent to
\[
g_{u_{\ell};\ell}^+ = \Poly(< g_{d-\ell-1;\ell}^+),
\] 
with $g_{u_{\ell};\ell}^+\notin Y$. Now, we terminate.
\end{itemize}

Finally, define 
\[
Y_{\tau_0}:= Y \sqcup \{g_{u_{\ell};\ell}^+:~\tau_0+\ell\in \Gaps(\tau_0)\} \subset X.
\]
Then $|Y_{\tau_0}| = E(\Delta) + |\Gaps(\tau_0)|$.
Moreover, the above shows that 
\[
H_{\Delta}^{[\tau]} \isomorphic \mathbb{A}_{X\setminus Y_{\tau_0}}^{\dim\Delta - |\Gaps(\tau_0)|} \subset J_{\Delta}\isomorphic \mathbb{A}_{X\setminus Y}^{\dim\Delta} \subset \mathbb{A}_X^{G(\Delta)},
\] 
with the extra dependent variables $Y_{\tau_0}\setminus Y$ are determined by inductive formulas
\[
g_{u_{\ell};\ell}^+ = \Poly(< g_{d-\ell-1;\ell}^+).
\]
This completes the proof of Theorem \ref{thm:affine_paving_for_Hilbert_scheme_over_generic_curve_singularity}.
\end{proof}

In conclusion, we obtain:
\begin{corollary}\label{cor:affine_paving_for_Hilbert_scheme}
Let $(C,0)$ is a generic irreducible planar curve singularity parametrized as in Definition \ref{def:generic_planar_curve_singularities}. Then for all $\tau\geq 0$, the Hilbert scheme $\Hilb^{[\tau]}(C,0)$ admits an affine paving:
\[
\Hilb^{[\tau]}(C,0) = \bigsqcup_{\substack{\Delta \in \Adm(dn,dm) \\ \tau - e(\Delta) \in \Delta}} H_{\Delta}^{[\tau]}, 
\quad \text{with} \quad 
H_{\Delta}^{[\tau]} \cong \mathbb{A}^{\dim\Delta - |\Gaps(\tau - e(\Delta))|}.
\]
In particular, the (co)homology of $\Hilb^{[\tau]}(C,0)$ is pure.
\end{corollary}

\begin{proof}
Via the Abel--Jacobian map $\AJ_{\tau}:\Hilb^{[\tau]}(C,0) \hookrightarrow \overline{J}(C,0)$, the increasing finite closed filtration is induced from that of $\overline{J}(C,0)$.
The result now follows immidiately from Theorem \ref{thm:affine_paving_for_Hilbert_scheme_over_generic_curve_singularity}.
\end{proof}

\begin{example}\label{ex:affine_paving_for_Hilbert_scheme_of_a_generic_planar_curve_singularity}
Let $(C,0)$ be the generic planar curve singularity parametrized by
\[
x = t^6,\quad y = t^9 + \lambda t^{10},\quad \lambda = \lambda_1 \in \field^{\times}.
\]
So, $(n,m,d) = (2,3,3)$, $\delta = 21$, and
\[
(\hat{a}_{i,j})_{0\leq i\leq d-1,0\leq j\leq n-1} = 
\left(\begin{array}{cc} 
0 & 9\\
19 & 28\\
38 & 47 
\end{array}\right),\quad 
(a_{i,j})_{0\leq i\leq d-1,0\leq j\leq n-1} = 
\left(\begin{array}{cc} 
0 & 1\\
3 & 4\\
6 & 7
\end{array}\right).
\]
Consider $\Delta\isomorphic (c_{i,j})_{0\leq i\leq d-1,0\leq j\leq n-1}  \in \Adm(dn=4,dm=6)$ with
\[
(c_{i,j})_{0\leq i\leq d-1,0\leq j\leq n-1} = 
\left(\begin{array}{cc} 
0 & 1\\
0 & 3\\
3 & 5
\end{array}\right).~~\Rightarrow~~
A(\Delta) = (\hat{b}_{i,j} = \hat{a}_{i,j} - dnc_{i,j})_{0\leq i\leq d-1,0\leq j\leq n-1} = 
\left(\begin{array}{cc} 
0 & 3\\
19 & 10\\
20 & 17
\end{array}\right)
\]
Then $\mathbb{N}\setminus\Delta = \{1,2,4,5,7,8,11,13,14\}$, $\max A(\Delta) = 20$, $c(\Delta) = 15$, and $e(\Delta) = 12$.

Take $\tau_0 = 12 \in \Delta$, and $\tau:= \tau_0 + e(\Delta) = 24$. We will compute $H_{\Delta}^{[\tau]}$.

\medskip\noindent
\textbf{Step~0.} We recall the general setup.

\begin{enumerate}[wide,labelwidth=!,labelindent=0pt,itemindent=!]
\item
$J_{\Delta}$ parameterizes $R$-submodules $M\subset \tilde{R}$ such that $\ord_t(M) = \Delta$.
Equivalently, it parameterizes 
\[
M:=\oplus_{a\in A(\Delta)}\cO_F\cdot g_a \subset \tilde{R},
\] 
such that $yM \subset M$, where 
\[
g_a = t^a + \sum_{\ell\geq 1: a+\ell\notin\Delta}g_{a;\ell}t^{a+\ell},
\] 
with $g_{a;\ell} \in \field$ indeterminate.
Concretely,
\begin{eqnarray*}
g_0 &=& 1 + g_{0;1}t + g_{0;2}t^2 + g_{0;4}t^4 + g_{0;5}t^5 + g_{0;7}t^7 + g_{0;8}t^8 + g_{0;11}t^{11} + g_{0;13}t^{13} + g_{0;14}t^{14};\\
g_3 &=& t^3 + g_{3;1}t^4 + g_{3;2}t^5 + g_{3;4}t^7 + g_{3;5}t^8 + g_{3;8}t^{11} + g_{3;10}t^{13} + g_{3;11}t^{14};\\
g_{19} &=& t^{19};\\
g_{10} &=& t^{10} + g_{10;1}t^{11} + g_{10;3}t^{13} + g_{10;4}t^{14};\\
g_{20} &=& t^{20};\\
g_{17} &=& t^{17}.
\end{eqnarray*}
There are $19$ variables in total.

\item
For any $b\in\Delta$, there exists a unique $g_b \in M$ of the form
\[
g_b = t^b + \sum_{\ell\geq 1: b+\ell\notin\Delta} g_{b;\ell}t^{b+\ell},\quad g_{b;\ell} \in \field[g_{a;\ell}: a\in A(\Delta), a+\ell \notin \Delta].
\]
By above, it remains to describe $g_b \mod t^{c(\Delta)}$ for $b\in [0,c(\Delta)-1]\cap \Delta$, $b\notin A(\Delta)$.
They are as follows:
\begin{eqnarray*}
g_{12} \mod t^{15} &\equiv& x^2g_0 \equiv  t^{12} + g_{0;1}t^{13} + g_{0;2}t^{14};\quad (13,14\notin\Delta)\\
xg_3 \mod t^{15} &\equiv& t^9 + g_{3;1}t^{10} + g_{3;2}t^{11} + g_{3;4}t^{13} + g_{3;5}t^{14},~~\Rightarrow~~\\
g_9 \mod t^{15} &\equiv& xg_3 - g_{3;1}g_{10} \equiv t^9 + (g_{3;2} - g_{3;1}g_{10;1})t^{11} + (g_{3;4} - g_{3;1}g_{10;3})t^{13} + (g_{3;5} - g_{3;1}g_{10;4})t^{14};\\
xg_0 \mod t^{15} &\equiv& t^6 + g_{0;1}t^7 + g_{0;2}t^8 + g_{0;4}t^{10} + g_{0;5}t^{11} + g_{0;7}t^{13} + g_{0;8}t^{14},~~\Rightarrow~~\\
g_6 \mod t^{15} &\equiv& xg_0 - g_{0;4}g_{10}\\
&\equiv& t^6 + g_{0;1}t^7 + g_{0;2}t^8 + (g_{0;5} - g_{0;4}g_{10;1})t^{11} + (g_{0;7} - g_{0;4}g_{10;3})t^{13} + (g_{0;8} - g_{0;4}g_{10;4})t^{14}.
\end{eqnarray*}

\item
The condition $yM\subset M$ is equivalent to $s_{a+dm} = 0$ for all $a\in A(\Delta)$, where
\[
s_{a+dm} = \sum_{\ell\in\Gaps(a+dm) - (a+dm)}s_{a+dm;\ell}t^{a+dm+\ell},\quad s_{a+dm;\ell} \in \field[g_{a;\ell}: a\in A(\Delta), a+\ell \notin \Delta],
\] 
is defined by 
\[
yg_a - s_{a+dm} \in \oplus_{a\in A(\Delta)}\cO_F\cdot g_a  \mod t^{c(\Delta) = 15}.
\]
Thus,
\[
J_{\Delta} \isomorphic \{(g_{a;\ell})_{a\in A(\Delta),a+\ell\notin\Delta}:~s_{a+dm;\ell} = 0, a\in A(\Delta), a+dm+\ell\notin \Delta\}.
\]

\item
Finally,
\[
H_{\Delta}^{[\tau]}  \isomorphic \{M\in J_{\Delta}:~t^{\tau_0}\in M\} \isomorphic \{(g_{a;\ell})_{a\in A(\Delta),a+\ell\notin\Delta}:~s_{a+dm;\ell} = 0, a\in A(\Delta), a+dm+\ell\notin \Delta;~g_{\tau_0} = t^{\tau_0}\}.
\]
\end{enumerate}

\medskip\noindent
\textbf{Step~1.}
We compute $s_{a+dm}$ for $a\in A(\Delta)$. 

By definition, $s_{a+dm;\ell} = 0$ when $a + dm + \ell \geq c(\Delta) = 15$.
For $a \in \{19, 10, 20, 17\}$, we have $yg_a \in t^{15}\tilde{R} \subset M$ automatically, so $s_{a+dm} = 0$.
It remains to compute $s_{a+dm}$ for $a \in \{0, 3\}$.

First, compute $yg_0 - xg_3 \mod t^{15}$:
\begin{eqnarray*}
&&yg_0 - xg_3 \mod t^{15}\\
&\equiv& (t^9+\lambda t^{10})(1 + g_{0;1}t + g_{0;2}t^2 + g_{0;4}t^4 + g_{0;5}t^5) - t^6(t^3 + g_{3;1}t^4 + g_{3;2}t^5 + g_{3;4}t^7 + g_{3;5}t^8)\\
 &\equiv& (\lambda + g_{0;1} - g_{3;1})t^{10} + (\lambda g_{0;1} + g_{0;2} - g_{3;2})t^{11} + \lambda g_{0;2}t^{12} + (g_{0;4} - g_{3;4})t^{13} + (\lambda g_{0;4} + g_{0;5} - g_{3;5})t^{14}.
 \end{eqnarray*}
 Then,
 \begin{eqnarray*}
s_9 &=& yg_0 - xg_3 - (\lambda + g_{0;1} - g_{3;1})g_{10} - \lambda g_{0;2}g_{12} \mod t^{15}\\
&=& [\lambda g_{0;1} + g_{0;2} - g_{3;2} - (\lambda + g_{0;1} - g_{3;1})g_{10;1}]t^{11} + [g_{0;4} - g_{3;4} - (\lambda + g_{0;1} - g_{3;1})g_{10;3} - \lambda g_{0;2}g_{0;1}]t^{13}\\
&& + [\lambda g_{0;4} + g_{0;5} - g_{3;5} - (\lambda + g_{0;1} - g_{3;1})g_{10;4} - \lambda g_{0;2}^2]t^{14}.
\end{eqnarray*}
In other words, 
\begin{eqnarray*}
&&s_{9;2} = \lambda g_{0;1} + g_{0;2} - g_{3;2} - (\lambda + g_{0;1} - g_{3;1})g_{10;1};\\
&&s_{9;4} = g_{0;4} - g_{3;4} - (\lambda + g_{0;1} - g_{3;1})g_{10;3} - \lambda g_{0;2}g_{0;1};\\
&&s_{9;5} = \lambda g_{0;4} + g_{0;5} - g_{3;5} - (\lambda + g_{0;1} - g_{3;1})g_{10;4} - \lambda g_{0;2}^2.
\end{eqnarray*}

Next, compute $yg_3 - x^2g_0 \mod t^{15}$:
\begin{eqnarray*}
&&yg_3 - x^2g_0 \mod t^{15} \equiv (t^9 + \lambda t^{10})(t^3 + g_{3;1}t^4 + g_{3;2}t^5) - t^{12}(1 + g_{0;1}t + g_{0;2}t^2)\\
&\equiv& (\lambda + g_{3;1} - g_{0;1})t^{13} + (\lambda g_{3;1} + g_{3;2} - g_{0;2})t^{14}.
\end{eqnarray*}
Then,
\[
s_{12} = yg_3 - x^2g_0 \mod t^{15} = (\lambda + g_{3;1} - g_{0;1})t^{13} + (\lambda g_{3;1} + g_{3;2} - g_{0;2})t^{14}.
\]
In other words, 
\begin{eqnarray*}
&&s_{12;1} = \lambda + g_{3;1} - g_{0;1};\\
&&s_{12;2} = \lambda g_{3;1} + g_{3;2} - g_{0;2}.
\end{eqnarray*}

\medskip\noindent
\textbf{Step~2.} We compute $J_{\Delta}$.

\begin{enumerate}[wide,labelwidth=!,labelindent=0pt,itemindent=!]
\item
By definition (see (\ref{eqn:index_set_hat-i-ell})), the nonempty $\hat{I}(i;\ell)$ for $i\in\mathbb{Z}/d\isomorphic \{0,1,\cdots,d-1\}$, $\ell\geq 1$, are:
\begin{eqnarray*}
&&\hat{I}(0;\ell) = \{0,1\} = \mathbb{Z}/n,~\forall \ell\in\{1,2,4,5,8,11\};\quad \hat{I}(0;\ell) = \{0\},~\forall \ell\in\{7,13,14\};\quad \hat{I}(0;10) = \{1\};\\
&&\hat{I}(1;\ell) = \{1\},~\forall \ell \in \{1,3,4\}.
\end{eqnarray*}
Since $N_{i;\ell}:= |\hat{I}(i;\ell)|$, the nonzero $N(i;\ell)$'s are:
\[
N(0;\ell) = 2,~\forall \ell\in\{1,2,4,5,8,11\};\quad N(0;\ell) = 1,~\forall \ell\in\{7,10,13,14\};\quad N(1;\ell) = 1,~\forall \ell \in \{1,3,4\}.
\]
By definition (see (\ref{eqn:basis_vector_g_i,ell-plus})), the nonzero $g_{i;\ell}^+$ are:
\begin{eqnarray*}
&&g_{0;\ell}^+ = g_{0;\ell} + g_{3;\ell}, ~\ell \in \{1,2,4,5,8,11\};\quad g_{0;\ell}^+ = g_{0;\ell}, ~\ell \in \{7,13,14\};\quad g_{0;10}^+ = g_{3;10};\\
&&g_{1;\ell}^+ = g_{10;\ell}, ~\ell \in \{1,3,4\}.
\end{eqnarray*}
The nonzero $\Gen_{i;\ell}^-$ spaces are:
\[
\Gen_{0;\ell}^- = \field\cdot (g_{0;\ell} - g_{3;\ell}),\quad \ell \in\{1,2,4,5,8,11\}.
\]

Following the proof of Theorem \ref{thm:affine_paving_for_Hilbert_scheme_over_generic_curve_singularity}, we construct a new basis for each nonzero $\Gen_{0;\ell}^-$:

Recall by \eqref{eqn:index_set_I-i,ell_1}--\eqref{eqn:index_set_I-i,ell_2} that for each $i\in \mathbb{Z}/d$ and $\ell\geq 1$: 
\begin{itemize}
\item
If $i$ is $\ell$-admissible, then $I(i;\ell):= \{j\in\mathbb{Z}/n:~\hat{b}_{i,j} + dm + \ell \notin \Delta\} \subsetneq \hat{I}(i;\ell)$. 

\item
Else, $I(i;\ell):= \{0,1,\cdots,n-2\} \subsetneq \mathbb{Z}/n$. 
\end{itemize}
Note that $0$ is not $\ell$-admissible if and only if $\ell = 2$. By a direct computation:
\[
I(0;1) = \{1\};\quad I(0;\ell) = \{0\}, ~\ell\in\{2,4,5\};\quad I(0;\ell) = \emptyset, ~\ell\in\{8,11\}.
\]
Recall that 
\[
g_{i,j;\ell}^-:= g_{\hat{b}_{i,j};\ell} - g_{\hat{b}_{i,j+1};\ell},\quad \forall j\in I(i;\ell).
\]
For $\ell\in \{8,11\}$, fix $j_0 = 1 \in \hat{I}(0;\ell)\setminus I(0;\ell) = \{0,1\}$. Then,
$\overline{I}(0;\ell):= \hat{I}(0;\ell)\setminus (I(0;\ell)\sqcup \{j_0\}) = \{0\}$, and
\[
h_{0,0;\ell}^-:= g_{\hat{b}_{0,0};\ell} - g_{\hat{b}_{0,j_0};\ell} = g_{0;\ell} - g_{3;\ell} \quad (= g_{0,0;\ell}^-).
\]
In all other cases, $\overline{I}(0;\ell):= \emptyset$.

Now,
\[
\Gen_{0;\ell}^- = \oplus_{j\in I(0;\ell)}\field\cdot g_{0,j;\ell}^- \oplus \oplus_{j\in\overline{I}(0;\ell)}\field\cdot h_{0,j;\ell}^-.
\]
Altogether, as in (\ref{eqn:new_variables_for_equations-affine_paving_for_Hilbert_scheme}), we obtain a new basis for $\Gen = \oplus_{a\in\Delta,a+\ell\notin\Delta}\field\cdot g_{a;\ell}$ given by
\begin{eqnarray*}
X &=& \{g_{0;\ell}^+: \ell \in\{1,2,4,5,7,8,10,11,13,14\};~~g_{1;\ell}^+: \ell \in \{1,3,4\};\\
&&g_{0,1;1}^-;~~g_{0,0;\ell}^-: \ell \in \{2,4,5\};~~h_{0,0;\ell}^-: \ell \in \{8,11\}\}.
\end{eqnarray*}

We take $X$ as the new set of variables. It remains to solve the $5$ equations $s_{a+dm;\ell} = 0$ for $a\in A(\Delta)$, $a+\ell\notin\Delta$.

\item
We review some conventions.
Recall that $\deg_+ g_{a;\ell} = \ell$ for all $a\in A(\Delta)$.
It follows that $\deg_+ g_{i;\ell}^+ = \ell$, $\deg_+ g_{i,j;\ell}^- = \ell$, and $\deg_+ h_{i,j;\ell}^- = \ell$.
The \textbf{partial order} on $X$ is defined by:
\[
g_{\bullet,\bullet;\ell}^-, h_{\bullet,\bullet;\ell}^- < g_{\bullet;\ell}^+ < g_{\bullet,\bullet;\ell+1}^-, h_{\bullet,\bullet;\ell+1}^-;\quad g_{d-\ell;\ell}^+ > g_{d-\ell+1;\ell}^+ > \cdots > g_{d-\ell-1;\ell}^+.
\]
In particular,
\[
g_{0;1}^+ > g_{1;1}^+,\quad g_{0;4}^+ > g_{1;4}^+.
\]
For any $g\in X$, $\Poly(<g)$ denotes any polynomial in $g'\in X$ with $g' < g$. For any $k\geq 1$, $\Poly(\leq k)$ denotes any polynomial in $g'\in X$ with $\deg_+g' \leq k$.

\item
We list the equations $s_{a+dm;\ell} = 0$ in increasing order of $\ell$.
Note that $i$ is not $\ell$-admissible if and only if $i = 0$ and $\ell = 2$.
    
\begin{enumerate}[wide,labelwidth=!,labelindent=0pt,itemindent=!,label=(3.\arabic*)]
\item
For $\ell = 1$, we get the single equation $({\color{blue}\mathrm{I}})$:
\[
0 = s_{12;1} = \lambda + g_{3;1} - g_{0;1} = \lambda + g_{0,1;1}^- \quad\Leftrightarrow\quad g_{0,1;1}^- = -\lambda = \Poly(<g_{0,1;1}^-).
\]

\item
For $\ell = 2$, since $i = 0$ is not 2-admissible, we replace $0 = s_{\hat{b}_{0,n-1}+dm;2} = s_{12;2}$ with the new equation:
\[
0 = \sum_{j\in\mathbb{Z}/n} s_{\hat{b}_{0,j}+dm;2} = s_{9;2} + s_{12;2} = \lambda g_{0;1}^+ - (\lambda - g_{0,1;1}^-)g_{1;1}^+.
\]
Using $({\color{blue}\mathrm{I}})$: $g_{0,1;1}^- = -\lambda$, this becomes $({\color{blue}\mathrm{II}})$:
\[
0 = \lambda g_{0;1}^+ - 2\lambda g_{1;1}^+ =  \lambda g_{i;1}^+ - \frac{n\lambda}{N_{i+1;1}}g_{i+1;1}^+ + \Poly(<g_{d-2;1}^+)~~(i = 0)~~\Leftrightarrow~~ g_{0;1}^+ = 2g_{1;1}^+ = \Poly( < g_{0;1}^+).
\]

For $\ell = 2$, there is an additional equation $({\color{blue}\mathrm{III}})$:
\begin{eqnarray*}
0 &=& s_{9;2} = \lambda g_{0;1} + g_{0;2} - g_{3;2} - (\lambda + g_{0;1} - g_{3;1})g_{10;1} 
= \frac{\lambda}{2}(g_{0;1}^+ - g_{0,1;1}^-) + g_{0,0;2}^- - (\lambda - g_{0,1;1}^-)g_{1;1}^+.\\
&\Leftrightarrow& g_{0,0;2}^- = \frac{\lambda}{2}(g_{0,1;1}^- - g_{0;1}^+) + (\lambda - g_{0,1;1}^-)g_{1;1}^+ = -\frac{\lambda^2}{2} + \lambda g_{1;1}^+ = \Poly( < g_{0,0;2}^-).
\end{eqnarray*}

\item
For $\ell = 4$, we get the fourth equation $({\color{blue}\mathrm{IV}})$:
\begin{eqnarray*}
0 &=& s_{9;4} = g_{0;4} - g_{3;4} - (\lambda + g_{0;1} - g_{3;1})g_{10;3} - \lambda g_{0;2}g_{0;1}\\
&=& g_{0,0;4}^- - (\lambda - g_{0,1;1}^-)g_{1;3}^+ - \frac{\lambda}{4}(g_{0;2}^+ + g_{0,0;2}^-)(g_{0;1}^+ - g_{0,1;1}^-),\\
&\Leftrightarrow& g_{0,0;4}^- = (\lambda - g_{0,1;1}^-)g_{1;3}^+ + \frac{\lambda}{4}(g_{0;2}^+ + g_{0,0;2}^-)(g_{0;1}^+ - g_{0,1;1}^-) = \Poly( < g_{0,0;4}^-).
\end{eqnarray*}

\item
For $\ell = 5$, we get the fifth equation $({\color{blue}\mathrm{V}})$:
\begin{eqnarray*}
0 &=& s_{9;5} = \lambda g_{0;4} + g_{0;5} - g_{3;5} - (\lambda + g_{0;1} - g_{3;1})g_{10;4} - \lambda g_{0;2}^2\\
&=& \frac{\lambda}{2}(g_{0;4}^+ + g_{0,0;4}^-) + g_{0,0;5}^- - (\lambda - g_{0,1;1}^-)g_{1;4}^+ - \frac{\lambda}{4}(g_{0;2}^+ + g_{0,0;2}^-)^2,\\
&\Leftrightarrow& g_{0,0;5}^- = -\frac{\lambda}{2}(g_{0;4}^+ + g_{0,0;4}^-) + (\lambda - g_{0,1;1}^-)g_{1;4}^+ + \frac{\lambda}{4}(g_{0;2}^+ + g_{0,0;2}^-)^2 = \Poly( < g_{0,0;5}^-).
\end{eqnarray*}
\end{enumerate}

Altogether, by induction, we have solved the $5$ equations. The set of $5$ dependent variables is
\[
Y = \{g_{0,1;1}^-;~~g_{0,0;\ell}^-: \ell \in \{2,4,5\};~~g_{0;1}^+\} \subset X,
\]
with solutions of the form $g = \Poly( < g)$ for $g\in Y$. This matches the definition in (\ref{eqn:dependent_variables_for_equations-affine_paving_for_Hilbert_scheme}). 
Therefore, 
\[
J_{\Delta}\isomorphic \{g\in X:~s_{a+dm;\ell} = 0, a\in \Delta, a+\ell\notin \Delta\} \isomorphic \mathbb{A}^{|X\setminus Y|} = \mathbb{A}^{14}.
\]
\end{enumerate}

\medskip\noindent
\textbf{Step~3.} 
We compute $H_{\Delta}^{[\tau=24]}$.

Recall from \textbf{Step~0} that 
\[
H_{\Delta}^{[\tau]} = \{(g)_{g\in X} \in J_{\Delta}:~g_{\tau_0} = t^{\tau_0}\}, 
\]
and
\[
g_{12} = t^{12} + g_{0;1}t^{13} + g_{0;2}t^{14}.
\]
Thus, in addition to the $5$ equations defining $J_{\Delta}$, there are $2$ extra defining equations for $H_{\Delta}^{[\tau]}$:
\begin{enumerate}[wide,labelwidth=!,labelindent=0pt,itemindent=!,label=(4.\arabic*)]
\item
For $\ell = 1$, we get the single equation $g_{\tau_0;1} = 0$:
\[
0 = g_{12;\ell=1} = g_{0;1} = \frac{1}{2}(g_{0;1}^+ - g_{0,1;1}^-) \quad\Leftrightarrow\quad g_{0;1}^+ = g_{0,1;1}^- = \Poly( < g_{d-2;1}^+ = g_{1;1}^+).
\]
Recall from \textbf{Step~2} that $g_{0;1}^+ \in Y$ is already solved by equation $({\color{blue}\mathrm{II}})$:
\[
0 = \lambda g_{0;1}^+ - 2\lambda g_{1;1}^+ =  \lambda g_{i;1}^+ - \frac{n\lambda}{N_{i+1;1}}g_{i+1;1}^+ + \Poly(<g_{d-2;1}^+)~~(i = 0) \quad\Leftrightarrow\quad g_{0;1}^+ = 2g_{1;1}^+.
\]
Thus, we may replace the equation $g_{\tau_0;1} = 0$ by $({\color{blue}\mathrm{VI}})$:
\[
0 = g_{0,1;1}^- - 2 g_{1;1}^+ = - \frac{n}{N_{i+1;1}}g_{i+1;1}^+ + \Poly(<g_{d-2;1}^+) \quad\Leftrightarrow\quad g_{1,1}^+ = \frac{1}{2}g_{0,1;1}^- = \Poly(<g_{1;1}^+).
\]
Now, $g_{1;1}^+\in X\setminus Y$, so $u_{\ell} = u_1 := 1 \in \mathbb{Z}/d \isomorphic \{0,1,\cdots,d-1\}$.

\item
For $\ell = 2$, we get the single equation $g_{\tau_0;2} = 0$, labelled by $({\color{blue}\mathrm{VII}})$:
\begin{eqnarray*}
0 &=& g_{12;\ell=2} = g_{0;2} = \frac{1}{2}(g_{0;2}^+ + g_{0,0;2}^-) \quad\Leftrightarrow\quad g_{0;2}^+ = - g_{0,0;2}^- = \Poly( < g_{d-\ell-1;\ell}^- = g_{0;2}^+).
\end{eqnarray*}
We have $g_{0;2}^+\notin Y$. So $u_{\ell} = u_2 := 0 \in \mathbb{Z}/d$.
\end{enumerate}

Altogether, by induction, we have solved the $7$ equations. The set of dependent variables is
\[
Y_{\tau_0} = Y_{12}:= Y\sqcup\{g_{u_{\ell};\ell}^+:~\tau_0 + \ell\notin\Delta\} = \{g_{0,1;1}^-;~~g_{0,0;\ell}^-: \ell \in \{2,4,5\};~~g_{0;1}^+;~~g_{1;1}^+;~~g_{0;2}^+\} \subset X,
\]
with solutions of the form $g = \Poly( < g)$ for $g\in Y_{\tau_0}$.

Therefore, 
\[
H_{\Delta}^{[\tau]} = H_{\Delta}^{[24]} \isomorphic \{g\in X:~s_{a+dm;\ell} = 0, a\in \Delta, a+\ell\notin \Delta;~g_{\tau_0;\ell} =0, \tau_0 + \ell\neq 0\} \isomorphic \mathbb{A}^{|X\setminus Y_{\tau_0}|} = \mathbb{A}^{12}.
\]
Indeed, by a direct verification, $\dim\Delta = 14$ and $|\Gaps(\tau_0)| = 2$, so 
\[
\dim\Delta - |\Gaps(\tau_0)| = 14 - 2 = 12.
\]
\end{example}

\subsection{Cell indices and dimensions via $(dn,dm)$-Dyck paths}\label{subsec:cell_indices_and_dimensions_via_Dyck_paths}

Let $(C,0)$ be the generic irreducible planar curve singularity as before.
In this subsection, we describe the combinatorics of the affine paving of $\Hilb^{[\tau]}(C,0)$ in terms of \textbf{$(dn,dm)$-Dyck paths}, extending the description of $\overline{J}(C,0)$ given in \cite{GMO25}.

\subsubsection{Admissible $(dn,dm)$-invariant subsets revisited}

To begin with, we fix some notations.
\begin{enumerate}[wide,labelwidth=!,labelindent=0pt,itemindent=!]
\item
Since $\gcd(n,m) = 1$, there exist unique integers $0<\uformn<n, 0<\vformn<m$ such that
\begin{equation}
\uformn m - \vformn n = 1.
\end{equation}

\item
For all $i,j\in\mathbb{Z}$, recall that $\hat{a}_{i,j}:= jdm + (dmn+1)i$. \textbf{Define}
\begin{equation}
a_{i,j} := \floor*{\frac{\hat{a}_{i,j}}{dn}} = mi + \floor*{\frac{jdm+i}{dn}} = mi + \floor*{\frac{jm}{n}},\quad \forall 0\leq i\leq d-1.
\end{equation}
In particular, $a_{0,0}=0$. 
By definition, 
\[
\delta=|\mathbb{N}-\Gamma| = \sum_{0\leq i\leq d-1,0\leq j\leq n-1} a_{i,j}.
\]
Moreover, for each $0\leq i\leq d-1$, $0\leq j\leq n-1$, \textbf{set} $a_{in+j}:= a_{i,j}$. Then
\[
a_0\leq a_1\leq \cdots\leq a_{dn-1}
\]
form a Young diagram of size $\delta$.

\item
Let $\Delta\subset \mathbb{N}$ be any fixed $dn$-invariant subset with $\Delta\supset \Gamma$.
We can write the set of $dn$-generators as
\begin{equation}
A(\Delta) = \{\hat{b}_{i,j} = \hat{a}_{i,j} - dnc_{i,j}, 0\leq i\leq d-1, 0\leq j\leq n-1\},\quad 0\leq c_{i,j}\in \mathbb{Z} \leq a_{i,j}.
\end{equation}
For later use, extend $c_{i,j}$ to all $j\in\mathbb{Z}$ by defining $c_{i,j+n} := c_{i,j}+m$. In particular, $\hat{b}_{i,j+n} = \hat{b}_{i,j}$.
Set $b_{i,j}:= \floor*{\frac{\hat{b}_{i,j}}{dn}} = a_{i,j}-c_{i,j}$. Then
\[
\delta(\Delta) = \sum_{0\leq i\leq d-1,0\leq j\leq n-1} b_{i,j},\quad e(\Delta) := \delta - \delta(\Delta) =  \sum_{0\leq i\leq d-1,0\leq j\leq n-1} c_{i,j}.
\]

\item
For $x,y\in\mathbb{Z}$, define
\begin{equation}
\rank(x,y):= mx - ny.
\end{equation}
Given $\Delta$ as above, define
$\vec{\rank}_{\Delta} = (\rank_{\Delta,0},\rank_{\Delta,1},\cdots,\rank_{\Delta,dn-1})$ by
\begin{equation}
\rank_{\Delta,in+j} = \rank_{\Delta,i,j} := \rank(in+j,c_{i,j}) = m(in+j) - nc_{i,j},\quad \forall 0\leq i\leq d-1, 0\leq j\leq n-1.
\end{equation}
Note that
\begin{equation}
0 \leq c_{i,j} \leq a_{i,j} = \floor{\frac{m(in+j)}{n}} \quad\Leftrightarrow\quad m(in+j)\geq \rank_{\Delta,i,j} \geq 0.
\end{equation}

\item
For later purpose, for $j\in\mathbb{Z}$, \textbf{define} $\underline{j} = \underline{j}_n \in\{0,1,\cdots,n-1\} \isomorphic \mathbb{Z}/n$ by $\underline{j} \equiv j \mod n$.
\end{enumerate}

\begin{lemma}\cite[Lem.~3.15]{GMO25}\label{lem:admissibility_implies_dmn+1-invariant}
Any $0$-normalized admissible $(dn,dm)$-invariant subset $\Delta\subset\mathbb{N}$ is also $(dmn+1)$-invariant.
\end{lemma}
Thus $\Adm(dn,dm)$ (see Definition \ref{def:admissibility}) can be identified with the set of $0$-normalized admissible $(dn,dm)$-invariant subsets.

The next lemma characterizes $\Delta\in\Adm(dn,dm)$ in terms of the the matrix $(c_{i,j})_{0\leq i\leq d-1,0\leq j\leq n-1}$.
\begin{lemma}\label{lem:semimodule_vs_c_ij}
Let $\Delta\subset \mathbb{N}$ be a $dn$-invariant subset with $\Delta\supset \Gamma$. Then:
\begin{enumerate}[wide,labelwidth=!,itemindent=!,labelindent=0pt]
\item
$\Delta$ is $dm$-invariant if and only if $(c_{i,j})$ is \textbf{row-monotone}:
\begin{equation*}
c_{i,j}\leq c_{i,j+1} \quad\Leftrightarrow\quad \rank_{\Delta,i,\underline{j+1}} \leq \rank_{\Delta,i,\underline{j}} + m),\quad \forall 0\leq i\leq d-1, 0\leq j\leq n-1.
\end{equation*}
In particular, for $j=n-1$, this requires $c_{i,n-1}\leq c_{i,0} + m$.

\item 
$\Delta$ is $(dmn+1)$-invariant if and only if $(c_{i,j})$ is \textbf{column-monotone}:
\[
c_{i,j}\leq c_{i+1,j},\quad \forall 0\leq i\leq d-2, 0\leq j\leq n-1;\quad c_{d-1,j} \leq c_{0,j+\uformn} + dm -\vformn, \quad \forall 0\leq j\leq n-1.
\]

\item 
For $0\leq i\leq d-2$, the index $i$ is admissible (Definition \ref{def:admissibility}) if and only if:
\begin{equation*}
\exists 0\leq j\leq n-1 \text{ such that } c_{i+1,j+1}\geq c_{i,j}+m \quad\Leftrightarrow\quad \rank_{\Delta,i+1,\underline{j+1}} \leq \rank_{\Delta,i,\underline{j}} + m.
\end{equation*}
In particular, for $j=n-1$, this requires $c_{i+1,0} \geq c_{i,n-1}$.
Moreover, $d-1$ is always admissible.
\end{enumerate}
\end{lemma}

\begin{proof}
\noindent{}$(1)$. 
By definition, $\Delta$ is $dm$-invariant iff $\hat{b}_{i,j}+dm\in\Delta$, $\forall 0\leq i\leq d-1, 0\leq j\leq n-1$, i.e., 
\[
\hat{b}_{i,j}+dm = \hat{a}_{i,j+1} - dnc_{i,j} \geq \hat{b}_{i,j+1} = \hat{a}_{i,j+1} - dnc_{i,j+1},
~~\Leftrightarrow~~ c_{i,j} \leq c_{i,j+1}.
\]
This in turn is equivalent to $\rank_{\Delta,i,\underline{j+1}}-\rank_{\Delta,i,\underline{j}}\leq m$, as can be checked separately for $j\leq n-2$ and $j=n-1$.

\noindent{}$(2)$. $\Delta$ is $(dmn+1)$-invariant iff $\hat{b}_{i,j}+dmn+1\in \Delta$, $\forall 0\leq i\leq d-1, 0\leq j\leq n-1$.
For $i\leq d-2$, this is equivalent to $c_{i,j}\leq c_{i+1,j}$.   
For $i=d-1$, the computation reduces to
\[
c_{d-1,j} \leq c_{0,j+\uformn} + dm - \vformn.
\]

\noindent{}$(3)$. 
By definition, $i$ is admissible iff $\hat{b}_{i,j}+dm+1\in\Delta$ for some $j$.  
For $i\leq d-2$, this condition reduces to $c_{i+1,j+1}\geq c_{i,j}+m$.  
Equivalently, $\rank_{\Delta,i+1,\underline{j+1}} \leq \rank_{\Delta,i,\underline{j}} + m$, as can be checked separately for $j\leq n-2$ and $j=n-1$.
For $i=d-1$, using $\hat{b}_{0,0} = 0\in \Delta$, one checks that admissibility holds automatically.
\end{proof}

\begin{definition}
Given $(c_{i,j})_{0\leq i\leq d-1,\,0\leq j\leq n-1}\in M_{d\times n}(\mathbb{Z})$ with $0\leq c_{i,j}\leq a_{i,j}$, we say that $(c_{i,j})$ is \textbf{admissible} if  
for each $0\leq i\leq d-2$ either $c_{i+1,0}\geq c_{i,n-1}$, or there exists $0\leq j\leq n-2$ such that $c_{i+1,j+1}\geq c_{i,j}+m$.

Given a rank matrix $(\rank_{i,j})_{0\leq i\leq d-1,\,0\leq j\leq n-1}\in M_{d\times n}(\mathbb{Z})$ with $0\leq \rank_{i,j}\in mj+n\mathbb{Z}\leq m(in+j)$, we say that $(\rank_{i,j})$ is \textbf{admissible} if:
\begin{enumerate}[wide,labelwidth=!,labelindent=0pt,itemindent=!,label=(\alph*)]
\item
(\emph{$(c_{\bullet,\bullet})$-row-monotone}) 
$\rank_{i,\underline{j+1}} \leq \rank_{i,\underline{j}}+m$ for all $i,j$;

\item
(\emph{$(c_{\bullet,\bullet})$-admissible}) for each $0\leq i\leq d-2$, there exists $0\leq j\leq n-1$ such that $\rank_{i+1,\underline{j+1}} \leq \rank_{i,\underline{j}}+m$.
\end{enumerate}
\end{definition}

\begin{corollary}\label{cor:admissible_invariant_subset_via_rank}
We have natural identifications
\begin{eqnarray*}
\Adm(dn,dm)
&\cong& \{(c_{i,j}):\,0\leq c_{i,j}\leq a_{i,j},\ (c_{i,j}) \text{ is row-monotone and admissible}\} \\
&\cong& \{(\rank_{i,j}):\,0\leq \rank_{i,j}\in mj+n\mathbb{Z}\leq m(in+j),\ (\rank_{i,j}) \text{ is admissible}\}.
\end{eqnarray*}
Here, the first identification comes from $A(\Delta) = \{\hat{a}_{i,j}-dnc_{i,j}\}$, and the second from $\rank_{i,j} = m(in+j)-nc_{i,j}$.
\end{corollary}

\begin{proof}
This follows immediately from Lemma \ref{lem:admissibility_implies_dmn+1-invariant} and Lemma \ref{lem:semimodule_vs_c_ij}.
\end{proof}

\textbf{From now on}, we will use the identification $\Delta\isomorphic (c_{i,j})\in\Adm(dn,dm)$, and write
\begin{equation}
\dim((c_{ij})) := \dim\Delta,\quad |(c_{i,j})|:= \sum_{0\leq i\leq d-1, 0\leq j \leq n-1}c_{i,j}.
\end{equation}

\subsubsection{The $(dn,dm)$-Dyck paths}

\begin{definition}
We introduce the following notions:
\begin{enumerate}[wide,labelwidth=!,labelindent=0pt,itemindent=!]
\item
A \textbf{box} is a unit square in $\mathbb{R}^2$ with vertices in the lattice $\mathbb{Z}^2$.  
\textbf{Convention:} we identify a box with its top-left vertex.  
For a box $z$ with top-left vertex $(x,y)$, set
\[
\rank(z) = \rank(x,y) := mx - ny.
\]

\item
Let $\cR_{dn,dm}$ be the set of boxes in the rectangle with vertices $(0,0)$, $(dn,0)$, $(0,dm)$, $(dn,dm)$.  
Under the above convention,
\[
\cR_{dn,dm} \cong \{(x,y)\in \mathbb{Z}^2 : 0\leq x\leq dn-1,\ 1\leq y\leq dm\}.
\]

\item
Let $\cR_{dn,dm}^+$ be the set of boxes contained in the triangle with vertices $(0,0),(dn,0),(dn,dm)$.  
Equivalently,
\[
\cR_{dn,dm}^+ \cong \{(x,y)\in \mathbb{Z}^2 : 0\leq x\leq dn-1,\ 1\leq y\leq \tfrac{m}{n}x\}.
\]

\item
Let $\cR_{dn,dm}^{y\leq 0}$ be the set of boxes given by
\[
\cR_{dn,dm}^{y\leq 0} \cong \{(x,y)\in \mathbb{Z}^2 : 0\leq x\leq dn-1,\ y\leq 0\}.
\]
Define
\[
\overline{\cR}_{dn,dm}^+ := \cR_{dn,dm}^+ \sqcup \cR_{dn,dm}^{y\leq 0}
\cong \{(x,y)\in \mathbb{Z}^2 : 0\leq x\leq dn-1,\ y\leq \tfrac{m}{n}x\}.
\]
\end{enumerate}
\end{definition}

\begin{definition}
For $0\leq i\leq d-1$, $0\leq j\leq n-1$, let $a_{in+j}:=a_{i,j}$.  
Define
\[
\Young(dn,dm) := \Bigl\{(y_{in+j}=y_{i,j})_{0\leq i\leq d-1,\,0\leq j\leq n-1} :
0\leq y_x \leq a_x,\ y_x \leq y_{x+1}\Bigr\}.
\]
Equivalently, $(y_{in+j})\in \Young(dn,dm)$ if and only if
\[
y_0 \leq y_1 \leq \cdots \leq y_{dn-1}
\]
forms a Young subdiagram of $a_0\leq a_1\leq \cdots \leq a_{dn-1}$.  
Any such $(y_{i,j})$ will be called a \textbf{$(dn,dm)$-type Young diagram}, or equivalently, a \textbf{$(dn,dm)$-Dyck path} (see \textbf{convention} below).
\end{definition}

\noindent\textbf{Convention.}  
We represent $(y_{i,j})\in \Young(dn,dm)$ by a Young diagram $D\subset \cR_{dn,dm}^+$, where $y_{i,j}$ is the number of boxes in the $(in+j+1)$-th column.  
The condition $y_{i,j}\leq a_{i,j}$ ensures that the boundary of $D$ lies weakly below the diagonal $y=\tfrac{m}{n}x$.  
Thus $(y_{i,j})$ corresponds naturally to a $(dn,dm)$-Dyck path (see Figure~\ref{fig:Dyck_paths}).

\begin{figure}[!htbp]
\begin{center}
\begin{tikzpicture}[scale=0.8]

\draw[thin] (0,0) rectangle (4,6);

\foreach \i in {1,2,3} {
    \draw[thin] (\i,0) -- (\i,6);
}

\foreach \j in {1,2,3,4,5} {
    \draw[thin] (0,\j) -- (4,\j);
}

\draw[thin] (0,0) -- (4,6);

\draw[red,line width=0.5mm,->] (0,0) -- (2,0) -- (2,3) -- (3,3) -- (3,4) -- (4,4) -- (4,6);

\fill[blue!20, opacity=0.5] (0,0) -- (2,0) -- (2,3) -- (3,3) -- (3,4) -- (4,4) -- (4,6) -- (4,0) -- (0,0);

\node[below left] at (0,0) {(0,0)};
\node[above right] at (4,6) {(4,6)};

\draw (4,1) node[right,blue] {$(y_{i,j}) = \left(\begin{array}{cc} 0 & 0\\ 3 & 4 \end{array}\right) \in \Young(4,6)$};

\draw (4,4.5) node[right,red] {The corresponding $(4,6)$-Dyck path};

\draw (3.3,1.5) node[right] {$z$};
\draw[thick,<->] (2.1,1.4) -- (2.9,1.4);
\draw (2.5,1.1) node[below] {{\tiny$\armlength(z) = 1$}};

\draw[thick,<->] (3.6,2.1) -- (3.6,3.9);
\draw (3.9,3) node[right] {{\tiny$\leglength(z) = 2$}};

\draw[fill=black] (3,2) circle (0.1);

\end{tikzpicture}
\end{center}
\caption{The correspondence between $(y_{i,j})\in \Young(dn,dm)$ and a $(dn,dm)$-Dyck path.  
Here $(n,m,d)=(2,3,2)$. The box $z$ has top-left vertex $(3,2)$, with $(\armlength,\leglength)(z)=(1,2)$ and $\rank(z)=5$.}
\label{fig:Dyck_paths}
\end{figure}
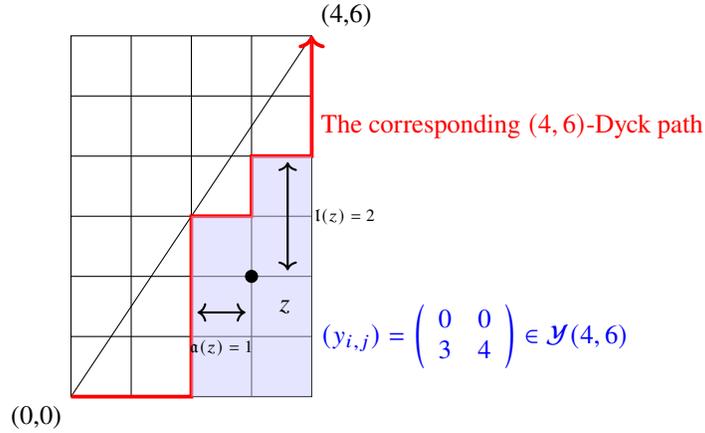

\begin{definition}
For a Young diagram $D = (y_{i,j})\in \Young(dn,dm)$, define
\[
\overline{D}:= D\sqcup \cR_{dn,dm}^{y\leq 0} \subset \overline{\cR}_{dn,dm}^+;\quad |D| := \sum_{0\leq i\leq d-1,0\leq j\leq n-1} y_{i,j}\quad(\textbf{size}).
\]
For any box $z\in D$, define its \textbf{arm length} $\armlength(z)$ and \textbf{leg length} $\leglength(z)$ as in Figure \ref{fig:Dyck_paths}.
Define
\begin{equation}
\dinv(D):=  \#\{z\in D:~\frac{\leglength(z)}{\armlength(z)+1} \leq \frac{m}{n} < \frac{\leglength(z) + 1}{\armlength(z)}\};\quad  \codinv(D):= \delta - \dinv(D).
\end{equation}
\end{definition}

Now, \textbf{fix} a Young diagram (equivalently, $(dn,dm)$-Dyck path) $D = (y_{i,j}) \in \Young(dn,dm)$, and
denote
\[
y_{in+j}:= y_{i,j},\quad 0\leq i\leq d-1, 0\leq j\leq n-1.
\]

\begin{definition}
Define the \textbf{rank vector} $\vec{\rank}_D = (\rank_{D,0},\rank_{D,1},\cdots,\rank_{D,dn-1})$ by
\begin{equation}
\rank_{D,in+j} = \rank_{D,i,j} := \rank(in+j,y_{in+j}) = m(in+j) - ny_{in+j}.
\end{equation}
\end{definition}

\noindent\textbf{Observations.}
\begin{enumerate}[wide,labelwidth=!,labelindent=0pt,itemindent=!,label=(\alph*)]
\item
For $0\leq x<dn-1$,
\[
y_x\leq y_{x+1}\quad\Leftrightarrow\quad \rank_{D,x+1}\leq \rank_{D,x}+m.
\]

\item
For $0\leq x\leq dn-1$,
\[
0\leq y_x\leq \tfrac{m}{n}x \quad\Leftrightarrow\quad 0\leq \rank_{D,x}\leq mx.
\]
In particular, $\rank_{D,0}=0$.

\item
For $0\leq x\leq (d-1)n$,
\[
y_{x+n-1}\leq y_x+m \quad\Leftrightarrow\quad \rank_{D,x}\leq \rank_{D,x+n-1}+m.
\]

\item
For $0\leq x\leq (d-1)n-1$,
\[
y_{x+n}\geq y_x+m \quad\Leftrightarrow\quad \rank_{D,x}\geq \rank_{D,x+n}.
\]
\end{enumerate}

\begin{corollary}\label{cor:Dyck_path_via_rank}
Via $\rank_x = mx - y_x$, we get a natural identification:
\begin{equation}
\Young(dn,dm) \isomorphic \{(\rank_{in+j} = \rank_{i,j})_{0\leq i\leq d-1,0\leq j\leq n-1}:~0\leq \rank_x \in mx + n\mathbb{Z} \leq mx, \rank_x \leq \rank_{x+1} + m\}.
\end{equation}
\end{corollary}

\begin{proof}
Immediate from $(a)$-$(b)$ above.
\end{proof}

\subsubsection{$(dn,dm)$-Dyck paths vs.\ admissible $(dn,dm)$-invariant subsets}\label{subsubsec:Dyck_paths_vs_admissible_invariant_subsets}

Fix $D = (y_{i,j})_{0\leq i\leq d-1,\,0\leq j\leq n-1} \in \Young(dn,dm)$.  
We introduce the following constructions (see Example~\ref{ex:enhanced_rank_associated_to_a_Dyck_path} for an illustration):

\begin{enumerate}[wide,labelwidth=!,labelindent=0pt,itemindent=!]
\item
\textbf{Rank vectors.}  
Set $\vec{\rank}^{[0]} := \vec{\rank}_D$.  
Inductively, for each $d-1\geq u \geq 1$, define the {\color{blue}move index} $\moveindex_u$ and the {\color{blue}rank vector} 
\[
\vec{\rank}^{[d-u]} = (\rank_0^{[d-u]},\dots,\rank_{un-1}^{[d-u]})
\] 
from $\vec{\rank}^{[d-u-1]} = (\rank_0^{[d-u-1]},\dots,\rank_{(u+1)n-1}^{[d-u-1]})$ by
\begin{equation}
\moveindex_u := \max\{\,0\leq x\leq un : \rank_x^{[d-u-1]} \leq \rank_{x+n-1}^{[d-u-1]} + m\,\}.
\end{equation}
Then set
\begin{equation}
\rank_x^{[d-u]} :=
\begin{cases}
\rank_x^{[d-u-1]}, & 0\leq x < \moveindex_u, \\
\rank_{x+n}^{[d-u-1]}, & \moveindex_u \leq x \leq un-1.
\end{cases}
\end{equation}
That is, $\vec{\rank}^{[d-u]}$ is obtained from $\vec{\rank}^{[d-u-1]}$ by deleting 
$\rank_{\moveindex_u}^{[d-u-1]},\dots,\rank_{\moveindex_u+n-1}^{[d-u-1]}$.

\item
\textbf{Permutations.}  
For each $0\leq j\leq n-1$ and $1\leq u\leq d-1$, define a {\color{blue}permutation} $\permutation_{j,u}$ of $[0,d-1]$\footnote{$[0,k] := \{0,1,\cdots,k\}$.} by
\begin{equation}
\permutationindex_{j,u}:=\min\{0\leq i\leq u:~in+j\geq \moveindex_u\}~~(\text{\color{blue}permutation index});\quad 
\permutation_{j,u} := (\permutationindex_{j,u}\,\cdots\,u).
\end{equation}
We also set $\permutationindex_{n,u}:=\permutationindex_{0,u}$ and $\permutation_{n,u}:=\permutation_{0,u}$.
From the definition one checks\footnote{Indeed, $\permutationindex_{j,u}n+j+1\geq \moveindex_u$ $\Rightarrow$ $\permutationindex_{j,u} + \floor{\frac{j+1}{n}} \geq \permutationindex_{j+1,u}$. Similarly, $\permutationindex_{j+1,u}n + \underline{j+1} + n - 1 \geq \moveindex_u$, where $j+1 \mod n \equiv \underline{j+1}\in \{0,1,\cdots,n-1\}$, $\Rightarrow$ $\permutationindex_{j+1,u}n + \underline{j+1} + n - 1 \geq \permutationindex_{j,u}n + j$, $\Rightarrow$ $\permutationindex_{j+1,u} \geq \permutationindex_{j,u} + \floor{\frac{j+1}{n}} - 1$.} that
\begin{equation}
\permutationindex_{j,u} + \floor{\frac{j+1}{n}} \geq \permutationindex_{j+1,u} \geq \permutationindex_{j,u} + \floor{\frac{j+1}{n}} - 1.
\end{equation}
Finally, define the compositions
\begin{equation}
\permutation_j := \permutation_{j,d-1}\circ\permutation_{j,d-2}\circ\cdots\circ\permutation_{j,1};\quad \permutation_n := \permutation_0.
\end{equation}

\item
\textbf{Modified $y$-matrices.}  
For $0\leq u\leq d-1$, $0\leq i\leq d-1$, $0\leq j\leq n-1$, set
\begin{equation}
y_{i,j}^{[u]} := y_{\permutation_{j,d-1}\circ\cdots\circ \permutation_{j,d-u}(i),\,j} 
 - \big(\permutation_{j,d-1}\circ\cdots\circ \permutation_{j,d-u}(i) - i\big)m.
\end{equation}
In particular, define
\begin{equation}
\Psi_d\big((y_{i,j})_{0\leq i<d,\,0\leq j<n}\big) := 
\big(y_{i,j}^{[d-1]}\big)_{0\leq i<d,\,0\leq j<n}.
\end{equation}

\item
\textbf{Enhanced ranks.}  
For each box $(x,y)$ (top-left vertex) with $0\leq x\leq dn-1$,
we may write $x = in+j$, $0\leq i\leq d-1$, $0\leq j\leq n-1$. For $0\leq u\leq d-1$, define
\begin{equation}
\hat{\rank}_D^{[u]}(x,y) := d\cdot \rank(x,y) + \permutation_{j,1}^{-1}\circ\cdots\circ \permutation_{j,d-u-1}^{-1}(i).
\end{equation}
In particular, the {\color{blue}enhanced rank} of $D$ is
\begin{equation}
\hat{\rank}_D(x,y) := \hat{\rank}_D^{[0]}(x,y) = d\cdot \rank(x,y) + \permutation_j^{-1}(i).
\end{equation}
Clearly, $\hat{\rank}_D^{[u]}(x,y+1) = \hat{\rank}_D^{[u]}(x,y) - dn$.
\end{enumerate}

The next result reformulates and refines \cite[Thm.~4.13]{GMO25} (see also \cite{GMV20}), where part~(2) was proved in different language.
\begin{theorem}\label{thm:admissible_invariant_subsets_vs_Dyck_paths}
For any $D=(y_{i,j}) \in \Young(dn,dm)$, we have:
\begin{enumerate}[wide,labelwidth=!,labelindent=0pt,itemindent=!]
\item The indices $\moveindex_u$ $(1\leq u\leq d-1)$ are well-defined and satisfy
\[
0 < \moveindex_1 < \moveindex_2 < \cdots < \moveindex_{d-1} \leq (d-1)n.
\]

\item The map $\Psi_d$ induces a bijection
\[
\Psi_d: \Young(dn,dm) \;\xrightarrow{\;\sim\;}\; \Adm(dn,dm),
\qquad (y_{i,j}) \mapsto (c_{i,j}) := \Psi_d((y_{i,j})),
\]
such that, writing $\Delta \cong (c_{i,j}) = \Psi_d(D)$,
\[
e(\Delta) = |(c_{i,j})| = |D|, 
\qquad \dim\Delta = \dim(c_{i,j}) = \codinv(D).
\]

\item The map $\hat{\rank}_D$ defines a bijection
\[
\hat{\rank}_D: \overline{\cR}_{dn,dm}^+ \;\xrightarrow{\;\sim\;}\; \mathbb{N},
\]
satisfying $\lfloor \hat{\rank}_D/d \rfloor = \rank$ and 
\[
\hat{\rank}_D(\cR_{dn,dm}^+\setminus D) = \mathbb{N}\setminus \Delta, 
\qquad \Delta=\Psi_d(D)\in \Adm(dn,dm).
\]
\end{enumerate}
\end{theorem}

\begin{remark}
Our construction of the bijection $\Psi_d$ is more direct than that in \cite{GMO25}, although it is not a priori evident that they coincide.  
In Section~\ref{sec:a_new_take_on_the_combinatorics} we will provide a new, simpler proof of Theorem~\ref{thm:admissible_invariant_subsets_vs_Dyck_paths} (see Propositions~\ref{prop:from_admissible_invariant_subsets_to_Dyck_paths}, \ref{prop:from_Dyck_paths_to_admissible_invariant_subsets}, \ref{prop:bijection_between_admissible_invariant_subsets_and_Dyck_paths}, Corollary~\ref{cor:enhanced_rank_is_a_bijection}, and Proposition~\ref{prop:sweep_map_vs_bijection}), which also shows that the two bijections agree (Remark~\ref{rem:compare_with_the_old_bijection}).
\end{remark}

\subsubsection{The combinatorics of the affine paving for the Hilbert schemes}

Recall from Corollary \ref{cor:affine_paving_for_Hilbert_scheme} that we obtained an affine paving for the Hilbert scheme of points on a generic planar curve singularity.
We now describe the combinatorics of this paving - cell indices and dimensions - in terms of the $(dn,dm)$-type Young diagrams (equivalently, $(dn,dm)$-Dyck paths).
Using Theorem \ref{thm:admissible_invariant_subsets_vs_Dyck_paths}, this is achieved as follows:

\begin{proposition}\label{prop:combinatorics_of_affine_paving_for_Hilbert_scheme}
Let $D = (y_{i,j})_{0\leq i\leq d-1,\,0\leq j\leq n-1} \in \Young(dn,dm)$ with $e = |D|$.
\begin{enumerate}[wide,labelwidth=!,labelindent=0pt,itemindent=!]

\item For any $\tau_0 \in \mathbb{N}$ written uniquely as
\[
\tau_0 = dr + k = d(m\ell - nh) + k, \qquad 0\leq k\leq d-1,\; 0\leq \ell \leq n-1,\; h \leq \Big\lfloor \tfrac{m\ell}{n}\Big\rfloor,
\]
we have
\[
\tau_0 \in \Delta := \Psi_d(D) 
\;\;\Longleftrightarrow\;\; h \leq y_{\permutation_{\ell}(k),\ell} - \permutation_{\ell}(k) m
\;\;\Longleftrightarrow\;\; r \geq \rank_{D,\permutation_{\ell}(k),\ell}.
\]

\item With $\Delta := \Psi_d(D)\in\Adm(dn,dm)$ and $\tau := \tau_0 + e$, the cell dimension is
\begin{equation}
\dim H_{\Delta}^{[\tau]} \;=\;
\Bigl|\{\,z\in D : \tfrac{\leglength(z)}{\armlength(z)+1} > \tfrac{m}{n} 
\text{ or } \tfrac{\leglength(z)+1}{\armlength(z)} \leq \tfrac{m}{n}\}\Bigr|
\;+\; \hat{\correctionterm}_D(\tau_0).
\end{equation}
Here,
\begin{align}
\correctionterm_D(\tau_0) 
&:= |\{\,z \in \cR_{dn,dm}^+\setminus D : \rank(z) = r,\; \hat{\rank}_D(z) < \tau_0\,\}| \nonumber\\
&= |\{\,0\leq i\leq k-1 : \rank_{D,\permutation_{\ell}(i),\ell} > r\,\}|,\\[0.5em]
\hat{\correctionterm}_D(\tau_0) 
&:= |\{\,z\in \cR_{dn,dm}^+\setminus D : \hat{\rank}_D(z) < \tau_0\,\}| \nonumber\\
&= |\{\,z\in \cR_{dn,dm}^+\setminus D : \rank(z) < r\,\}| + \correctionterm_D(\tau_0).
\end{align}
Moreover, $\correctionterm_D(\tau_0)=0$ for all $\tau_0\in \Gamma$.
\end{enumerate}
\end{proposition}

\begin{remark}
If $d = 1$, then $\correctionterm_D(\tau_0)=0$ identically.  
For $d > 1$, however, $\correctionterm_D(\tau_0)$ may be positive.  

For instance, let $(n,m,d)=(2,3,2)$ and $D=(y_{i,j})=\bigl(\begin{smallmatrix}0&0\\3&4\end{smallmatrix}\bigr)$ as in Figure~\ref{fig:Dyck_paths}.  
Then $\moveindex_1=2=(d-1)n$, $\permutation_{0,1}=\permutation_{1,1}=\id$, so $\Delta\cong (c_{i,j}):=\Psi_d((y_{i,j}))=(y_{i,j})$.  
Thus $\Delta\in\Adm(4,6)$ is
\[
\Delta = \{\hat{a}_{i,j}-4c_{i,j}\} + 4\mathbb{N} = \{0,6,1,3\}+4\mathbb{N}.
\]
Moreover,
\[
\hat{\rank}_D(x,y) = 2\cdot\rank(x,y)+\Big\lfloor \tfrac{x}{2}\Big\rfloor = 4x - 6y + \Big\lfloor \tfrac{x}{2}\Big\rfloor.
\]
Taking $\tau_0=3\in \Delta$, we have $r=\lfloor \tau_0/d\rfloor=1$, $k=\tau_0-dr=1$.  
Then $\{(x,y)\in\cR_{dn,dm}^+\setminus D : \rank(z)=r,\;\hat{\rank}_D(z)<\tau_0\}=\{(1,1)\}$, hence $\correctionterm_D(\tau_0)=1>0$.
\end{remark}

\begin{proof}
\noindent (1). 
By Theorem~\ref{thm:admissible_invariant_subsets_vs_Dyck_paths}, we have 
$c_{i,j}=y_{\permutation_j(i),j}-(\permutation_j(i)-i)m$, so $\Delta\cong(c_{i,j})\in\Adm(dn,dm)$.  
By definition,
\[
\tau_0=\hat{a}_{k,\ell}-dn(h+km)\in\Delta
\;\;\Longleftrightarrow\;\;
h+km \leq c_{k,\ell} = y_{\permutation_\ell(k),\ell}-(\permutation_\ell(k)-k)m
\]
\[
\Longleftrightarrow\;\;
h \leq y_{\permutation_\ell(k),\ell}-\permutation_\ell(k)m
\;\;\Longleftrightarrow\;\;
\rank_{D,\permutation_\ell(k),\ell} \leq r.
\]

\noindent (2). 
Since $\delta = |\cR_{dn,dm}^+|$, Theorem \ref{thm:admissible_invariant_subsets_vs_Dyck_paths}.$(2)$ gives
\[
\dim\Delta = \codinv(D) = |\cR_{dn,dm}^+\setminus D| + |\{z\in D:~\frac{\leglength(z)}{\armlength(z)+1} > \frac{m}{n}, \text{ or } \frac{\leglength(z)+1}{\armlength(z)} \leq \frac{m}{n}\}|.
\]
By Theorem \ref{thm:admissible_invariant_subsets_vs_Dyck_paths}.$(3)$, we have
\begin{eqnarray*}
|\cR_{dn,dm}^+\setminus D| - |\Gaps_{\Delta}(\tau_0)| &=& |(\mathbb{N}\setminus \Delta)\cap(0,\tau_0)| = |\{z\in\cR_{dn,dm}^+\setminus D:~\hat{\rank}_D(z) < \tau_0 = dr+k\}|\\
&=& |\{z\in \cR_{dn,dm}^+\setminus D:~\rank(z)< r\}| + \correctionterm_D(\tau_0).
\end{eqnarray*}
Thus, by Corollary~\ref{cor:affine_paving_for_Hilbert_scheme}, we obtain the desired formula for $\dim H_\Delta^{[\tau]}$.  

It remains to verify the properties of $\correctionterm_D(\tau_0)$. Indeed, by definition, Theorem \ref{thm:admissible_invariant_subsets_vs_Dyck_paths}.$(3)$, and $(1)$,
\[
\correctionterm_D(\tau_0) = |\{0\leq i\leq k-1:~dr+i\notin \Delta| = |\{0\leq i\leq k-1:~\rank_{D,\permutation_{\ell}(i),\ell} > r\}|.
\]
Finally, if $\tau_0 = dr + k = \hat{a}_{k,\ell} - dn(h+mk) \in\Gamma$, i.e., $h+mk\leq 0$, then for all $0\leq i\leq k$, we also have
\[
dr + i = \hat{a}_{i,\ell} - dn(h+mi) \in \Gamma \subset \Delta.
\]
Thus, $\correctionterm_D(\tau_0) = |\{0\leq i\leq k-1:~dr+i\notin \Delta|  = 0$.
\end{proof}

\begin{remark}\label{rem:cell_dimension_via_dn-generators}
Given $\Delta\in\Adm(dn,dm)$ and $\tau_0 = \tau - e(\Delta) \in \Delta$, one can also express $\dim H_{\Delta}^{[\tau]}$ directly in terms of the set of $dn$-generators $A(\Delta)$:
\begin{equation}
\dim H_{\Delta}^{[\tau]} = \sum_{a,b\in A(\Delta):\,a>b} \Big\lfloor \tfrac{a-b}{dn}\Big\rfloor 
- \sum_{a,b\in A(\Delta):\,a>b+dm} \Big\lfloor \tfrac{a-b-dm}{dn}\Big\rfloor 
- \sum_{a\in A(\Delta):\,a>\tau_0}\Big\lfloor \tfrac{a-\tau_0}{dn}\Big\rfloor.
\end{equation}
Indeed, for $b\in \Delta$, $a\in A(\Delta)$, and $\ell\in\mathbb{Z}$, one has $a - dn\ell \in \Gaps_{\Delta}(b)$ iff $b \leq a - dn\ell < a$. That is, $a > b$ and $1\leq \ell \leq \floor{\frac{a-b}{dn}}$. Thus, 
\[
|\{\,x\in\Gaps_\Delta(b): x\equiv a \pmod{dn}\}| = 
\begin{cases}
\lfloor \tfrac{a-b}{dn}\rfloor, & a>b,\\
0, & \text{otherwise}.
\end{cases}
\]
Hence,
\[
|\Gaps_\Delta(b)|=\sum_{a\in A(\Delta):\,a>b}\lfloor\tfrac{a-b}{dn}\rfloor,
\] 
and the stated formula follows from Corollary~\ref{cor:affine_paving_for_Hilbert_scheme}.
\end{remark}

\begin{example}\label{ex:enhanced_rank_associated_to_a_Dyck_path}
Let $(n,m,d) = (2,3,3)$. Recall $\hat{a}_{i,j} = jdm + (dmn+1)i$ and $a_{i,j} = \floor{\frac{a_{i,j}}{dn}} = mi + \floor{\frac{jm}{n}}$. So,

\[
(\hat{a}_{i,j})_{0\leq i\leq d-1,\,0\leq j\leq n-1} =
\begin{pmatrix}
0 & 9\\
19 & 28\\
38 & 47
\end{pmatrix},\qquad
(a_{i,j})_{0\leq i\leq d-1,\,0\leq j\leq n-1} =
\begin{pmatrix}
0 & 1\\
3 & 4\\
6 & 7
\end{pmatrix},
\]
and $\delta = \sum_{0\leq i\leq d-1,0\leq j\leq n-1}a_{i,j} = \frac{d(dmn - m - n + 1)}{2} = 21$.
Consider the $(dn,dm)=(4,6)$-type Young diagram
\[
D=(y_{i,j})=
\begin{pmatrix}
0 & 0\\
0 & 1\\
2 & 7
\end{pmatrix},\qquad e:=|D|=10.
\]

\begin{enumerate}[wide,labelwidth=!,labelindent=0pt,itemindent=!,label=(\Roman*)]
\item
As above, we make the following constructions:
\begin{enumerate}[wide,labelwidth=!,labelindent=0pt,itemindent=!,label=(\arabic*)]
\item
\textbf{Ranks.}
For $(i,j)$, we set $\rank_{D,i,j} := \rank(in+j,y_{i,j}) = m(in+j)-ny_{i,j}$.
Hence
\[
\vec{\rank}_D^{[0]} = (\rank_{D,0},\ldots,\rank_{D,5}) = (0,3,6,7,8,1).
\]
Removing blocks according to the definition of $\moveindex_u$, we find
\[
\moveindex_{2}=3,\qquad \rank_D^{[1]}=(0,3,6,1);\qquad
\moveindex_{1}=1,\qquad \rank_D^{[2]}=(0,1).
\]
Thus $0<\moveindex_1<\moveindex_2\leq (d-1)n$ as expected.

\item
\textbf{Permutations.}
For $j=0,1$ and $u=1,2$, we obtain
\[
\permutationindex_{0,1} = 1,~~\permutation_{0,1} = \id;\quad \permutationindex_{0,2} = 2,~~\permutation_{0,2} = \id;\quad
\permutationindex_{1,1} = 0,~~\permutation_{1,1} = (0~1);\quad \permutationindex_{1,2} = 1,~~\permutation_{1,2} = (1~2).
\]
Thus,
\[
\permutation_0 = \id;\quad \permutation_1 = (1~2)\circ(0~1) = (2~1~0).
\]

\item
\textbf{The set $\Delta$.}
Applying $\Psi_d$ gives
\[
(c_{i,j}) = \Psi_3((y_{i,j})) = (y_{i,j}^{[2]})_{0\leq i\leq 2,0\leq j\leq 1} = \begin{pmatrix}
y_{\permutation_0(0),0} - 3\permutation_0(0) & y_{\permutation_1(0),1} - 3\permutation_1(0)\\
y_{\permutation_0(1),0} - 3(\permutation_0(1) - 1) & y_{\permutation_1(1),1} - 3(\permutation_1(1) - 1)\\
y_{\permutation_0(2),0} - 3(\permutation_0(2) - 2) & y_{\permutation_1(2),1} - 3(\permutation_1(2) - 2)
\end{pmatrix} = 
\begin{pmatrix}
0 & 1\\
0 & 3\\
2 & 4
\end{pmatrix}.
\]
Hence,
\[
A(\Delta)=\{\hat{a}_{i,j}-dn\,c_{i,j}\}=\{0,3,19,10,26,23\},\qquad 
\Delta=\{0,3,19,10,26,23\}+6\mathbb{N}.
\]

\item
\textbf{Enhanced rank.}
For a box with top-left vertex $(in+j,y)$,
\[
\hat{\rank}_D(in+j,y)=d(m(in+j)-ny)+\permutation_j^{-1}(i).
\]
Explicitly,
\[
\begin{array}{ll}
\hat{\rank}_D(0,y)=-6y,& \hat{\rank}_D(1,y)=10-6y,\\
\hat{\rank}_D(2,y)=19-6y,& \hat{\rank}_D(3,y)=29-6y,\\
\hat{\rank}_D(4,y)=38-6y,& \hat{\rank}_D(5,y)=45-6y.
\end{array}
\]
Thus $\hat{\rank}_D(x,y+1)=\hat{\rank}_D(x,y)-6$.  
Figure~\ref{fig:enhanced_rank} illustrates $\hat{\rank}_D$: blue numbers correspond to $\Delta=\Psi_d(D)\in\Adm(dn,dm)$.
\end{enumerate}

\begin{figure}[!htbp]
\vspace{-0.1in}
\begin{center}
\begin{tikzpicture}[scale=0.7]

\draw[thin] (0,0) rectangle (6,9);

\foreach \i in {1,2,3,4,5} {
    \draw[thin] (\i,0) -- (\i,9);
}

\foreach \j in {1,2,3,4,5,6,7,8} {
    \draw[thin] (0,\j) -- (6,\j);
}

\draw[thin] (0,0) -- (6,9);

\draw[red, line width=0.5mm,->] (0,0) -- (3,0) -- (3,1) -- (4,1) -- (4,2) -- (5,2) -- (5,7) -- (6,7) -- (6,9);

\fill[blue!10, opacity=0.5] (3,0) -- (3,1) -- (4,1) -- (4,2) -- (5,2) -- (5,7) -- (6,7) -- (6,0) -- (3,0);

\node[blue] at (0.5,-0.5) {$0$};

\node[blue] at (1.5,-0.5) {$10$};
\node at (1.5,0.5) {$4$};

\node[blue] at (2.5,-0.5) {$19$};
\node at (2.5,0.5) {$13$};
\node at (2.5,1.5) {$7$};
\node at (2.5,2.5) {$1$};

\node[blue] at (3.5,-0.5) {$29$};
\node[blue] at (3.5,0.5) {$23$};
\node at (3.5,1.5) {$17$};
\node at (3.5,2.5) {$11$};
\node at (3.5,3.5) {$5$};

\node[blue] at (4.5,-0.5) {$38$};
\node[blue] at (4.5,0.5) {$32$};
\node[blue] at (4.5,1.5) {$26$};
\node at (4.5,2.5) {$20$};
\node at (4.5,3.5) {$14$};
\node at (4.5,4.5) {$8$};
\node at (4.5,5.5) {$2$};

\node[blue] at (5.5,-0.5) {$45$};
\node[blue] at (5.5,0.5) {$39$};
\node[blue] at (5.5,1.5) {$33$};
\node[blue] at (5.5,2.5) {$27$};
\node[blue] at (5.5,3.5) {$21$};
\node[blue] at (5.5,4.5) {$15$};
\node[blue] at (5.5,5.5) {$9$};
\node[blue] at (5.5,6.5) {$3$};

\node[blue] at (3,-1.2) {$\vdots$};

\node[below left] at (0,0) {$(0,0)$};
\node[above right] at (6,9) {$(6,9)$};


\end{tikzpicture}
\end{center}
\vspace{-0.1in}
\caption{The $(dn,dm)$-type Dyck path $D = (y_{in+j} = y_{i,j})$ and the associated \textbf{enhanced rank} $\hat{\rank}_D$ for Example \ref{ex:enhanced_rank_associated_to_a_Dyck_path}: In the figure, $(n,m,d) = (2,3,3)$, and $(y_0,y_1,\cdots,y_5) = (0,0,0,1,2,7)$. Only the positive values of $\hat{\rank}_D$ are indicated.}
\label{fig:enhanced_rank}
\end{figure}
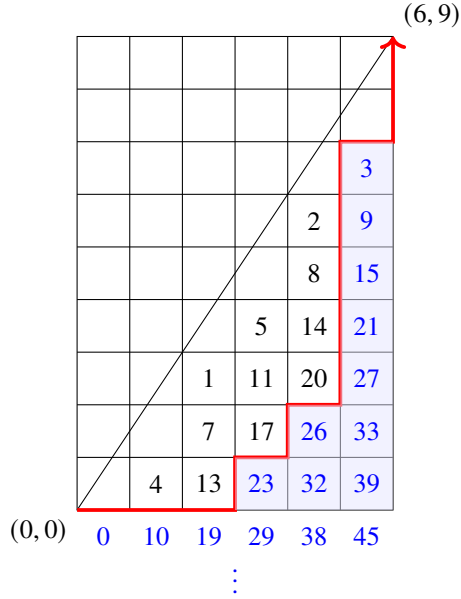

\item
\textbf{Dimension computation.}
We compute $\dim H_\Delta^{[\tau]}$ for $\tau_0=10\in\Delta$, hence $\tau=\tau_0+e=20$.
We \textbf{identify} a box $z\in \overline{\cR}_{dn,dm}^+$ with $\hat{\rank}_D(z)\in\mathbb{N}$ via the bijection $\hat{\rank}_D:\overline{\cR}_{dn,dm}^+ \xrightarrow[]{\simeq} \mathbb{N}$.

\emph{Method 1 (via Remark~\ref{rem:cell_dimension_via_dn-generators}).}  
With $A(\Delta)=\{0,3,10,19,23,26\}$,
\[
\dim H_\Delta^{[20]}=
\sum_{a>b}\Big\lfloor \tfrac{a-b}{dn}\Big\rfloor
-\sum_{a>b+dm}\Big\lfloor \tfrac{a-b-dm}{dn}\Big\rfloor
-\sum_{a>\tau_0}\Big\lfloor \tfrac{a-\tau_0}{dn}\Big\rfloor
=26-10-5=11.
\]

\emph{Method 2 (via Proposition~\ref{prop:combinatorics_of_affine_paving_for_Hilbert_scheme}).}  
One checks
\[
|\{z\in D:\tfrac{\leglength(z)}{\armlength(z)+1}>\tfrac{m}{n}\}|=5,\qquad
|\{z\in D:\tfrac{\leglength(z)+1}{\armlength(z)}\leq \tfrac{m}{n}\}|=0.
\]
Since $r=\lfloor \tau_0/d\rfloor=3$ and $k=1$, we get
\[
\hat{\correctionterm}_D(\tau_0)=|\{z\in\cR_{dn,dm}^+\setminus D:~\hat{\rank}_D(z)<\tau_0\}|=6.
\]
Therefore
\[
\dim H_\Delta^{[\tau]}=(5+0)+6=11,
\]
matching the previous result. Moreover,
\[
|\{z\in\cR_{dn,dm}^+\setminus D:\rank(z)<r\}|=6,\qquad
\correctionterm_D(\tau_0)=0,
\]
which agrees with
\[
\hat{\correctionterm}_D(\tau_0)=|\{z\in \cR_{dn,dm}^+\setminus D:\rank(z)<r\}|+\correctionterm_D(\tau_0).
\]
\end{enumerate}
\end{example}

\subsection{Connection to the ORS and Cherednik-Danilenko-Philipp conjectures}\label{subsec:connection_to_ORS_and_CDP}

\subsubsection{The weight polynomial}

We begin with some \textbf{terminology}.  
An \textbf{ind-variety} over $\field$ is a set $X$ with a filtration 
\[
X_{-1} = \emptyset \subset X_0 \subset X_1 \subset \cdots,\qquad X = \bigcup_{j\geq 0} X_j,
\] 
such that each $X_j$ is a $\field$-variety and the inclusions $i_j:X_j\hookrightarrow X_{j+1}$ are closed embeddings.

\begin{definition}
Let $X=\bigcup_{j\geq 0}X_j$ be an ind-variety over $\field$, and let $A$ be a commutative ring.
\begin{enumerate}[wide,labelwidth=!,labelindent=0pt,itemindent=!]
\item A map $\varphi:X\to A$ is a \textbf{constructible function} if each $X_j$ admits a decomposition $X_j=\bigsqcup_{\alpha\in I_j}X_{j,\alpha}$ into finitely many locally closed subvarieties on which $\varphi$ is constant.

\item Suppose $\mathbb{Z}[[\sqrt{q}]]\subset A$. For a constructible function $\varphi:X\to A$, define the \textbf{weight polynomial} of $(X,\varphi)$ by
\begin{equation}
\Weight(X,\varphi) = \int_X \varphi := \sum_{a\in A}\Weight(\varphi^{-1}(a);q)\cdot a,
\end{equation}
where $\Weight(-;q)$ is the \textbf{Deligne weight polynomial}: for any ind-variety $Y=\bigcup_{i\geq 0}Y_i$, 
\[
\Weight(Y;q):= \sum_{i,j,k}\dim\Gr_k^W H_c^j((Y_i\setminus Y_{i-1})^{\an};\mathbb{C})\,q^{k/2}(-1)^j.
\]
\end{enumerate}
\end{definition}

\begin{remark}\label{rem:Katz's_theorem}
By our convention, $\Weight(Y,q)=\mathfrak{W}(Y;t=-\sqrt{q})$ in \cite{ORS18}, thus the specialization to \cite{OS12} is $q=1$.
In addition, by N. Katz's theorem (see \cite[Appendix]{HRV08} or \cite[Thm.A.6]{DL12}), if a $\field$-variety $Y$ (defined over some number field) has a \emph{strongly polynomial-count}, i.e. there exists $P\in\mathbb{Z}[T]$ such that $|Y(\mathbb{F}_q)|=P(q)$ for all but finitely many characteristics, then 
\begin{equation}
\Weight(Y;q) = P(q).
\end{equation}
\end{remark}

\subsubsection{Hilbert $L$-function and the ORS conjecture}

Let $(C,0)$ be an irreducible planar curve singularity, and let $K_{(C,0)}$ be the associated singularity knot.
As before, $\delta = \delta(R) = \dim_{\field}\tilde{R}/R$, and assume $(C,0)$ is parametrized by
\[
x(t) = t^{dn},\quad y(t) = t^{dm} + \text{ h.o.t.},\quad \gcd(n,m) = 1,~~m > n.
\]

\begin{definition}\label{def:Hilbert_L-function}
Define the \textbf{length} and \textbf{special rank} on $\Hilb^{[\bullet]}(C,0)$ by
\begin{equation*}
(\length,\specialrank):\Hilb^{[\bullet]}(C,0) \rightarrow \mathbb{N}\times [1,dn]\footnote{Note that $1\leq \dim_{\field} I/\fm_R I \leq \dim_{\field}I/xI = dn$, for all $I\in \Hilb^{[\bullet]}(C,0)$.}: [I] \mapsto (\dim_{\field}R/I,\dim_{\field}I/\fm_RI),
\end{equation*}
where $\fm_R:=(x,y)R \subset R$ is the maximal ideal.
The \textbf{Hilbert $L$-function} of $(C,0)$ is
\begin{equation*}
L(q,t,a) = L((C,0);q,t,a):= (1-t)\int_{\Hilb^{[\bullet]}(C,0)} t^{\length}\prod_{j=1}^{\specialrank - 1}(1+q^ja).
\end{equation*}
\end{definition}

\noindent{}\textbf{Note}: In our case, the Milnor number is $\mu = 2\delta + 1 - b = 2\delta$, and the generating function $\overline{\cP}_{\alg}(a,q,t)$ in \cite[Conj.2]{ORS18} is related to $L(q,t,a)$ by:
\[
\overline{\cP}_{\alg}:= (\frac{a}{q})^{\mu-1} \text{\tiny$\sum_{\ell,m}$} q^{2\ell}a^{2m}t^{m^2}\fw(\Hilb^{[\ell\leq\ell+m]}(C,0);t) = (\frac{a}{q})^{2\delta - 1}\frac{1+a^2t}{1-q^2}L(t^2,q^2,a^2t).
\]
Here, $\Hilb^{[\ell\leq\ell+m]}(C,0)$ denotes the \textbf{$(\ell,\ell+m)$-nested Hilbert scheme}:
\[
\{(I_0,I_1)\in \Hilb^{[\ell]}(C,0)\times \Hilb^{[\ell+m]}(C,0):\ I_0\supset I_1\supset \fm_R I_0\}.
\]

Now, the Oblomkov-Rasmussen-Shende conjecture connects $L(q,t,a)$ to knot homologies:
\begin{conjecture}[{\cite[Conj.~2]{ORS18}}]\label{conj:ORS_conjecture} 
\[
\overline{\cP}_{\alg}(a,q,t)=\overline{\cP}(K_{(C,0)}),
\]
where $\overline{\cP}(-)$ is the generating function for triply graded HOMFLY--PT homology of knots/links.
\end{conjecture}

For generic singularities, our affine paving of $\Hilb^{[\bullet]}(C,0)$ yields an explicit computation of the lowest $a$-degree coefficient of $\overline{\cP}_{\alg}$, i.e. $L(q,t,0)$ (see Corollary \ref{cor:perverse_filtration_for_generic_planar_curve_singularities}).

\begin{lemma}[{\cite[Prop.~3]{ORS18}}]
The Hilbert $L$-function satisfies:
\[
L(q,t,a)\in\mathbb{Z}[\sqrt{q},t,a]~~(\text{\textbf{rationality}}); \qquad q^\delta t^{2\delta}L(q,1/(qt),a)=L(q,t,a)~~(\text{\textbf{functional equation}}).
\]
\end{lemma}

\begin{remark}
For $r\geq 1$, \textbf{denote} 
\[
\Hilb_r^{[\ell]}(C,0):=\{I\in\Hilb^{[\ell]}(C,0):~\dim_{\field}I/\fm_RI = r\} \hookrightarrow \Hilb^{[\ell]}(C,0),
\] 
i.e., the locally closed subscheme of ideals whose minimal number of generators is $r$. Then,
\[
L(q,t,a) = (1-t) \sum_{\ell\geq 0,r\geq 1}\Weight(\Hilb_r^{[\ell]}(C,0);q)t^{\ell}\prod_{j=1}^{r-1}(1+q^ja).
\]
Examples suggest that $\Hilb_r^{[\ell]}(C,0)$ always has a strongly polynomial-count. By N.Katz's theorem, this would imply a \emph{stronger rationality}: $L(q,t,a)\in\mathbb{Z}[q,t,a]$.
\end{remark}

\subsubsection{Hilbert $L$-function vs.\ perverse filtration}

Using $\overline{\{f(x,y) = 0\}} \hookrightarrow \mathbb{P}^2$, after blowing-ups, the planar curve singularity $(C,0)$ can be realized as the unique singularity of a \textbf{rational} locally planar curve $C$. In particular, the arithmetic genus of $C$ is $g(C) = \delta$.
By \cite[Prop.3.7]{Bea99}, we have 
\[
(\overline{J}(C,0))^{\red} \xrightarrow[]{\simeq} (\overline{J}C)^{\red},
\] 
where $\overline{J}C$ denotes the compactified Jacobian, i.e., the moduli space of rank $1$ torsion-free coherent sheaves of degree $0$ on $C$.
By \cite{FGS99}, there exists a deformation $\pi:\cC \rightarrow B$ of $C$ such that $B$ and the relative compactified Jacobian $\cJ = \overline{J}(\cC/B)$ are smooth.
By the BBDG decomposition theorem \cite{BBDG18},
\[
R\pi_*\underline{\mathbb{Q}}_{\cJ}[\delta + \dim B] = \oplus_{i=-\delta}^{\delta}\cF_i[-i],
\]
where each $\cF_i$ is a perverse sheaf on $B$. Say, $0 \in B$ and $C = \pi^{-1}(0)$. Set $b:= \dim B$.

\begin{definition}
The \textbf{perverse filtration} on $H^*(\overline{J}(C,0))$ is defined by
\begin{equation*}
P_kH^*(\overline{J}(C,0)) := \im[({}^p\tau_{\leq k+b}R\pi_*\underline{\mathbb{Q}}_{\cJ})|_{0} \rightarrow (R\pi_*\underline{\mathbb{Q}}_{\cJ})|_{0} = H^*(\overline{J}(C,0))]
= \oplus_{j=0}^k\cF_{j-\delta,0} [-j-b].
\end{equation*}
In particular, $\Gr_k^PH^*(\overline{J}(C,0)) = \cF_{k-\delta,0}[-k-b] \neq 0$ $\Rightarrow$ $0\leq k\leq 2\delta$.
\end{definition}

\begin{lemma}[{\cite{MS13,MY14}}]\label{lem:Hilbert_L-function_vs_perverse_filtration}
For any irreducible planar curve singularity $(C,0)$,
\[
L(q,t,0) = \sum_{i=0}^{2\delta}t^i\Weight(\Gr_i^PH^*(\overline{J}(C,0));q).
\]
\end{lemma}

As an application, we compute the perverse filtration on $H^*(\overline{J}(C,0))$ for generic singularities.
\begin{corollary}\label{cor:perverse_filtration_for_generic_planar_curve_singularities}
For any generic $(C,0)$ as in Definition \ref{def:generic_planar_curve_singularities},
\[
\sum_{i=0}^{2\delta}t^i\Weight(\Gr_i^PH^*(\overline{J}(C,0));q) = L(q,t,0) = (1-t)\sum_{D\in\Young(dn,dm)}\sum_{z\in \overline{D}} q^{|D| - \dinv(D) + \hat{\epsilon}_D(z)} t^{|D| + \hat{\rank}_D(z)}.
\]
\end{corollary}
See Remark \ref{rem:a_simplification_of_the_combinatorial_two-variable_CDP_conjecture} for a further simplification.

\begin{proof}
By Lemma \ref{lem:Hilbert_L-function_vs_perverse_filtration}, it suffices to compute $L(q,t,0)$.
The stated formula now follows from Corollary \ref{cor:affine_paving_for_Hilbert_scheme}, Theorem \ref{thm:admissible_invariant_subsets_vs_Dyck_paths} and Proposition  \ref{prop:combinatorics_of_affine_paving_for_Hilbert_scheme}.
\end{proof}

\subsubsection{Motivic superpolynomial and the Cherednik-Danilenko-Philipp conjecture}

\begin{definition}\label{def:normalized_motivic_superpolynomial}
Define the \textbf{minimal valuation} and \textbf{special rank} on $\overline{J}(C,0)$ as
\begin{equation*}
(\minval,\specialrank):\overline{J}(C,0) \rightarrow [0,\delta]\times [1,dn]: J\mapsto (\min\ord_t(J),\dim_{\field}J/\fm_RJ).
\end{equation*}
The \textbf{normalized motivic superpolynomial} of $(C,0)$ is then
\begin{equation*}
\mH^{\mot}(q,t,a):=\cH^{\mot}(qt,t,a);\quad \cH^{\mot}(q,t,a) := \int_{\overline{J}(C,0)} t^{\delta - \minval}\prod_{j=1}^{\specialrank - 1}(1+q^ja).
\end{equation*}
\end{definition}

The following is an abridged version of the Cherednik-Danilenko-Philipp (CDP) conjecture:
\begin{conjecture}[\cite{CD16,CP18}]\label{conj:abridged_CDP_conjecture}
$L(q,t,a) = \mH^{\mot}(q,t,a)$.
\end{conjecture}

From now on, assume that $(C,0)$ is \textbf{generic}. 
For any $D\in\Young(dn,dm)$, $\forall z\in \overline{D} = D \sqcup \overline{\cR}_{dn,dm}$, \textbf{denote}
\[
\correctionterm_D(z):= \correctionterm_D(\hat{\rank}_D(z));\quad \hat{\correctionterm}_D(z):= \hat{\correctionterm}_D(\hat{\rank}_D(z)).\quad (\text{see Proposition \ref{prop:combinatorics_of_affine_paving_for_Hilbert_scheme}.$(2)$})
\]
Recall that
\[
|D| - \dinv(D) = |\{w\in D:~\frac{\leglength(w)}{\armlength(w)+1} > \frac{m}{n} \text{ or } \frac{\leglength(w)+1}{\armlength(w)} \leq \frac{m}{n}\}|.
\]
From Corollary~\ref{cor:perverse_filtration_for_generic_planar_curve_singularities}, we know $L(q,t,0)$ explicitly.  
On the other hand, by Theorem \ref{thm:admissible_invariant_subsets_vs_Dyck_paths}, we have 
\[
\mH^{\mot}(q,t,0) = \sum_{D\in\Young(dn,dm)}(qt)^{\codinv(D)}t^{\delta - |D|} =  \sum_{D\in\Young(dn,dm)} q^{\codinv(D)} t^{\codinv(D) + \delta - |D|}.
\]
Thus, for generic $(C,0)$, the $(q,t)$-version of Conjecture \ref{conj:abridged_CDP_conjecture} reduces to:
\begin{conjecture}\label{conj:combinatorial_two-variable_CDP_conjecture_for_generic_curve_singularity}
\begin{equation*}
 (1-t)\sum_{D\in\Young(dn,dm)}\sum_{z\in \overline{D}} q^{|D| - \dinv(D) + \hat{\epsilon}_D(z)} t^{|D| + \hat{\rank}_D(z)}
 =
 \sum_{D\in\Young(dn,dm)} q^{\codinv(D)} t^{\codinv(D) + \delta - |D|}.
 \end{equation*}
\end{conjecture}

\begin{remark}\label{rem:a_simplification_of_the_combinatorial_two-variable_CDP_conjecture}
For each $D = (y_{in+j} = y_{i,j})_{0\leq i\leq d-1,0\leq j\leq n-1} \in \Young(dn,dm)$, \textbf{denote}
\begin{eqnarray*}
&&L_D(q,t):= (1-t)\sum_{z\in \overline{D}} q^{|D| - \dinv(D) + \hat{\epsilon}_D(z)} t^{|D| + \hat{\rank}_D(z)};\\
&&\mH_D^{\mot}(q,t):=  q^{\codinv(D)} t^{\codinv(D) + \delta - |D|}.
\end{eqnarray*}
Then, 
\[
L(q,t,0) = \sum_{D\in\Young(dn,dm)} L_D(q,t), \quad 
\mH^{\mot}(q,t,0) = \sum_{D\in\Young(dn,dm)} \mH_D^{\mot}(q,t).
\]
By Theorem \ref{thm:admissible_invariant_subsets_vs_Dyck_paths}, we can \textbf{identify} $\overline{D}$ with $\hat{\rank}_D(\overline{D})$ via $\hat{\rank}_{D}$. 
So,
\[
c(D) = \max(\cR_{dn,dm}^+\setminus D) + 1 = \max\{\hat{\rank}_D(x,y_x):~0\leq x\leq dn-1\} - dn + 1.
\] 
For any $z\geq c(D)$, we have $z + \mathbb{N} \subset \overline{D}$ and $\hat{\epsilon}_D(z) = |\cR_{dn,dm}^+\setminus D|$. Hence,
$|D| - \dinv(D) + \hat{\epsilon}_D(z) = \delta - \dinv(D) = \codinv(D)$.
Thus, for any fixed $z_0\geq c(D)$, we may rewrite $L_D(q,t)$ as a \textbf{finite sum}:
\begin{equation}
L_D(q,t) = q^{\codinv(D)}t^{|D| + z_0} + (1-t)\sum_{z\in \overline{D}\cap [0,z_0-1]}q^{|D| - \dinv(D) + \hat{\epsilon}_D(z)} t^{|D| + z}.
\end{equation}
\end{remark}

Above all, following a comment by Eugene Gorsky on the relevance of \cite{CGHM24} to our work, we are able to relate 
Conjecture~\ref{conj:combinatorial_two-variable_CDP_conjecture_for_generic_curve_singularity} 
back to the ORS conjecture~\ref{conj:ORS_conjecture}. 

In \cite{CGHM24}, the authors introduce a family of generalized Schröder polynomials 
$S_{\tau}(q,t,a)$ indexed by triangular partitions $\tau$. 
For the generic planar curve singularity $(C,0)$ considered here, the Young diagram 
\[
(a_{in+j} = a_{i,j} = mi + \big\lfloor \tfrac{jm}{n} \big\rfloor)_{0 \le i \le d-1,\; 0 \le j \le n-1}
\]
corresponds to $\tau = \tau_{md,nd}$ in their terminology. 
By definition,
\begin{equation}
S_{\tau_{dm,dn}}(q,t,0) = t^{\delta}\, \cH^{\mot}(t^{-1}, q, 0).
\end{equation}

Moreover, their results imply the following as a special case.

\begin{lemma}[{\cite[Thm.~C]{CGHM24}}]
For generic $(C,0)$ as above, up to a suitable normalization, we have
\[
S_{\tau_{dm,dn}}(q,t,a) = \overline{\cP}(K_{(C,0)}).
\]
\end{lemma}

Recall that the right-hand side is the generating function for the triply graded 
HOMFLY–PT (or Khovanov–Rozansky) homology of the knot $K_{(C,0)}$. 
As a consequence, we obtain the following.

\begin{corollary}
For generic $(C,0)$, the lowest $a$–degree part of the ORS conjecture 
\ref{conj:ORS_conjecture} is equivalent to 
Conjecture~\ref{conj:combinatorial_two-variable_CDP_conjecture_for_generic_curve_singularity}.
\end{corollary}

Finally, we compute an example for Conjecture \ref{conj:combinatorial_two-variable_CDP_conjecture_for_generic_curve_singularity}.
\begin{example}
Let $(n,m,d) = (2,3,2)$, then $\delta = \frac{d(dmn - m - n + 1)}{2} = 8$, and
\[
(\hat{a}_{in+j} = \hat{a}_{i,j})_{0\leq i\leq d-1,0\leq j\leq n-1} = (0, 6, 13, 19);\quad (a_{in+j} = a_{i,j})_{0\leq i\leq d-1,0\leq j\leq n-1} = (0, 1, 3, 4).
\]
Thus, $\Young(dn=4,dm=6)$ consists of $23$ Young diagrams $D = (y_{in+j} = y_{i,j})_{0\leq i\leq d-1,0\leq j\leq n-1}$ as follows:
For simplicity, we represent a $d$-by-$n$ matrix $(y_{in+j} = y_{i,j})_{0\leq i\leq d-1,0\leq j\leq n-1}$ as the vector $(y_0,y_1,\cdots,y_{dn-1})$.
{\small\[
\left.\begin{array}{lllll}
D_1 = (0, 0, 0, 0); & D_2 = (0, 0, 0, 1); & D_3 = (0, 0, 0 , 2); & D_4 = (0 , 0, 0 , 3); & D_5 = (0 , 0, 0 , 4);\\
D_6 = (0 , 0, 1 , 1); & D_7 = (0 , 0, 1 , 2); & D_8 = (0 , 0, 1 , 3); & D_9 = (0 , 0, 1 , 4); & D_{10} = (0 , 0, 2 , 2);\\
D_{11} = (0 , 0, 2 , 3); & D_{12} = (0 , 0, 2 , 4); & D_{13} = (0 , 0, 3 , 3); & D_{14} = (0 , 0, 3 , 4); & D_{15} = (0 , 1, 1 , 1);\\
D_{16} = (0 , 1, 1 , 2); & D_{17} = (0 , 1, 1 , 3); & D_{18} = (0 , 1, 1 , 4); & D_{19} = (0 , 1, 2 , 2); & D_{20} = (0 , 1, 2 , 3);\\
D_{21} = (0 , 1, 2 , 4); & D_{22} = (0 , 1, 3 , 3); & D_{23} = (0 , 1, 3 , 4). & &
\end{array}\right.
\]}
For $k\neq 5$, $\Delta_k:= \Psi_2(D_k) = D_k\in \Adm(4,6)$ and 
\[
\hat{\rank}_{D_k}(x,y) = d\cdot \rank(x,y) + \floor{\frac{x}{n}} = 6x - 4y +  \floor{\frac{x}{2}}.
\]
For $k = 5$, we have: $\moveindex_1 = 1$; $\permutationindex_{0,1} = 1$, $\permutation_0 = \permutation_{0,1} = \id$; $\permutationindex_{1,1} = 0$, $\permutation_1 = \permutation_{1,1} = (0~1)$. It follows that
\begin{eqnarray*}
&&\hat{\rank}_{D_5}(x=2i+0,y) = d\cdot \rank(x,y) + \permutation_0^{-1}(i) = 6x - 4y + i,\quad \forall i\in\{0,1\};\\ 
&&\hat{\rank}_{D_5}(x=2i+1,y) = d\cdot \rank(x,y) + \permutation_1^{-1}(i) = 6x - 4y + 1 - i, \quad \forall i\in\{0,1\}.
\end{eqnarray*}
See Figure \ref{fig:enhanced_ranks_for_4,6-Dyck_paths} for an illustration.

\begin{figure}[!htbp]
\vspace{-0.2in}
\begin{center}
\begin{tikzcd}[column sep=2pc,ampersand replacement=\&]

\begin{tikzpicture}[scale=0.5]

\draw[thin] (0,0) rectangle (4,6);

\foreach \i in {1,2,3} {
    \draw[thin] (\i,0) -- (\i,6);
}

\foreach \j in {1,2,3,4,5} {
    \draw[thin] (0,\j) -- (4,\j);
}

\draw[thin] (0,0) -- (4,6);



\node at (0.5,-0.6) {$0$};

\node at (1.5,-0.6) {$6$};
\node at (1.5,0.4) {$2$};

\node at (2.5,-0.6) {$13$};
\node at (2.5,0.4) {$9$};
\node at (2.5,1.4) {$5$};
\node at (2.5,2.4) {$1$};

\node at (3.5,-0.6) {$19$};
\node at (3.5,0.4) {$15$};
\node at (3.5,1.4) {$11$};
\node at (3.5,2.4) {$7$};
\node at (3.5,3.4) {$3$};

\node[left] at (0,0) {$(0,0)$};
\node[right] at (4,6) {$(4,6)$};


\end{tikzpicture}

\&

\begin{tikzpicture}[scale=0.5]

\draw[thin] (0,0) rectangle (4,6);

\foreach \i in {1,2,3} {
    \draw[thin] (\i,0) -- (\i,6);
}

\foreach \j in {1,2,3,4,5} {
    \draw[thin] (0,\j) -- (4,\j);
}

\draw[thin] (0,0) -- (4,6);

\draw[red, line width=0.5mm,->] (0,0) -- (3,0) -- (3,4) -- (4,4) -- (4,6);

\fill[blue!10, opacity=0.5] (3,0) -- (3,4) -- (4,4) -- (4,0) -- (3,0);

\node at (6.5,2) {{\color{blue}$D_5 = \left(\begin{array}{cc} 0 & 0\\ 0 & 4 \end{array}\right)$}};

\node[blue] at (0.5,-0.6) {$0$};

\node[blue] at (1.5,-0.6) {$7$};
\node at (1.5,0.4) {$3$};

\node[blue] at (2.5,-0.6) {$13$};
\node at (2.5,0.4) {$9$};
\node at (2.5,1.4) {$5$};
\node at (2.5,2.4) {$1$};

\node[blue] at (3.5,-0.6) {$18$};
\node[blue] at (3.5,0.4) {$14$};
\node[blue] at (3.5,1.4) {$10$};
\node[blue] at (3.5,2.4) {$6$};
\node[blue] at (3.5,3.4) {$2$};

\node[left] at (0,0) {$(0,0)$};
\node[right] at (4,6) {$(4,6)$};


\end{tikzpicture}

\end{tikzcd}
\end{center}
\vspace{-0.2in}
\caption{The enhanced rank for $D_k$, $1\leq k\leq 23$, with $(n,m,d) = (2,3,2)$. Left figure: $k\neq 5$; Right figure: $k=5$.}
\label{fig:enhanced_ranks_for_4,6-Dyck_paths}
\end{figure}
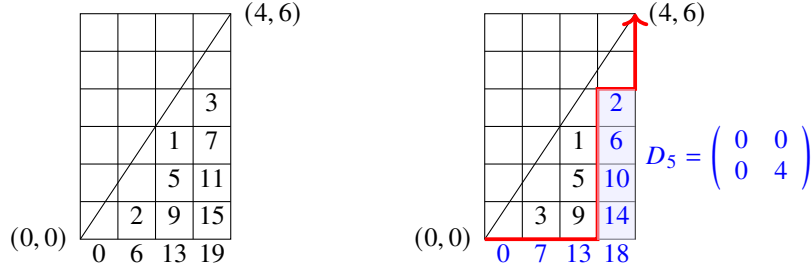

For each $D =  (y_{in+j} = y_{i,j})_{0\leq i\leq d-1,0\leq j\leq n-1}\in \Young(dn,dm)$, we use Remark \ref{rem:a_simplification_of_the_combinatorial_two-variable_CDP_conjecture}. In particular, we identify $\overline{D}$ with $\hat{\rank}_D(\overline{D})$ via $\hat{\rank}_{D}$. 
By a direct computation, we obtain ($L_{D_i} = L_{D_i}(q,t)$):

{\tiny
\begin{minipage}[t]{0.5\textwidth}
\centering
\begin{align*}
L_{D_1} =&~1 - t + q^3t^4(1-t) + q^4t^6(1-t) + q^5t^8(1-t)\\
                &~+ q^6t^{10}(1-t) + q^7t^{12}(1-t^3) + q^8t^{16};\\
L_{D_2} =&~1 - t^2 + q^3t^5(1-t) + q^4t^7(1-t)\\
&~+ q^5t^9(1-t) + q^6t^{11}(1-t) + q^7t^{13};\\
L_{D_3} =&~ t^2(1-t) + q^3t^6(1-t) + q^4t^8(1-t)\\
                       &~+ q^5t^{10}(1-t) + q^6t^{12};\\
L_{D_4} =&~qt^3(1-t) + q^4t^7(1-t) + q^5t^9(1-t^3) + q^6t^{13};\\ 
L_{D_5} =&~ q^2t^4(1-t) + q^3t^6(1-t) + q^4t^8(1-t)\\
&~+ q^5t^{10}(1-t^3) + q^6t^{14};\\
L_{D_6} =&~qt^2(1-t) + q^4t^6(1-t) + q^5t^8(1-t)\\
&~+ q^6t^{10}(1-t^3) + q^7t^{14};\\
L_{D_7} =&~t^3(1-t) + q^3t^7(1-t) + q^4t^9(1-t) + q^5t^{11};\\
L_{D_8} =&~t^4(1-t) + q^3t^8(1-t) + q^4t^{10};\\
L_{D_9} =&~qt^5(1-t) + q^3t^8(1-t^2) + q^4t^{11};\\
L_{D_{10}} =&~qt^4(1-t) + q^4t^8(1-t^3) + q^5t^{12};\\
L_{D_{11}} =&~t^5(1-t) + q^3t^9;\\
\end{align*}
\end{minipage}%
\vrule depth 80ex
\begin{minipage}[t]{0.5\textwidth}
\centering
\begin{align*}
L_{D_{12}} =&~ t^6(1-t) + q^2t^9;\\
L_{D_{13}} =&~q^2t^6(1-t^2) + q^4t^{10};\\
L_{D_{14}} =&~qt^7(1-t^2) + q^2t^{10};\\
L_{D_{15}} =&~q^2t^3(1-t) + q^3t^5(1-t) + q^4t^7(1-t)\\
&~+ q^5t^9(1-t) + q^6t^{11}(1-t^3) + q^7t^{15};\\
L_{D_{16}} =&~q^2t^4(1-t) + q^3t^6(1-t) + q^4t^8(1-t)\\
&~+ q^5t^{10}(1-t) + q^6t^{12};\\
L_{D_{17}} =&~q^2t^5(1-t) + q^3t^7(1-t) + q^4t^9(1-t)\\
&~+ q^5t^{11};\\
L_{D_{18}} =&~q^2t^6(1-t) + q^3t^8(1-t^3) + q^4t^{12};\\
L_{D_{19}} =&~q^2t^5(1-t) + q^3t^7(1-t) + q^4t^9(1-t^3)\\
&~+ q^5t^{13};\\
L_{D_{20}} =&~qt^6(1-t) + q^2t^8(1-t) + q^3t^{10};\\
L_{D_{21}} =&~t^7(1-t) + qt^9;\\
L_{D_{22}} =&~q^2t^7(1-t^3) + q^3t^{11};\\
L_{D_{23}} =&~t^8.\\
\end{align*}
\end{minipage}
}

On the other hand, we have ($\mH_{D_i}^{\mot} = \mH_{D_i}^{\mot}(q,t)$):
{\small\[
\left.\begin{array}{llllll}
\mH_{D_1}^{\mot} = q^8t^{16}; & \mH_{D_2}^{\mot} = q^7t^{14}; & \mH_{D_3}^{\mot} = q^6t^{12}; & \mH_{D_4}^{\mot} = q^6t^{11}; & \mH_{D_5}^{\mot} = q^6t^{10}; & \mH_{D_6}^{\mot} = q^7t^{13};\\
\mH_{D_7}^{\mot} = q^5t^{10}; & \mH_{D_8}^{\mot} = q^4t^8; & \mH_{D_9}^{\mot} = q^4t^7; & \mH_{D_{10}}^{\mot} = q^5t^9; & \mH_{D_{11}}^{\mot} = q^3t^6; & \mH_{D_{12}}^{\mot} = q^2t^4;\\
\mH_{D_{13}}^{\mot} = q^4t^6; & \mH_{D_{14}}^{\mot} = q^2t^3; &  \mH_{D_{15}}^{\mot} = q^7t^{12}; & \mH_{D_{16}}^{\mot} = q^6t^{10}; & \mH_{D_{17}}^{\mot} = q^5t^8; & \mH_{D_{18}}^{\mot} = q^4t^6;\\
\mH_{D_{19}}^{\mot} = q^5t^8; & \mH_{D_{20}}^{\mot} = q^3t^5; & \mH_{D_{21}}^{\mot} = qt^2; & \mH_{D_{22}}^{\mot} = q^3t^4; & \mH_{D_{23}}^{\mot} = 1. &
\end{array}\right.
\]}
Summing up, we obtain
\begin{eqnarray*}
L(q,t) = \mH^{\mot}(q,t) &=& 1 + qt^2 + q^2t^3(1+t) + q^3t^4(1+t+t^2) + q^4t^6(2+t+t^2)\\
&&+ q^5t^8(2+t+t^2) + q^6t^{10}(2+t+t^2) + q^7t^{12}(1+t+t^2) + q^8t^{16}.
\end{eqnarray*} 
This verifies Conjecture \ref{conj:combinatorial_two-variable_CDP_conjecture_for_generic_curve_singularity} for this example.
\end{example}

\begin{remark}
With the help of a computer program (e.g., SageMath), one can verify Conjecture \ref{conj:combinatorial_two-variable_CDP_conjecture_for_generic_curve_singularity} for many examples. 
\end{remark}

\section{A new take on the combinatorics}\label{sec:a_new_take_on_the_combinatorics}

In this section we establish Theorem~\ref{thm:admissible_invariant_subsets_vs_Dyck_paths}, which in particular yields a new proof of part~(2) of that theorem, originally obtained in \cite[Thm.~4.13]{GMO25} and \cite{GMV20}. 

As in Remark~\ref{rem:finite_determinancy_of_planar_curve_singularities}, we may assume that $(C,0)$ is a \emph{generic} irreducible planar curve singularity parametrized by
\[
x(t) = t^{dn}, 
\qquad 
y(t) = t^{dm} + \lambda t^{dm+1} + \sum_{\substack{j \geq 2 \\ dm+j \notin \Gamma}} \lambda_j t^{dm+j},
\]
where $(n,m) = 1$, $m > n$, $\lambda \in \field^{\times}$, and $\lambda_j \in \field$.

\subsection{From admissible $(dn,dm)$-invariant subsets to $(dn,dm)$-Dyck paths}

\begin{definition}
Let $(c_{in+j})_{0 \leq i \leq d-1,\, 0 \leq j \leq n-1}$ be a $d \times n$ matrix of nonnegative integers such that each row $(c_{i,j})_{0 \leq j \leq n-1}$ is non-decreasing. We identify $\mathbb{Z}/n$ with $\{0,1,\dots,n-1\}$ whenever convenient. For $0 \leq i \leq d-1$ and $j \in \mathbb{Z}/n$, set
\[
\rank_{i,j}((c_{\bullet,\bullet})) =  \rank_{in+j}((c_{\bullet,\bullet})) := \rank(in+j,c_{i,j}) = m(in+j) - nc_{i,j}.
\]
\end{definition}

\begin{remark}
Fix $0 \leq u < v \leq d-1$. Then for $0 \leq j \leq n-2$,
\[
\rank_{v,j+1}((c_{\bullet,\bullet})) \leq \rank_{u,j}((c_{\bullet,\bullet})) + m 
\quad \Longleftrightarrow \quad 
c_{v,j+1} \geq c_{u,j} + (v-u)m.
\]
For $j = n-1$,
\[
\rank_{v,0}((c_{\bullet,\bullet})) \leq \rank_{u,n-1}((c_{\bullet,\bullet})) + m 
\quad \Longleftrightarrow \quad 
c_{v,0} \geq c_{u,n-1} + (v-u-1)m.
\]
\end{remark}

Now fix $\Delta = (c_{i,j}) \in \Adm(dn,dm)$. Define
\[
\rank_{\Delta,in+j} = \rank_{\Delta,i,j} := m(in+j) - nc_{i,j},\qquad 0\leq i\leq d-1, 0 \leq j \leq n-1. 
\]
\textbf{Note.} Under the identification $\mathbb{Z}/n \isomorphic \{0,1,\cdots,n-1\}$, we have $ \rank_{\Delta,i,n} = \rank_{\Delta,i,0}$ but $\rank_{\Delta,in+n} = \rank_{\Delta,i+1,0}$.
For brevity, we write
\begin{equation}
(i,j) := in+j,\qquad 0\leq i\leq d-1, 0\leq j \leq n-1,
\end{equation}
so that in particular $(i,n) = (i+1,0)$ for $i<d-1$.

We now introduce the following \textbf{constructions}.

\begin{enumerate}[wide,labelwidth=!,labelindent=0pt,itemindent=!]
\item
\emph{Move indices and rank vectors.}  
Set $\rank_{\Delta,in+j}^{[0]} = \rank_{\Delta,i,j}^{[0]} := \rank_{\Delta,i,j}$.
Inductively, for $1\leq u \leq d-1$, define the {\color{blue}move index} $\tilde{\moveindex}_u$ and {\color{blue}rank vector}
$(\rank_{\Delta,in+j}^{[u]} = \rank_{\Delta,i,j}^{[u]})_{0\leq i\leq d-1,0\leq j\leq n-1}$ from $(\rank_{\Delta,in+j}^{[u-1]} = \rank_{\Delta,i,j}^{[u-1]})_{0\leq i\leq d-1,0\leq j\leq n-1}$ by:
\footnote{Let's emphasize again that $\rank_{\Delta,u,n}^{[u-1]} := \rank_{\Delta,u,0}^{[u-1]}$.}
\begin{equation}
\tilde{\moveindex}_u := \max\{in+j+1:~0\leq i\leq u-1,0\leq j\leq n-1, \rank_{\Delta,u,j+1}^{[u-1]} \leq \rank_{\Delta,i,j}^{[u-1]} + m\},
\end{equation}
which depends only on $(\rank_{\Delta,i,j}^{[u-1]})_{0\leq i\leq u,0\leq j\leq n-1})$.
Then, 
\begin{equation}
\rank_{\Delta,x=in+j}^{[u]}:= \left\{\begin{array}{ll}
\rank_{\Delta,x}^{[u-1]} = \rank_{\Delta,i,j}^{[u-1]} & 0\leq x < \tilde{\moveindex}_u \text{ or } (u+1)n \leq x < dn;\\
\rank_{\Delta,x-n}^{[u-1]} = \rank_{\Delta,i-1,j}^{[u-1]} & \tilde{\moveindex}_u + n \leq x < (u+1)n;\\
\rank_{\Delta,un+j}^{[u-1]} = \rank_{\Delta,u,j}^{[u-1]} & \tilde{\moveindex}_u \leq x < \tilde{\moveindex}_u + n.
\end{array}\right. 
\end{equation}
Equivalently, say, $\tilde{\moveindex}_u = kn+\ell$, $0\leq k\leq u$, $0\leq \ell \leq n-1$, then
\begin{eqnarray*}
(\rank_{\Delta,0}^{[u]},\text{\tiny$\cdots$},\rank_{\Delta,dn-1}^{[u]})
&:=& (\rank_{\Delta,0}^{[u-1]},\text{\tiny$\cdots$},\rank_{\Delta,\tilde{\moveindex}_u-1}^{[u-1]},\rank_{\Delta,un+\ell}^{[u-1]},\text{\tiny$\cdots$},\rank_{\Delta,un+n-1}^{[u-1]},\\
&&\rank_{\Delta,un}^{[u-1]},\text{\tiny$\cdots$},\rank_{\Delta,un+\ell-1}^{[u-1]},\rank_{\Delta,\tilde{\moveindex}_u}^{[u-1]},\text{\tiny$\cdots$},\rank_{\Delta,un-1}^{[u-1]},\rank_{\Delta,(u+1)n}^{[u-1]},\text{\tiny$\cdots$},\rank_{\Delta,dn-1}^{[u-1]}).
\end{eqnarray*}

\item
\emph{Permutations.}  
For $0\leq j\leq n-1$ and $1\leq u\leq d-1$, define a {\color{blue}permutation} $\tilde{\permutation}_{j,u}$ of $[0,d-1]$\footnote{$[0,k] := \{0,1,\cdots,k\}$.} by
\begin{equation}
\tilde{\permutationindex}_{j,u} \coloneqq \min\{0\leq i\leq u:~in+j\geq \tilde{\moveindex}_u\}~~(\text{\color{blue}permutation index});\quad \tilde{\permutation}_{j,u} \coloneqq (\tilde{\permutationindex}_{j,u}\ \cdots\ u).
\end{equation}
We also set $\tilde{\permutationindex}_{n,u} \coloneqq \tilde{\permutationindex}_{0,u}$ and $\tilde{\permutation}_{n,u} \coloneqq \tilde{\permutation}_{0,u}$.
\textbf{Observe} that
\begin{equation}
\rank_{\Delta,i,j}^{[u]} = \rank_{\Delta,\tilde{\permutation}_{j,u}^{-1}(i),j}^{[u-1]}.
\end{equation}
For $0\leq j\leq n-1$, define the composition
\begin{equation}
\tilde{\permutation}_j := \tilde{\permutation}_{j,d-1}\circ\cdots\circ\tilde{\permutation}_{j,1}.
\end{equation}

\item
\emph{Updated $c$-matrices.}  
For $0\leq u\leq d-1$, $0\leq i\leq d-1$, and $0\leq j\leq n-1$, define
\begin{equation}
c_{i,j}^{[u]}:= c_{\tilde{\permutation}_{j,1}^{-1}\circ\cdots\circ\tilde{\permutation}_{j,u}^{-1}(i),j} - (\tilde{\permutation}_{j,1}^{-1}\circ\cdots\circ\tilde{\permutation}_{j,u}^{-1}(i) - i)m.
\end{equation}
In particular, $c_{i,j}^{[u]} = c_{\tilde{\permutation}_{j,u}^{-1}(i),j}^{[u-1]} - (\tilde{\permutation}_{j,u}^{-1}(i) - i)m$.
Then,
\begin{equation}
\rank_{\Delta,i,j}^{[u]} = \rank_{i,j}((c_{\bullet,\bullet}^{[u]})) = m(in+j) - nc_{i,j}^{[u]}.
\end{equation}
Finally, \textbf{define}
\begin{equation}
\Phi_d((c_{i,j})) := (c_{i,j}^{[d-1]}).
\end{equation}

\item
\emph{Enhanced ranks.}  
For each box $(x,y)$ (top-left vertex) with $0\leq x\leq dn-1$,
write $x = in+j$, $0\leq i\leq d-1$, $0\leq j\leq n-1$. For $0\leq u\leq d-1$, define
\begin{equation}
\hat{\rank}_{\Delta}^{[u]}(x,y) := d\cdot \rank(x,y) + \tilde{\permutation}_{j,1}^{-1}\circ\cdots\circ \tilde{\permutation}_{j,u}^{-1}(i).
\end{equation}
In particular, 
\begin{equation}
\hat{\rank}_{\Delta}(x,y) := \hat{\rank}_{\Delta}^{[0]}(x,y) = d\cdot \rank(x,y)  + i.
\end{equation}
Clearly, $\hat{\rank}_{\Delta}^{[u]}(x,y+1) = \hat{\rank}_{\Delta}^{[u]}(x,y) - dn$.
\end{enumerate}

\begin{proposition}\label{prop:from_admissible_invariant_subsets_to_Dyck_paths}
For every $\Delta \in \Adm(dn,dm)$:
\begin{enumerate}[wide,labelwidth=!,labelindent=0pt,itemindent=!,label=(\alph*)]
\item The move indices $\tilde{\moveindex}_u$ are well-defined and satisfy
\[
0 < \tilde{\moveindex}_1 < \tilde{\moveindex}_2 < \cdots < \tilde{\moveindex}_{d-1} \leq (d-1)n.
\]
\item The map
\[
\Phi_d: \Adm(dn,dm) \longrightarrow \Young(dn,dm), \qquad 
(c_{i,j}) \mapsto (c_{i,j}^{[d-1]}),
\]
is well-defined.
\end{enumerate}
\end{proposition}

\begin{proof}
Proof proceeds by induction on $d$.
The case $d = 1$ is trivial. Suppose that the result holds for `$<d$', $d\geq 2$. For the induction, consider the case `$d$'.
Fix any $\Delta \isomorphic (c_{i,j})_{0\leq i\leq d-1,0\leq j\leq n-1} \in \Adm(dn,dm)$. By inductive hypothesis, we already have: 
$\tilde{\moveindex}_u$, $1\leq u < d-1$, are well-defined, and $0 < \tilde{\moveindex}_1 < \cdots < \tilde{\moveindex}_{d-2}\leq (d-2)n$; $(\rank_{\Delta,i,j}^{[u]})$ and
$(c_{i,j}^{[u]})$, $0\leq u < d-1$, are well-defined; $(c_{i,j}^{[d-2]})_{0\leq i\leq d-2,0\leq j\leq n-1} \in \Young((d-1)n,(d-1)m)$.

Observe that by definition, we have: $\rank_{\Delta,i,j}^{[u]} = \rank_{\Delta,i,j}$, $c_{i,j}^{[u]} = c_{i,j}$, $\forall 0\leq u < d-2$, $i\geq u+1$. Moreover, 
$\rank_{\Delta,i,j}^{[d-2]} = \rank_{\Delta,\tilde{\permutation}_{j,d-2}^{-1}(i),j}^{[d-3]}$, and $c_{i,j}^{[d-2]} = c_{\tilde{\permutation}_{j,d-2}^{-1}(i),j}^{[d-3]} - (\tilde{\permutation}_{j,d-2}^{-1}(i) - i)m$. 
Now, as $(c_{i,j})_{0\leq i\leq d-1,0\leq j\leq n-1} \in \Adm(dn,dm)$, by Lemma \ref{lem:semimodule_vs_c_ij}, $\exists 0\leq j_0\leq n-1$ such that 
\[
\rank_{\Delta,d-1,j_0+1}^{[d-2]} = \rank_{\Delta,d-1,j_0+1} \leq \rank_{\Delta,d-2,j_0} + m = \rank_{\Delta,d-2,j_0}^{[d-3]} + m = \rank_{\Delta,\tilde{\permutation}_{j_0,d-2}(d-2),j_0}^{[d-2]} + m.
\]
Thus, by definition, $\tilde{\moveindex}_{d-1}$ is well-defined, and 
\[
(d-1)n \geq \tilde{\moveindex}_{d-1} \geq \tilde{\permutation}_{j_0,d-2}(d-2)\cdot n + j_0 + 1 = \tilde{\permutationindex}_{j_0,d-2}\cdot n + j_0 + 1 \geq \tilde{\moveindex}_{d-2} + 1.
\]

It remains to show that $(c_{i,j}^{[d-1]})_{0\leq i\leq d-1,0\leq j\leq n-1} \in \Young(dn,dm)$. By Corollary \ref{cor:Dyck_path_via_rank}, this is equivalent to: $(a)$. $\rank_{\Delta,x+1}^{[d-1]} \leq \rank_{\Delta,x}^{[d-1]} + m$, $\forall 0\leq x < dn - 1$; $(b)$. $0 \leq \rank_{\Delta,x}^{[d-1]} \leq mx$, $\forall 0\leq x\leq dn-1$.

Indeed, by inductive hypothesis, we already have $(c_{i,j}^{[d-2]})_{0\leq i\leq d-2,0\leq j\leq n-1} \in \Young((d-1)n,(d-1)m)$. So, $\rank_{\Delta,x+1}^{[d-2]}\leq \rank_{\Delta,x}^{[d-2]} + m$, $\forall 0\leq x < (d-1)n-1$, and $0 \leq \rank_{\Delta,x}^{[d-2]} \leq mx$, $\forall 0\leq x\leq (d-1)n-1$.

Firstly, we prove $(a)$. Say, $\tilde{\moveindex}_{d-1} = kn + \ell$, $0\leq k\leq d-1$, $0\leq \ell \leq n-1$.
\begin{enumerate}[wide,labelwidth=!,labelindent=0pt,itemindent=!]
\item
If $0\leq x < \tilde{\moveindex}_{d-1} - 1$, then by definition, we have
\[
\rank_{\Delta,x+1}^{[d-1]} = \rank_{\Delta,x+1}^{[d-2]} \leq \rank_{\Delta,x}^{[d-2]} + m = \rank_{\Delta,x}^{[d-1]} + m.
\]

\item
If $\tilde{\moveindex}_{d-1} + n \leq x < dn - 1$, then by definition, we have
\[
\rank_{\Delta,x+1}^{[d-1]} = \rank_{\Delta,x+1-n}^{[d-2]} \leq \rank_{\Delta,x-n}^{[d-2]} + m = \rank_{\Delta,x}^{[d-1]} + m.
\]

\item
If $\tilde{\moveindex}_{d-1} \leq x = kn + j < kn + n-1$ or $(k+1)n \leq x = (k+1)n + j < \tilde{\moveindex}_{d-1} + n - 1$, then by definition,
\[
\rank_{\Delta,x+1}^{[d-1]} = \rank_{\Delta,(d-1)n+j+1}^{[d-2]} \leq \rank_{\Delta,(d-1)n+j}^{[d-2]} + m = \rank_{\Delta,x}^{[d-1]} + m.
\]

\item
If $x = \tilde{\moveindex}_{d-1} - 1$, say, $x = in+j$, $0\leq i\leq d-2$, $0\leq j\leq n-1$, then by definition of $\tilde{\moveindex}_{d-1} = in + j + 1$,
\[
\rank_{\Delta,x+1}^{[d-1]} = \rank_{\Delta,\tilde{\moveindex}_{d-1}}^{[d-1]} = \rank_{\Delta,d-1,j+1}^{[d-2]} \leq \rank_{i,j}^{[d-2]} + m = \rank_{\Delta,x}^{[d-1]} + m.
\]

\item
If $x = kn + n-1$, then by definition,
\[
\rank_{\Delta,x+1}^{[d-1]} = \rank_{\Delta,k+1,0}^{[d-1]} = \rank_{\Delta,d-1,0} \leq \rank_{\Delta,d-1,n-1} + m = \rank_{\Delta,k,n-1}^{[d-1]} + m = \rank_{\Delta,x}^{[d-1]} + m.
\]
Here, the inequality follows from Lemma \ref{lem:semimodule_vs_c_ij}.$(3)$.

\item
It remains to consider the case that $x = (k+1)n + \ell - 1$. If $\ell = 0$, then $x = kn + n-1$, which is just $(5)$. 
Thus, it suffices to consider the case that $\ell \geq 1$. Then by the definition of $\tilde{\moveindex}_{d-1}$, we have 
\[
\rank_{\Delta,x+1}^{[d-1]} = \rank_{\Delta,k,\ell}^{[d-2]} < \rank_{\Delta,d-1,\ell+1}^{[d-2]} - m = \rank_{\Delta,d-1,\ell+1} - m \leq \rank_{\Delta,d-1,\ell}
\leq \rank_{\Delta,d-1,\ell-1} + m = \rank_{\Delta,x}^{[d-1]} + m.
\]
Here, the last two inequalities follow from Lemma \ref{lem:semimodule_vs_c_ij}.$(1),(3)$. This completes the proof of $(a)$.
\end{enumerate}

Finally, it remains to prove $(b)$. Indeed, by definition, we have $\rank_{\Delta,i,j}^{[d-1]} = \rank_{\Delta,\tilde{\permutation}_j^{-1}(i),j} \geq 0$. On the other hand, for any $0\leq x\leq dn-1$, by $(a)$, we have $\rank_{\Delta,x}^{[d-1]} \leq \rank_{\Delta,x-1}^{[d-1]} + m \leq \cdots \leq \rank_{\Delta,0}^{[d-1]} + mx$. 
It suffices to show that $\rank_{\Delta,0}^{[d-1]} = 0$. 
In fact, as $0 < \tilde{\moveindex}_u$, $\forall 1\leq u\leq d-1$, it follows by definition that, $\tilde{\permutation}_{0,u}(0) = 0$. Then, $\tilde{\permutation}_0(0) = 0$, and hence $\rank_{\Delta,0}^{[d-1]} = \rank_{\Delta,\tilde{\permutation}_0^{-1}(0),0} = \rank_{\Delta,0,0} = 0$. Done.
\end{proof}

\subsection{From $(dn,dm)$-Dyck paths to admissible $(dn,dm)$-invariant subsets}

Fix any 
\[
D = (y_{in+j} = y_{i,j})_{0 \leq i \leq d-1,\,0 \leq j \leq n-1} \in \Young(dn,dm).
\]
Recall the notation
\[
\rank_{in+j} = \rank_{i,j} = \rank_{D,in+j} = \rank_{D,i,j} := \rank(in+j,y_{i,j}) = m(in+j) - ny_{i,j}.
\]
In Section \ref{subsubsec:Dyck_paths_vs_admissible_invariant_subsets}, we defined: 
\[
\moveindex_u,\ \vec{\rank}^{[d-u]} = (\rank_x^{[d-u]})_{0 \leq x \leq dn-1},\ 
\permutationindex_{j,u},\ \permutation_{j,u},\ \permutation_j,
\]
for all $d-1 \geq u \geq 1$, $0 \leq j \leq d-1$, together with
\[
(y_{in+j}^{[u]} = y_{i,j}^{[u]})_{0 \leq i \leq d-1,\,0 \leq j \leq n-1},\quad 
\hat{\rank}_D^{[u]}, \quad \forall\, 0 \leq u \leq d-1.
\]

We now make the following \textbf{construction}, extending $\vec{\rank}^{[d-u]}$:

\begin{enumerate}[wide,labelwidth=!,labelindent=0pt,itemindent=!,label=(\arabic*')]
\item
Set $\rank_{D,in+j}^{[0]} = \rank_{D,i,j}^{[0]} := \rank_{D,in+j} = \rank_{D,i,j}$. 
Inductively, for each $d-1\leq u\leq 1$, \textbf{define} 
\[
\bigl(\rank_{D,in+j}^{[d-u]} = \rank_{D,i,j}^{[d-u]}\bigr)_{0 \leq i \leq d-1,\,0 \leq j \leq n-1}
\] 
from the previous step as follows: if $\moveindex_u = kn+\ell$ with $0 \leq k \leq u$, $0 \leq \ell \leq n-1$, then
\[
\rank_{D,x}^{[d-u]} :=
\begin{cases}
\rank_{D,x}^{[d-u-1]}, & 0 \leq x < \moveindex_u \ \text{or}\ (u+1)n \leq x \leq dn-1, \\[4pt]
\rank_{D,x+n}^{[d-u-1]}, & \moveindex_u \leq x \leq un-1, \\[4pt]
\rank_{D,(k+1)n+j}^{[d-u-1]}, & un \leq x = un+j < un+\ell, \\[4pt]
\rank_{D,kn+j}^{[d-u-1]}, & un+\ell \leq x = un+j < (u+1)n.
\end{cases}
\]
Equivalently, in vector notation,
\begin{eqnarray*}
(\rank_{D,0}^{[d-u]},\text{\tiny$\cdots$},\rank_{D,dn-1}^{[d-u]})
&:=& (\rank_{D,0}^{[d-u-1]},\text{\tiny$\cdots$},\rank_{D,\moveindex_u-1}^{[d-u-1]},\rank_{D,\moveindex_u+n}^{[d-u-1]},\text{\tiny$\cdots$},\rank_{D,un+n-1}^{[d-u-1]},\rank_{D,(k+1)n}^{[d-u-1]},\text{\tiny$\cdots$},\\
&&\rank_{D,(k+1)n+\ell-1}^{[d-u-1]},\rank_{D,\moveindex_u}^{[d-u-1]},\text{\tiny$\cdots$},\rank_{D,kn+n-1}^{[d-u-1]},\rank_{D,(u+1)n}^{[d-u-1]},\text{\tiny$\cdots$},\rank_{D,dn-1}^{[d-u-1]}).
\end{eqnarray*}
By definition, $\rank_x^{[d-u]} = \rank_{D,x}^{[d-u]}$ for all $0 \leq x \leq un$.

\item
By definition of $\permutation_{j,u}$,
\begin{equation}
\rank_{D,i,j}^{[d-u]} = \rank_{D,\permutation_{j,u}(i),j}^{[d-u-1]}.
\end{equation}

\item
From the definition of $y_{i,j}^{[d-u]}$, we have $y_{i,j}^{[d-u]} = y_{\permutation_{j,u}(i),j}^{[d-u-1]}$, and hence
\[
\rank_{D,i,j}^{[d-u]} = \rank_{i,j}\bigl((y_{\bullet,\bullet})^{[d-u]}\bigr) 
= \rank(in+j, y_{i,j}^{[d-u]}) 
= m(in+j) - ny_{i,j}^{[d-u]}.
\]
\end{enumerate}

\begin{proposition}\label{prop:from_Dyck_paths_to_admissible_invariant_subsets}
We have:
\begin{enumerate}[wide,labelwidth=!,labelindent=0pt,itemindent=!,label=(\alph*)]
\item
For every $D \cong (y_{i,j}) \in \Young(dn,dm)$, the move indices 
$\moveindex_u$, $1 \leq u \leq d-1$, are well-defined, and
\[
0 < \moveindex_1 < \moveindex_2 < \cdots < \moveindex_{d-1} \leq (d-1)n.
\]

\item
The map
\begin{equation*}
\Psi_d: \Young(dn,dm) \rightarrow \Adm(dn,dm): (y_{i,j})\mapsto (c_{i,j}):= \Psi_d((y_{i,j})),
\end{equation*}
is well-defined.

\item
For each $1\leq u\leq d-1$ and $\moveindex_u \leq x < un$, we have
\begin{equation*}
\rank_{D,x}^{[d-u-1]} > \rank_{D,x+n}^{[d-u-1]}.
\end{equation*}
\end{enumerate}
\end{proposition}

\begin{proof}
We prove by induction on $d$. The case $d = 1$ is trivial. Suppose that the result holds for `$<d$', $d\geq 2$. For the induction, consider the case `$d$'.
Fix any $D \isomorphic (y_{in+j} = y_{i,j})_{0\leq i\leq d-1,0\leq j\leq n-1} \in \Young(dn,dm)$. By Corollary \ref{cor:Dyck_path_via_rank}, we have: $y_{x+1} \leq y_x + m$, $\forall 0\leq x < dn-1$; $0\leq y_x \leq mx$, $\forall 0 \leq x\leq dn-1$.

\begin{enumerate}[wide,labelwidth=!,labelindent=0pt,itemindent=!]
\item
First, we show that $\moveindex_{d-1}$ is well-defined and $0 < \moveindex_{d-1} \leq (d-1)n$.
Recall that 
\[
\moveindex_{d-1}:= \max\{0\leq u\leq (d-1)n:~\rank_{D,x}^{[0]} \leq \rank_{D,x+n-1}^{[0]} + m~~(\Leftrightarrow~~y_{x+n-1} \leq y_x + m)\}.
\]
If we already have $\rank_{D,(d-1)n}^{[0]} \leq \rank_{D,dn-1}^{[0]} + m$, then $0 < \moveindex_{d-1} = (d-1)n \leq (d-1)n$. 
Else, there exists a minimal $0\leq x_0\leq (d-1)n$ such that $\rank_{D,x}^{[0]} \geq \rank_{D,x+n-1}^{[0]} + m + 1$, $\forall x_0 \leq x\leq (d-1)n$.
Fix any $x_0 \leq x\leq (d-1)n$. Say, $k_x\in\mathbb{N}$ is maximal such that $x + k(n-1) \leq dn - 1$. That is, $k_x = \floor{\frac{dn-1-x}{n-1}}$. 
Then
\[
mx \geq \rank_{D,x}^{[0]} \geq  \rank_{D,x+n-1}^{[0]} + m + 1 \geq \cdots \geq \rank_{D,x+k_x(n-1)}^{[0]} + k_x(m+1) \geq \floor{\frac{dn-1-x}{n-1}}(m+1).
\]
Observe that for $x = d-1$, we have $mx = (d-1)m < \floor{\frac{dn-1-x}{n-1}}(m+1) = d(m+1)$. It follows that $x_0\geq d$. Thus, by definition, $0< d-1 \leq \moveindex_{d-1} = x_0-1 \leq (d-1)n - 1$. Done.

Now, $(\rank_{D,in+j}^{[1]} = \rank_{D,i,j}^{[1]})_{0\leq i\leq d-1,0\leq j\leq n-1}$ and $(y_{in+j}^{[1]} = y_{i,j}^{[1]})_{0\leq i\leq d-1,0\leq j\leq n-1}$ are defined, and 
$\rank_{D,x}^{[1]} = \rank_x((y_{\bullet,\bullet}^{[1]})) = mx - ny_x^{[1]}$. Say, $\moveindex_{d-1} = kn + \ell$, $0\leq k\leq d-1$, $0\leq \ell \leq n-1$.

\item
We show that $y_{dn-1}^{[1]} \leq y_{(d-1)n}^{[1]} + m$ ($\Leftrightarrow$ $\rank_{D,(d-1)n}^{[1]} \leq \rank_{D,dn-1}^{[1]} + m$).
Indeed, if $\ell = 0$, then
\[
\rank_{D,(d-1)n}^{[1]} = \rank_{D,\moveindex_{d-1}}^{[0]} \leq \rank_{D,\moveindex_{d-1}+n-1}^{[0]} + m = \rank_{D,dn-1}^{[1]} + m.
\]
Here, the inequality follows from the definition of $\moveindex_{d-1}$. 
If $\ell\geq 1$, then
\[
\rank_{D,(d-1)n}^{[1]} = \rank_{D,(k+1)n} \leq \rank_{D,kn+n-1} + m = \rank_{D,dn-1}^{[1]} + m.
\]

\item
We show that $\rank_{D,x+1}^{[1]} \leq \rank_{D,x}^{[1]} + m$, $\forall 0\leq x < (d-1)n-1$ or $(d-1)n \leq x < dn-1$.
\begin{enumerate}[wide,labelwidth=!,labelindent=0pt,itemindent=!,label=(3.\arabic*)]
\item
If $0\leq x < \moveindex_{d-1} - 1$, then
\[
\rank_{D,x+1}^{[1]}  = \rank_{D,x+1} \leq \rank_{D,x} + m = \rank_{D,x}^{[1]} + m.
\]

\item
If $\moveindex_{d-1}\leq x < (d-1)n - 1$, then
\[
\rank_{D,x+1}^{[1]}  = \rank_{D,x+1+n} \leq \rank_{D,x+n} + m = \rank_{D,x}^{[1]} + m.
\]

\item
If $(d-1)n \leq x = (d-1)n + j < (d-1)n + \ell - 1$, then
\[
\rank_{D,x+1}^{[1]}  = \rank_{D,(k+1)n+j+1} \leq \rank_{D,(k+1)n+j} + m = \rank_{D,x}^{[1]} + m.
\]

\item
If $x = (d-1)n + \ell - 1\geq (d-1)n$, then
\[
\rank_{D,x+1}^{[1]}  = \rank_{D,kn+\ell = \moveindex_{d-1}} \leq \rank_{D,\moveindex_{d-1}+n-1} + m = \rank_{D,x}^{[1]} + m.
\]
Here, the inequality follows from the definition of $\moveindex_{d-1}$.

\item
If $(d-1)n + \ell \leq x = (d-1)n + j < dn - 1$, then
\[
\rank_{D,x+1}^{[1]}  = \rank_{D,kn+j+1} \leq \rank_{D,kn+j} + m = \rank_{D,x}^{[1]} + m.
\]
\end{enumerate}

\item
We show that $0\leq \rank_{D,x}^{[1]} \leq mx$, $\forall 0 \leq x\leq dn-1$.
For any $0\leq i\leq d-1$, $0\leq j\leq n-1$, by definition, we have $\rank_{D,x=in+j}^{[1]} =\rank_{D,i,j}^{[1]} = \rank_{D,\permutation_{j,d-1}(i),j} \geq 0$. 
It remains to show: $\rank_{D,x}^{[1]} \leq mx$, $\forall 0 \leq x\leq dn-1$.
\begin{enumerate}[wide,labelwidth=!,labelindent=0pt,itemindent=!,label=(4.\arabic*)]
\item
If $0\leq x < (d-1)n$, then by $(3)$, we have
\[
\rank_{D,x}^{[1]} \leq \rank_{D,x-1}^{[1]} + m \leq \cdots \leq \rank_{D,0}^{[1]} + mx = \rank_{D,0} + mx = mx. 
\]

\item
If $(d-1)n \leq x = (d-1)n + j < dn$, then
\[
\rank_{D,x}^{[1]} = \rank_{D,\permutation_{j,d-1}(d-1),j} \leq m(\permutation_{j,d-1}(d-1)\cdot n + j) \leq m((d-1)n + j) = mx.
\]
\end{enumerate}

\item
Observe that, by Corollary \ref{cor:Dyck_path_via_rank}, we have shown that $(y_{i,j}^{[1]})_{0\leq i\leq d-2,0\leq j\leq n-1} \in \Young((d-1)n,(d-1)m)$. 
Besides, for any $\moveindex_{d-1} \leq x \leq (d-2)n$, we have
\[
\rank_{D,x}^{[1]} = \rank_{D,x+n} > \rank_{D,x+n+n-1} + m = \rank_{D,x+n-1}^{[1]} + m.
\]
So by definition of $\moveindex_{d-2}$, we have $\moveindex_{d-2} < \moveindex_{d-1}$.
Now, by the definition of $\moveindex_u$, $d-2\geq u \geq 1$, and the inductive hypothesis applied to $(y_{i,j}^{[1]})_{0\leq i\leq d-2,0\leq j\leq n-1} \in \Young((d-1)n,(d-1)m)$, we conclude that $\moveindex_u$, $d-2\geq u\leq 1$, are well-defined, and $0 < \moveindex_1 < \cdots < \moveindex_{d-2} \leq (d-2)n$.
Altogether, $0 < \moveindex_1 < \cdots < \moveindex_{d-1} \leq (d-1)n$. 
In particular, $(y_{i,j}^{[d-1]})_{0\leq i\leq d-1,0\leq j\leq n-1}$ is defined.
Moreover, again by inductive hypothesis applied to $(y_{i,j}^{[1]})_{0\leq i\leq d-2,0\leq j\leq n-1}$,
for each $1\leq u < d-1$, i.e., $1 \leq d-u-1 \leq d-2$, and $\moveindex_u \leq x < un$, we have
\begin{equation*}
\rank_{D,x}^{[d-u-1]} > \rank_{D,x+n}^{[d-u-1]}.
\end{equation*}

\item
We show that $\rank_{D,x} > \rank_{D,x+n}$, $\forall \moveindex_{d-1} \leq x < (d-1)n$.
Indeed, by definition of $\moveindex_{d-1}$, as $x+1> \moveindex_{d-1}$,
\[
\rank_{D,x+n} < \rank_{D,x+1} - m \leq \rank_{D,x}.
\]

\item
It remains to show that $(y_{i,j}^{[d-1]})_{0\leq i\leq d-1,0\leq j\leq n-1}\in \Adm(dn,dm)$.
Observe that by definition, we have 
\[
y_{d-1,j}^{[d-1]} = y_{d-1,j}^{[1]};\quad (y_{i,j}^{[d-1]})_{0\leq i\leq d-2,0\leq j\leq n-1} = \Psi_{d-1}((y_{i,j}^{[1]})_{0\leq i\leq d-2,0\leq j\leq n-1}).
\]
So by inductive hypothesis, we already have $(y_{i,j}^{[d-1]})_{0\leq i\leq d-2,0\leq j\leq n-1} \in \Adm((d-1)n,(d-1)m)$.
Altogether, by what we have shown above and Corollary \ref{cor:admissible_invariant_subset_via_rank}, it remains to show that $(d-2)$ is admissible: 
$\exists j\in \mathbb{Z}/n \isomorphic \{0,1,\cdots,n-1\}$ such that $\rank_{D,d-1,j+1}^{[d-1]} \leq \rank_{D,d-2,j}^{[d-1]} + m$.

Indeed, recall that for any $0\leq j\leq d-1$, $1\leq u\leq d-1$, we have $\permutationindex_{j,u} :=\min\{0\leq i\leq u:~in+j\geq \moveindex_u\}$, and $\permutation_{j,u}(u) = \permutationindex_{j,u}$. So,
\[
\rank_{D,d-1,j}^{[d-1]} = \rank_{D,d-1,j}^{[1]} = \rank_{D,\permutationindex_{j,d-1},j};\quad \rank_{D,d-2,j}^{[d-1]} = \rank_{D,d-2,j}^{[2]} = \rank_{D,\permutationindex_{j,d-2},j}^{[1]}.
\]

\begin{enumerate}[wide,labelwidth=!,labelindent=0pt,itemindent=!,label=(7.\arabic*)]
\item
If $\ell \geq 1$, 
then $0\leq k\leq d-2$ (as $\moveindex_{d-1} = kn+\ell \leq (d-1)n$), and
$kn + \ell - 1 = \moveindex_{d-1} - 1 \geq \moveindex_{d-2}$. So by definition, $\permutationindex_{\ell-1,d-2} \leq k$.
It follows that
\[
\rank_{D,d-1,\ell}^{[d-1]} = \rank_{D,k,\ell} \leq \rank_{D,k,\ell-1} + m = \rank_{D,k,\ell-1}^{[1]} \leq \rank_{D,\permutationindex_{\ell-1,d-2},\ell-1}^{[1]} + m = \rank_{D,d-2,\ell-1}^{[d-1]} + m.
\]
Here, the second inequality follows from the last inequality in $(5)$.

\item
If $\ell = 0$, then $0 \leq k - 1 \leq d - 2$ (as $0 < \moveindex_{d-1} = kn \leq (d-1)n$), and $(k-1)n + n - 1 = \moveindex_{d-1} - 1 \geq \moveindex_{d-2}$. 
So by definition, $\permutationindex_{n-1,d-2} \leq k - 1$. 
As above, it follows that
\[
\rank_{D,d-1,\ell}^{[d-1]} = \rank_{D,k,0} \leq \rank_{D,k-1,n-1} + m = \rank_{D,k-1,n-1}^{[1]} \leq \rank_{D,\permutationindex_{n-1,d-2},n-1}^{[1]} + m = \rank_{D,d-2,n-1}^{[d-1]} + m.
\]
\end{enumerate}
Altogether, this shows that $d-2$ is admissible.
\end{enumerate}
This completes the proof of Proposition \ref{prop:from_Dyck_paths_to_admissible_invariant_subsets}.
\end{proof}

\begin{remark}\label{rem:intermediate_diagrams_from_Dyck_paths_to_admissible_invariant_subsets}
For every $D \cong (y_{i,j}) \in \Young(dn,dm)$ and $0 \leq u \leq d-1$, the proof of Proposition~\ref{prop:from_Dyck_paths_to_admissible_invariant_subsets} shows that
\[
(y_{i,j}^{[d-u-1]})_{0 \leq i \leq u,\,0 \leq j \leq n-1} \in \Young(un,um),
\quad
y_{i,j}^{[d-u-1]} = y_{i,j}^{[d-1]},\ \forall\, u+1 \leq i \leq d-1,\,0 \leq j \leq n-1.
\]
In particular,
\[
y_{i,0}^{[d-u-1]} \leq y_{i,1}^{[d-u-1]} \leq \cdots \leq y_{i,n-1}^{[d-u-1]} \leq y_{i,0}^{[d-u-1]} + m,
\quad \forall\, u+1 \leq i \leq d-1.
\]
\end{remark}

\subsection{The bijection}

Set $[1,k]:=\{1,\dots,k\}$ for all $k\in\mathbb{Z}_{>0}$. 
So far, we have obtained well-defined maps
\[
(\Phi_d,\tilde{\moveindex}_1,\cdots,\tilde{\moveindex}_{d-1}):\Adm(dn,dm) \rightarrow \Young(dn,dm)\times \prod_{u=1}^{d-1}[1,un],\quad 0 < \tilde{\moveindex}_1 < \cdots < \tilde{\moveindex}_{d-1},
\]
and
\[
(\Psi_d,\moveindex_1,\cdots,\moveindex_{d-1}):\Young(dn,dm) \rightarrow \Adm(dn,dm)\times \prod_{u=1}^{d-1}[1,un],\quad 0 < \moveindex_1 < \cdots < \moveindex_{d-1},
\]

\begin{proposition}\label{prop:bijection_between_admissible_invariant_subsets_and_Dyck_paths}
We have 
\[
\tilde{\moveindex}_u=\moveindex_u\circ\Phi_d,
\quad 
\moveindex_u=\tilde{\moveindex}_u\circ\Psi_d,\qquad 
\forall\,0\leq u\leq d-1.
\]
Moreover, $\Phi_d,\Psi_d$ are inverse to each other. In particular, for any 
$D\in\Young(dn,dm)$ with $\Delta=\Psi_d(D)\in\Adm(dn,dm)$, we have
\[
\permutationindex_{j,u}=\tilde{\permutationindex}_{j,u},\quad 
\permutation_{j,u}=\tilde{\permutation}_{j,u},\quad 
\permutation_j=\tilde{\permutation}_j,
\quad \forall\,0\leq j\leq n-1,\,1\leq u\leq d-1.
\]
\end{proposition}

\begin{proof}
We proceed by induction on $d$. The case $d = 1$ is trivial. 
Assume the statement holds for all $d'<d$ with $d\geq2$, and consider the case $d$.

\noindent{}\textbf{Step} $1$.
We show that $\tilde{\moveindex}_u = \moveindex_u\circ\Phi_d$ and $\Psi_d\circ\Phi_d = \id$.

Fix 
\[
\Delta \isomorphic (c_{in+j} = c_{i,j})_{0\leq i\leq d-1,0\leq j\leq n-1} \in \Adm(dn,dm). 
\]
Let
\[
D\isomorphic (y_{in+j} = y_{i,j})_{0\leq i\leq d-1,0\leq j\leq n-1} := \Phi_d((c_{\bullet,\bullet}))\in \Young(dn,dm). 
\]
In other words, 
\[
\rank_{D,in+j} = \rank_{D,i,j} := \rank_{\Delta,in+j}^{[d-1]} = \rank_{\Delta,i,j}^{[d-1]}.
\]
For simplicity, denote
\[
\tilde{\moveindex}_u:= \tilde{\moveindex}_u((c_{i,j})),\quad \moveindex_u:= \moveindex_u((y_{i,j})),\qquad \forall 1\leq u\leq d-1.
\]

By definition, 
\[
(c_{i,j}^{[d-2]})_{0\leq i\leq d-2,\,0\leq j\leq n-1}
=\Phi_{d-1}((c_{i,j})_{0\leq i\leq d-2,\,0\leq j\leq n-1})
\in \Young((d-1)n,(d-1)m),
\]
and $c_{d-1,j}^{[d-2]}=c_{d-1,j}$. Thus,
\[
\tilde{\moveindex}_{d-1} = \max\{in+j+1:~0\leq i\leq d-2,0\leq j\leq n-1, \rank_{\Delta,d-1,j+1} \leq \rank_{\Delta,i,j}^{[d-2]} + m\}.
\]
By induction,
\[
\moveindex_{d-2}((c_{i,j}^{[d-2]})_{0\leq i\leq d-2,0\leq j\leq n-1}) = 
\tilde{\moveindex}_{d-2} < \tilde{\moveindex}_{d-1}.
\]

\begin{enumerate}[wide,labelwidth=!,labelindent=0pt,itemindent=!]
\item
We show that $\moveindex_{d-1} = \tilde{\moveindex}_{d-1}$.
\begin{enumerate}[wide,labelwidth=!,labelindent=0pt,itemindent=!,label=(1.\arabic*)]
\item
If $\tilde{\moveindex}_{d-1} < x = in+j+1 \leq \tilde{\moveindex}_{d-1} + n-1$, with $0\leq j\leq n-1$, then by definition, we have
\[
\rank_{D,x} = \rank_{\Delta,x}^{[d-1]} = \rank_{\Delta,d-1,j+1} > \rank_{\Delta,i,j}^{[d-2]} + m = \rank_{\Delta,i+1,j}^{[d-1]} + m = \rank_{D,x+n-1} + m.
\]
Here, the inequality follows from the definition of $\tilde{\moveindex}_{d-1}$.

\item
If $\tilde{\moveindex}_{d-2} + n < \tilde{\moveindex}_{d-1} + n \leq x = in+j \leq (d-1)n$, with $0\leq j\leq n-1$, then $x-n > \tilde{\moveindex}_{d-2} = \moveindex_{d-2}((c_{i,j}^{[d-2]})_{0\leq i\leq d-2,0\leq j\leq n-1})$. It follows that
\[
\rank_{D,x} = \rank_{\Delta,x}^{[d-1]} = \rank_{\Delta,x-n}^{[d-2]} > \rank_{\Delta,x-1}^{[d-2]} + m = \rank_{\Delta,x+n-1}^{[d-1]} + m = \rank_{D,x+n-1} + m.
\]

\item
Notice that by Lemma \ref{lem:semimodule_vs_c_ij}.$(1)$, we have $\rank_{\Delta,d-1,j+1} \leq \rank_{\Delta,d-1,j} + m$, $\forall j \in \mathbb{Z}/n \isomorphic \{0,1,\cdots,n-1\}$.
Now, for $x = \tilde{\moveindex}_{d-1} = kn + \ell$, with $0\leq \ell \leq n-1$, we have
\[
\rank_{D,x} = \rank_{\Delta,x}^{[d-1]} = \rank_{\Delta,d-1,\ell} \leq \rank_{\Delta,d-1,\ell-1} + m = \rank_{\Delta,kn+\ell+n-1}^{[d-1]} + m = \rank_{D,x+n-1} + m.
\]
\end{enumerate}
Altogether, by definition, this shows that $\moveindex_{d-1} = \tilde{\moveindex}_{d-1}$, as desired.

\item
Now, by definition, we have $\tilde{\permutation}_{j,d-1} = \permutation_{j,d-1}$, $\forall 0\leq j\leq n-1$. Then for any $0\leq i\leq d-1$, $0\leq j\leq n-1$, it follows that
\[
\rank_{D,i,j}^{[1]} = \rank_{D,\permutation_{j,d-1}(i),j} = \rank_{\Delta,\tilde{\permutation}_{j,d-1}(i),j}^{[d-1]} = \rank_{\Delta,i,j}^{[d-2]},~~\Rightarrow~~y_{i,j}^{[1]} = c_{i,j}^{[d-2]}.
\]
Then by definition, we have 
\[
y_{d-1,j}^{[d-1]} = y_{d-1,j}^{[1]} = c_{d-1,j}^{[d-2]} = c_{d-1,j},\quad\forall 0\leq j\leq n-1.
\]
Moreover, by definition and inductive hypothesis, we have
\begin{eqnarray*}
(y_{i,j}^{[d-1]})_{0\leq i\leq d-2,0\leq j\leq n-1} &=& \Psi_{d-2}((y_{i,j}^{[1]})_{0\leq i\leq d-2,0\leq j\leq n-1}) = \Psi_{d-2}((c_{i,j}^{[d-2]})_{0\leq i\leq d-2,0\leq j\leq n-1})\\
&=& \Psi_{d-2}\circ\Phi_{d-2}((c_{i,j})_{0\leq i\leq d-2,0\leq j\leq n-1}) = (c_{i,j})_{0\leq i\leq d-2,0\leq j\leq n-1},
\end{eqnarray*}
and for each $1\leq u\leq d-2$.
\[
\moveindex_u = \moveindex_u((y_{i,j}^{[1]})_{0\leq i\leq d-2,0\leq j\leq n-1}) = \moveindex_u\circ\Phi_{d-2}((c_{i,j})_{0\leq i\leq d-2,0\leq j\leq n-1}) = \tilde{\moveindex}_u((c_{i,j})_{0\leq i\leq d-2,0\leq j\leq n-1}) = \tilde{\moveindex}_u.
\]
Altogether, as $\Psi_d((y_{i,j})_{0\leq i\leq d-1,0\leq j\leq n-1}) =  (y_{i,j}^{[d-1]})_{0\leq i\leq d-1,0\leq j\leq n-1}$, this shows that
\[
\Psi_d\circ\Phi_d((c_{i,j})_{0\leq i\leq d-1,0\leq j\leq n-1}) =  (y_{i,j}^{[d-1]})_{0\leq i\leq d-1,0\leq j\leq n-1} =  (c_{i,j})_{0\leq i\leq d-1,0\leq j\leq n-1}.
\]
\end{enumerate}
This completes \textbf{Step} $1$.

\noindent{}\textbf{Step} $2$.
We show that $\moveindex_u = \tilde{\moveindex}_u\circ\Psi_d$ and $\Phi_d\circ\Psi_d = \id$.

Fix 
\[
D \isomorphic (y_{in+j} = y_{i,j})_{0\leq i\leq d-1,0\leq j\leq n-1} \in \Young(dn,dm).
\]
Let
\[
\Delta \isomorphic (c_{in+j} = c_{i,j})_{0\leq i\leq d-1,0\leq j\leq n-1} := \Psi_d((y_{\bullet,\bullet}))\in \Adm(dn,dm).
\] 
So, $\rank_{\Delta,in+j} = \rank_{\Delta,i,j} := \rank_{D,in+j}^{[d-1]} = \rank_{D,i,j}^{[d-1]}$.
For simplicity, denote
\[
\moveindex_u:= \moveindex_u((y_{i,j})),\quad \tilde{\moveindex}_u:= \tilde{\moveindex}_u((c_{i,j})),\qquad \forall 1\leq u\leq d-1.
\]

By definition, $c_{d-1,j} = y_{d-1,j}^{[1]}$,
$(y_{i,j}^{[1]})_{0\leq i < d-1,0\leq j < n}\in\Young((d-1)n,(d-1)m)$, and $(c_{i,j})_{0\leq i < d-1,0\leq j < n} = \Psi_{d-2}((y_{i,j}^{[1]})_{0\leq i < d-1,0\leq j < n})$.
By induction,
\begin{eqnarray*}
(c_{i,j}^{[d-2]})_{0\leq i < d-1,0\leq j < n} &=& \Phi_{d-2}((c_{i,j})_{0\leq i < d-1,0\leq j < n})\\
&=& \Phi_{d-2}\circ\Psi_{d-2}((y_{i,j}^{[1]})_{0\leq i < d-1,0\leq j < n}) = (y_{i,j}^{[1]})_{0\leq i < d-1,0\leq j < n},
\end{eqnarray*}
and for each $1\leq u\leq d-2$,
\[
\tilde{\moveindex}_u = \tilde{\moveindex}_u((c_{i,j})_{0\leq i < d-1,0\leq j < n}) = \tilde{\moveindex}_u\circ\Psi_{d-2}((y_{i,j}^{[1]})_{0\leq i < d-1,0\leq j < n}) = \moveindex_u((y_{i,j}^{[1]})_{0\leq i < d-1,0\leq j < n}) = \moveindex_u.
\]

\begin{enumerate}[wide,labelwidth=!,labelindent=0pt,itemindent=!]
\item
We show that $\tilde{\moveindex}_{d-1} = \moveindex_{d-1}$.
\begin{enumerate}[wide,labelwidth=!,labelindent=0pt,itemindent=!,label=(1.\arabic*)]
\item
Fix any $\moveindex_{d-1} < x = in+j+1 \leq (d-1)n$, with $0\leq i\leq d-2$, $0\leq j\leq n-1$.
Define $0\leq \underline{j+1}\leq n-1$ by $j+1 \equiv \underline{j+1} \mod n$. By definition of $\permutationindex_{\underline{j+1},d-1}$, we have
$x = in + j+1 \geq \permutationindex_{\underline{j+1},d-1}\cdot n + \underline{j+1}$.
So,
\[
\rank_{\Delta,d-1,j+1} = \rank_{D,d-1,j+1}^{[1]}
= \rank_{D,\permutationindex_{\underline{j+1},d-1}+\underline{j+1}} \geq \rank_{D,x} > \rank_{D,x+n-1} + m = \rank_{D,x-1}^{[1]} + m = \rank_{\Delta,i,j}^{[d-2]} + m.
\]
Here, the first inequality follows from Proposition \ref{prop:from_Dyck_paths_to_admissible_invariant_subsets}.$(c)$.

\item
For $x = in+j+1 = \moveindex_{d-1}$ with $0\leq i\leq d-2$, $0\leq j\leq n-1$,
\[
\rank_{\Delta,d-1,j+1} = \rank_{D,d-1,j+1}^{[1]} = \rank_{D,x} \leq \rank_{D,x-1} + m = \rank_{D,i,j}^{[1]} + m = \rank_{\Delta,i,j}^{[d-2]} + m.
\]
\end{enumerate}
Altogether, by definition, this shows that $\tilde{\moveindex}_{d-1} = \moveindex_{d-1}$, as desired.

\item
Now by definition, $\permutation_j = \tilde{\permutation}_j$ for all $0\leq j\leq n-1$. For $0\leq i\leq d-1$, $0\leq j\leq n-1$, it follows that
\[
\rank_{\Delta,i,j}^{[d-1]} = \rank_{\Delta,\tilde{\permutation}_j^{-1}(i),j} = \rank_{D,\permutation_j^{-1}(i),j}^{[d-1]} = \rank_{D,i,j},~~\Rightarrow~~c_{i,j}^{[d-1]} = y_{i,j}.
\]
In other words, we have
\[
\Phi_d\circ\Psi_d((y_{i,j})_{0\leq i\leq d-1,0\leq j\leq n-1}) = (c_{i,j}^{[d-1]})_{0\leq i\leq d-1,0\leq j\leq n-1} = (y_{i,j})_{0\leq i\leq d-1,0\leq j\leq n-1}.
\]
\end{enumerate}
This completes the proof of Proposition \ref{prop:bijection_between_admissible_invariant_subsets_and_Dyck_paths}.
\end{proof}

\begin{corollary}\label{cor:enhanced_rank_is_a_bijection}
Theorem~\ref{thm:admissible_invariant_subsets_vs_Dyck_paths}(3) holds:  
for any $D\cong(y_{i,j})\in\Young(dn,dm)$, the map $\hat{\rank}_D$ defines a bijection
\[
\hat{\rank}_D:\overline{\cR}_{dn,dm}^+\xrightarrow{\;\simeq\;}\mathbb{N},
\]
such that $\lfloor\tfrac{\hat{\rank}_D}{d}\rfloor=\rank$, and 
\[
\hat{\rank}_D(\cR_{dn,dm}^+\setminus D)=\mathbb{N}\setminus\Delta,
\qquad \Delta:=\Psi_d(D)\in\Adm(dn,dm).
\]
\end{corollary}

\begin{proof}
Identify each box with its top-left vertex. Then 
$\rank_{\Delta,i,j}=\rank_{D,\permutation_j(i),j}$.
 Recall
\[
\hat{\rank}_D(in+j,y):=d\cdot\rank(x,y)+\permutation_j^{-1}(i)
=d(m(in+j)-ny)+\permutation_j^{-1}(i),\quad 0\leq i\leq d-1, 0\leq j\leq n-1.
\]
Thus $\lfloor\hat{\rank}_D/d\rfloor=\rank$.

Next, for fixed $i,j$, $\hat{\rank}_D$ maps 
$\{(in+j,y):y\in\mathbb{Z}\}$ bijectively onto an arithmetic progression $a_{\permutation_j^{-1}(i),j}+dn\mathbb{Z}$. 
Since the residues form a complete set mod $dn$, $\hat{\rank}_D$ is a bijection 
$\{(x,y):0\leq x\leq dn-1\}\to\mathbb{Z}$. 
Moreover,
\[
\hat{\rank}_D(in+j,y)\geq0 
\iff y\leq a_{i,j} \iff (in+j,y)\in\overline{\cR}_{dn,dm}^+.
\]
So $\hat{\rank}_D$ restricts to $\overline{\cR}_{dn,dm}^+\to\mathbb{N}$ bijectively.

Finally, denote $(c_{i,j}) := (y_{i,j}^{[d-1]})$. Note that $A(\Delta) = \{\hat{b}_{k,\ell}:~0\leq k \leq d-1,0\leq \ell \leq n-1\}$, with
\[
\hat{b}_{k,\ell} = \hat{a}_{k,\ell} - dnc_{k,\ell} = d(m(kn+\ell) - nc_{k,\ell}) + k = d\cdot\rank_{\Delta,k,\ell} + k = d\cdot\rank_{D,\permutation_{\ell}(k),\ell} + k.
\]
Thus, for $x = in + j$ with $0\leq i\leq d-1$, $0\leq j\leq n-1$,
\begin{eqnarray*}
&&\hat{\rank}_D(x,y) \in \Delta~~\Leftrightarrow~~\hat{\rank}_D(x,y) = d\cdot\rank(x,y) + \permutation_j^{-1}(i) \geq \hat{b}_{\permutation_j^{-1}(i),j} = d\cdot\rank_{D,i,j} + \permutation_j^{-1}(i)\\
&\Leftrightarrow& \rank(x,y) = mx - ny \geq \rank_{D,i,j} = mx - ny_x ~~\Leftrightarrow~~ y \leq y_x.
\end{eqnarray*}
Hence $\hat{\rank}_D(\overline{D})=\Delta$, completing the proof.
\end{proof}

\subsection{The statistics}

\begin{definition}
Fix $\Delta \cong (c_{i,j}) \in \Adm(dn,dm)$. 
\begin{enumerate}[wide,labelwidth=!,labelindent=0pt,itemindent=!]
\item An element $b \notin \Delta$ is an \textbf{$(dm)$-cogenerator} if $b+dm \in \Delta$.
\item Define
\[
\GC((c_{i,j})) = \GC(\Delta) := 
\{(a,b) \mid b>a,\; a\in A(\Delta),\; b \text{ an $(dm)$-cogenerator of }\Delta\}.
\]
For a $(dn)$-generator $a$ of $\Delta$, set
\[
\GC(a;(c_{i,j})) = \GC(a;\Delta) := \{(a,b) \in \GC(\Delta)\}.
\]
\end{enumerate}
\end{definition}

\begin{lemma}[{\cite[Lem.~4.1]{GMO25}}]\label{lem:cell_dimesion_via_generator-cogenerators}
For any $a \in A(\Delta)$,
\[
|\Gaps(a)| - |\Gaps(a+dm)| = |\GC(a;\Delta)|.
\]
In particular,
\[
\dim(\Delta) = |\GC(\Delta)|.
\]
\end{lemma}

We now reinterpret the lemma in our terminology.  
Fix $\Delta \cong (c_{i,j}) \in \Adm(dn,dm)$. Recall
\[
\hat{\rank}_{\Delta}(x=in+j,y) := d\cdot \rank(x,y)+i 
= d(mx-ny)+i,\qquad 0\leq i\leq d-1,\; 0\leq j\leq n-1.
\]
Moreover, $\Delta = \bigsqcup_{a\in A(\Delta)} \bigl(a+dn\mathbb{N}\bigr)$, where
\[
A(\Delta) = \{\hat{a}_{i,j} - dn c_{i,j} : 0\leq i\leq d-1,\;0\leq j\leq n-1\}.
\]
That is,
\begin{equation}
\hat{a}_{i,j} - dn c_{i,j}
= d(m(in+j)-nc_{i,j})+i
= d\cdot \rank_{\Delta,i,j}+i
= \hat{\rank}_{\Delta}(in+j,c_{i,j})
=: \hat{\rank}_{\Delta,i,j}.
\end{equation}

Note $c_{i,n} := c_{i,0}+m$, which is compatible with
$\hat{\rank}_{\Delta,i,n}=\hat{\rank}_{\Delta,i,0}$.  
For any $z_{i,j}\in\mathbb{Z}$, set
\begin{equation}
\hat{z}_{i,j}:=\hat{\rank}_{\Delta}(in+j,z_{i,j}) = d(m(in+j) - nz_{i,j}) + i = \hat{a}_{i,j} - dnz_{i,j},\quad 0\leq i\leq d-1, 0\leq j\leq n-1.
\end{equation}

\begin{lemma}\label{lem:characterization_of_cogenerator}
For all $z_{i,j}\in\mathbb{Z}$, $0\leq i\leq d-1$, $0\leq j\leq n-1$,
\begin{align*}
\hat{z}_{i,j} \text{ is an $(dm)$-cogenerator of }\Delta 
&\;\;\Longleftrightarrow\;\;
\hat{\rank}_{\Delta,i,j} > \hat{\rank}_{\Delta}(in+j,z_{i,j}) 
\geq \hat{\rank}_{\Delta,i,j+1}-dm \\
&\;\;\Longleftrightarrow\;\;
\rank_{\Delta,i,j} > \rank(in+j,z_{i,j}) 
\geq \rank_{\Delta,i,j+1}-m \\
&\;\;\Longleftrightarrow\;\;
c_{i,j} < z_{i,j} \leq c_{i,j+1}.
\end{align*}
\end{lemma}

\begin{proof}
Since $\hat{z}_{i,j} \equiv \hat{\rank}_{\Delta,i,j} (\in A(\Delta)) \pmod{dn}$ and 
$\hat{z}_{i,j}+dm \equiv \hat{\rank}_{\Delta,i,j+1} (\in A(\Delta)) \pmod{dn}$, the claim follows directly from the definition of a $(dm)$-cogenerator.
\end{proof}

\begin{definition}
Fix $\Delta \cong (c_{in+j}) = (c_{i,j})_{0\leq i\leq d-1,0\leq j\leq n-1}\in \Adm(dn,dm)$, with $c_{i,n}:=c_{i,0}+m$.
\begin{enumerate}[wide,labelwidth=!,labelindent=0pt,itemindent=!]
\item 
For $(p,q)\in\mathbb{Z}^2$, $p=kn+\ell$ with $0\leq p\leq dn-1$, $0\leq \ell\leq n-1$, define
\begin{align}
\cG(\Delta;(p,q)) 
:=&\;\#\{(x=in+j,y)\in\mathbb{Z}^2 : 0\leq x\leq dn-1,\;
\hat{\rank}_{\Delta}(x,y) > \hat{\rank}_{\Delta}(p,q), \nonumber\\
&\hspace{5em}\rank_{\Delta,i,j+1}-m \leq \rank(x,y) < \rank_{\Delta,i,j}\}. 
\end{align}
Here, by $x = in+j$, we mean $0\leq j\le n-1$.
Equivalently, by Lemma~\ref{lem:characterization_of_cogenerator},
\begin{align*}
\cG(\Delta;(p,q)) 
=&\;\#\{(x=in+j,y)\in\mathbb{Z}^2 \mid 0\leq x<dn,\;
\rank(x,y) > \rank(p,q),\; c_{i,j}<y\leq c_{i,j+1}, \\
&\;\;\text{or } k<i\leq d-1,\;\rank(x,y)=\rank(p,q),\; c_{i,j}<y\leq c_{i,j+1}\}.
\end{align*}

\item
Define
\[
\cG(\Delta) := (\cG_0(\Delta),\dots,\cG_{dn-1}(\Delta)),
\]
the nondecreasing reordering of $\cG(\Delta;(p,c_p))$ for $0\leq p\leq dn-1$.
\end{enumerate}
\end{definition}
\noindent\textbf{Note.} By Lemmas~\ref{lem:cell_dimesion_via_generator-cogenerators} and \ref{lem:characterization_of_cogenerator}, for any $\Delta\cong(c_{i,j})\in \Adm(dn,dm)$,
\begin{equation}
\dim \Delta = |\GC(\Delta)| = |\cG(\Delta)| := \sum_{x=0}^{dn-1}\cG_x(\Delta).
\end{equation}

\begin{definition}\label{def:sweep_map}
Let $(y_{in+j})=(y_{i,j})_{0\leq i\leq d-1,0\leq j\leq n-1}$ be nonnegative integers with $y_0=0$. Set $y_{dn}=y_{d,0}:=dm$.  
For $(p,q)\in\mathbb{Z}^2$, $0\leq p\leq dn-1$:
\begin{enumerate}[wide,labelwidth=!,labelindent=0pt,itemindent=!]
\item 
For $0\leq x\leq dn-1$, define
\begin{align}\label{eqn:sweep_map-index_functions}
\epsilon((y_{\bullet,\bullet});(p,q);x) :=&\; \#\{0<y\leq y_{x+1} \mid \rank(x,y)>\rank(p,q)\} \\
&- \#\{0<y\leq y_x \mid \rank(x,y)>\rank(p,q)\}, \nonumber\\
\eta((y_{\bullet,\bullet});(p,q);x) :=&\; \delta_{[p,dn)}(x)\Bigl(\#\{0<y\leq y_{x+1} \mid \rank(x,y)=\rank(p,q)\} \\
&- \#\{0<y\leq y_x \mid \rank(x,y)=\rank(p,q)\}\Bigr).\nonumber
\end{align}
Here, $\delta_{[p,dn)}(x)$ is defined as $1$ if $p\leq x < dn$, and $0$ else.

\item
For $0\leq j\leq n-1$, set
\begin{align}
\epsilon_j((y_{\bullet,\bullet});(p,q)) &:= \sum_{i=0}^{d-1} \epsilon((y_{\bullet,\bullet});(p,q);in+j), \quad 
\epsilon((y_{\bullet,\bullet});(p,q)):=\sum_{j=0}^{n-1}\epsilon_j((y_{\bullet,\bullet});(p,q)), \\
\eta_j((y_{\bullet,\bullet});(p,q)) &:= \sum_{i=0}^{d-1} \eta((y_{\bullet,\bullet});(p,q);in+j), \quad
\eta((y_{\bullet,\bullet});(p,q)):=\sum_{j=0}^{n-1}\eta_j((y_{\bullet,\bullet});(p,q)).
\end{align}
\noindent{}\textbf{Note}: if $\rank(p,q)=\rank(p',q')$, then 
$\epsilon_j((y_{\bullet,\bullet});(p,q)) = \epsilon_j((y_{\bullet,\bullet});(p',q'))$.

\item
\textbf{Define}
\begin{equation}\label{eqn:sweep_map}
\zeta((y_{\bullet,\bullet});(p,q)) := \epsilon((y_{\bullet,\bullet});(p,q)) +  \eta((y_{\bullet,\bullet});(p,q)).
\end{equation}

\item
({\color{blue}sweep map})
Finally, define the \emph{sweep map}
\[
\zeta((y_{\bullet,\bullet})) := (\zeta_0((y_{\bullet,\bullet})),\dots,\zeta_{dn-1}((y_{\bullet,\bullet}))),
\]
the nondecreasing reordering of $\zeta((y_{\bullet,\bullet});(p,y_p))$ for $0\leq p\leq dn-1$.
\end{enumerate}
\end{definition}
\noindent\textbf{Note.} If $D\cong(y_{in+j}=y_{i,j})\in \Young(dn,dm)$, then
\begin{align*}
\epsilon(D;(p,q)) &= \#\{(x,y): 0\leq x\leq dn-1,\; y_x<y\leq y_{x+1},\; \rank(x,y)>\rank(p,q)\},\\
\eta(D;(p,q)) &= \#\{(x,y): p\leq x\leq dn-1,\; y_x<y\leq y_{x+1},\; \rank(x,y)=\rank(p,q)\}.
\end{align*}

\begin{proposition}\label{prop:sweep_map_vs_bijection}
We have
\[
\zeta = \cG\circ\Psi_d.
\]
Moreover, for $D\cong(y_{i,j})\in \Young(dn,dm)$ with 
$\Delta:=\Psi_d(D)\cong(c_{i,j})\in\Adm(dn,dm)$, and any $p=kn+\ell$ with $0\leq k\leq d-1$, $0\leq \ell\leq n-1$, we have
\[
\cG((c_{i,j});(p,c_p))
= \zeta\bigl((y_{i,j});(\permutation_\ell(k)n+\ell,\; y_{\permutation_\ell(k),\ell})\bigr).
\]
In particular,
\[
e(\Delta) = |(c_{i,j})| = |D|,\quad \dim\Delta = |\zeta(D)| = \codinv(D).
\]
\end{proposition}

\begin{remark}\label{rem:compare_with_the_old_bijection}
In \cite{GMV20}, a bijection $\cD^{-1}$ was constructed in a more complicated way, and shown to satisfy the same property $\zeta=\cG\circ \cD^{-1}$.  
It is known that the sweep map $\zeta\colon \Young(dn,dm)\to\Young(dn,dm)$ is a bijection \cite{TW18}. Hence $\cD^{-1}=\Psi_d$, i.e.\ $\cD=\Phi_d$.
\end{remark}

The proof of Proposition~\ref{prop:sweep_map_vs_bijection} occupies the rest of this subsection.  
In what follows, $y\in\mathbb{Z}$ will be implicit.

\begin{lemma}\label{lem:alternative_description_of_index_functions}
In Definition \ref{def:sweep_map}, denote 
\[
\rank_x := \rank(x,y_x) = mx - ny_x, \qquad r := \rank(p,q).
\]
In particular, $\rank_{dn} = \rank_0 = 0$.
Suppose $p = kn+\ell$ with $0\leq k\leq d-1$, $0\leq \ell \leq n-1$. Fix $0\leq j\leq n-1$.
\begin{enumerate}[wide,labelwidth=!,labelindent=0pt,itemindent=!]
\item 
We have
\[
\epsilon_j((y_{\bullet,\bullet});(p,q)) 
= \sum_{i=0}^{d-1} \Bigl[ \#\{y : r < \rank(in+j,y) < \rank_{in+j}\} 
   - \#\{y : r < \rank(in+j,y) < \rank_{in+j+1} - m\} \Bigr].
\]
Moreover, if $(\tilde{y}_{\bullet,\bullet})$ is another such matrix and there exists a sequence 
of permutations $s_b\acts [0,d-1]$, $0\leq b \leq n-1$, with $s_0(0)=0$ and 
\[
y_{a,b} = \tilde{y}_{s_b(a),b} - (s_b(a)-a)m,
\]
then
\[
\epsilon_j((y_{\bullet,\bullet});(p,q)) 
= \epsilon_j((\tilde{y}_{\bullet,\bullet});(p,q)).
\]

\item
For $j\neq \ell$, we have $\eta_j((y_{\bullet,\bullet});(p,q))=0$. Moreover,
\[
\eta((y_{\bullet,\bullet});(p,q)) 
= \eta_\ell((y_{\bullet,\bullet});(p,q)) 
= \sum_{i=k}^{d-1}\Bigl[ \delta_{(-\infty,\rank_{in+\ell})}(r) 
   - \delta_{(-\infty,\rank_{in+\ell+1}-m)}(r)\Bigr].
\]
Here, for any $-\infty\leq a < b < \infty$, set $\delta_{(a,b)}(c) := 1$ if $a<c<b$ and $0$ otherwise.
\end{enumerate}
\end{lemma}

\begin{proof}
\noindent (1) Fix $0\leq i\leq d-1$ and set $x := in+j$. Observe:
\[
0<y\leq y_x \;\;\Longleftrightarrow\;\; \rank_x \leq \rank(x,y) < mx,
\qquad 
0<y\leq y_{x+1} \;\;\Longleftrightarrow\;\; \rank_{x+1}-m \leq \rank(x,y) < mx.
\]
Thus,
\begin{eqnarray*}
&&\#\{y:~0<y\leq y_x, \rank(x,y) > r\} = \#\{y:~\rank_x \leq \rank(x,y) < mx, \rank(x,y) > r\}\\
&=& \#\{y:~r < \rank(x,y) < mx\} - \#\{y:~r < \rank(x,y) < \rank_x\}. 
\end{eqnarray*}
Similarly,
\[
\#\{y:~0<y\leq y_{x+1}, \rank(x,y) > r\} = \#\{y:~r < \rank(x,y) < mx\} - \#\{y:~r < \rank(x,y) < \rank_{x+1} - m\}. 
\]
The first formula follows.

For the invariance formula: for each $x=in+j$, we have
\[
\rank_x = \rank(x,y_x) = \rank(\tilde{x} = s_j(i)n+j,\tilde{y}_{s_j(i)n+j}) =: \tilde{\rank}_{\tilde{x}};~~ \rank(x,y) = \rank(\tilde{x} = s_j(i)n+j,\tilde{y} = y+(s_j(i) - i)m).
\]
Thus,
\[
r < \rank(x=in+j,y) < \rank_x \iff r < \rank(\tilde{x} = s_j(i)n+j,\tilde{y} = y+(s_j(i) - i)m) < \tilde{\rank}_{\tilde{x}}.
\]
Summing up,
\[
\sum_{i=0}^{d-1}\#\{y:~r < \rank(x=in+j,y) < \rank_x\} = \sum_{\tilde{i}=0}^{d-1}\#\{\tilde{y}:~r < \rank(\tilde{x}=\tilde{i}n+j,\tilde{y}) < \tilde{\rank}_{\tilde{x}}\}. 
\]

Similarly, for $0\leq j\leq n-2$,
\begin{eqnarray*}
&&\rank_{x+1} = \rank(x+1 = in+j+1,y_{i,j+1}) = \rank(\tilde{x}+1 = s_{j+1}(i)n+j+1,\tilde{y}_{s_{j+1}(i),j+1}) = \tilde{\rank}_{\tilde{x}+1},\\
&&\rank(x = in+j,y) = \rank(\tilde{x} = s_{j+1}(i)n+j,y+(s_{j+1}(i)-i)m).
\end{eqnarray*}
For $j = n-1$, denote $s_0(d):= d$, as $s_0(0) = 0$, $s_0$ is a bijection on $\{1,2,\cdots,d\}$. Then,
\begin{eqnarray*}
&&\rank_{x+1} = \rank(x+1=(i+1)n,y_{i+1,0}) = \rank(\tilde{x}+1=s_0(i+1)n,\tilde{y}_{s_0(i+1),0}) = \tilde{\rank}_{\tilde{x}+1},\\
&&\rank(x=(i+1)n-1,y) = \rank(\tilde{x} = s_0(i+1)n-1,\tilde{y} = y + (s_0(i+1)-(i+1))m).
\end{eqnarray*}
In either case, summing up,
\[
\sum_{i=0}^{d-1}\#\{y:~r < \rank(x=in+j,y) < \rank_{x+1} - m\} = \sum_{\tilde{i}=0}^{d-1}\#\{\tilde{y}:~r < \rank(\tilde{x}=\tilde{i}n+j,\tilde{y}) < \tilde{\rank}_{\tilde{x}+1} - m\}. 
\]
Hence, the invariance formula follows.

\noindent{}$(2)$. For any $x = in + j$, $0\leq i\leq d-1$, if $\rank(x,y) = r = \rank(p,q)$, then 
\[
mj \equiv r \equiv m\ell \pmod{dn} \;\;\Longrightarrow\;\; j=\ell.
\]
Thus $\eta_j((y_{\bullet,\bullet});(p,q))=0$ unless $j=\ell$. 


Now for $j=\ell$, we have 
\[
\rank(x=in+\ell,y)=r \iff y=q+(i-k)m.
\]
Moreover
\[
0<q+(i-k)m \leq y_x \iff \rank_x \leq r<mx, 
\qquad 
0<q+(i-k)m \leq y_{x+1} \iff \rank_{x+1}-m \leq r<mx.
\]
Finally $x=in+\ell\geq p=kn+\ell \iff i\geq k$.
Therefore
\begin{align*}
\eta_\ell((y_{\bullet,\bullet});(p,q))
&= \sum_{i=k}^{d-1} \Bigl[\delta_{[\rank_{in+\ell+1}-m,\; m(in+\ell))}(r) 
   - \delta_{[\rank_{in+\ell},\; m(in+\ell))}(r)\Bigr] \\
&= \sum_{i=k}^{d-1} \Bigl[\delta_{(-\infty,\rank_{in+\ell})}(r) 
   - \delta_{(-\infty,\rank_{in+\ell+1}-m)}(r)\Bigr],
\end{align*}
using $\delta_{[a,b)}(r) = \delta_{(-\infty,b)}(r) - \delta_{(-\infty,a)}(r)$ for $-\infty < a \leq b < \infty$.
\end{proof}

Now, fix $D\isomorphic (y_{in+j} = y_{i,j})_{0\leq i\leq d-1,0\leq j\leq n-1} \in \Young(dn,dm)$.
Say, $\Delta \isomorphic (c_{in+j} = c_{i,j})_{0\leq i\leq d-1,0\leq j\leq n-1} \in \Adm(dn,dm)$.
For $0\leq u\leq d-1$, $0\leq i\leq d-1$, and $0\leq j\leq n-1$, recall
\begin{eqnarray*}
&&y_{in+j}^{[u]} = y_{i,j}^{[u]} := y_{\permutation_{j,d-1}\circ\cdots\circ\permutation_{j,d-u}(i),j} - (\permutation_{j,d-1}\circ\cdots\circ\permutation_{j,d-u}(i) - i)m;\\
&&\rank_{D,in+j}^{[u]} = \rank_{D,i,j}^{[u]} := \rank(in+j,y_{i,j}^{[u]}) = \rank_{D,\permutation_{j,d-1}\circ\cdots\permutation_{j,d-u}(i),j}.
\end{eqnarray*}
In particular, $c_{i,j} = y_{i,j}^{[d-1]} = y_{\permutation_j(i),j} - (\permutation_j(i) - i)m$.
Set $y_{dn}^{[u]} = y_{d,0}^{[u]} := dm$.
Via Definition \ref{def:sweep_map}, \textbf{set}
\begin{eqnarray}
&&\epsilon^{[u]}((p,q);x) := \epsilon((y_{\bullet,\bullet}^{[u]});(p,q);x);\quad \eta^{[u]}((p,q);x) := \eta((y_{\bullet,\bullet}^{[u]});(p,q);x);\nonumber\\
&&\epsilon_j^{[u]}(p,q) := \epsilon_j((y_{\bullet,\bullet}^{[u]});(p,q));\quad \eta_j^{[u]}(p,q) := \eta_j((y_{\bullet,\bullet}^{[u]});(p,q));\\
&&\epsilon^{[u]}(p,q) := \epsilon((y_{\bullet,\bullet}^{[u]});(p,q));\quad \eta^{[u]}(p,q) := \eta((y_{\bullet,\bullet}^{[u]});(p,q));
\quad \zeta^{[u]}(p,q) := \zeta((y_{\bullet,\bullet}^{[u]});p,q).\nonumber
\end{eqnarray}
For $u = 0$, we drop the superscript $[0]$ whenever convenient.

\begin{corollary}\label{cor:invariance_of_epsilon_j}
For $0\leq u \leq d-1$ and $0\leq j\leq n-1$, we have
\begin{equation}
\epsilon_j^{[u]}(p,q) = \epsilon_j(p,q).
\end{equation}
In particular, $\epsilon^{[u]}(p,q) = \epsilon(p,q)$.
\end{corollary}

\begin{proof}
Immediate from Lemma \ref{lem:alternative_description_of_index_functions}.$(1)$.
\end{proof}

\begin{lemma}\label{lem:invariance_of_eta}
Fix $0\leq p=kn+\ell \leq dn-1$, with $0\leq k\leq d-1$, $0\leq \ell \leq n-1$. For $1\leq u \leq d-1$, $0\leq j\leq n-1$, we have
\begin{equation}
\eta^{[d-u-1]}(kn+\ell,y_{k,\ell}^{[d-u-1]}) = \eta^{[d-u]}(\permutation_{\ell,u}^{-1}(k)n+\ell,y_{\permutation_{\ell,u}^{-1}(k),\ell}^{[d-u]}).
\end{equation}
Thus, $\eta(kn+\ell,y_{k,\ell}) = \eta^{[d-1]}(\permutation_{\ell}^{-1}(k)n+\ell,y_{\permutation_{\ell}^{-1}(k),\ell}^{[d-1]})$. $\iff$
$\eta(\permutation_{\ell}(k)n+\ell,y_{\permutation_{\ell}(k),\ell}) = \eta^{[d-1]}(p,y_p^{[d-1]})$.
\end{lemma}

\begin{proof}
Recall
\begin{eqnarray*}
&&\moveindex_u:= \max\{0\leq x\leq un:~\rank_{D,x}^{[d-u-1]} \leq \rank_{D,x+n-1}^{[d-u-1]} + m\};\\
&&\permutationindex_{j,u}:= \min\{0\leq i\leq u:~in+j \geq \moveindex_u\},\quad \permutation_{j,u}:= (\permutationindex_{j,u}~\cdots~u-1~u),\quad 0\leq j\leq n-1.
\end{eqnarray*}
In particular, $\permutation_{j,u}$ is identity on $[u+1,d-1]$. 
Besides, $\permutationindex_{n,u}:=\permutationindex_{0,u}$ and $\permutation_{n,u}:=\permutation_{0,u}$, and
\begin{equation*}
\permutationindex_{j,u} + \floor{\frac{j+1}{n}} \geq \permutationindex_{j+1,u} \geq \permutationindex_{j,u} + \floor{\frac{j+1}{n}} - 1.
\end{equation*}
Say, $\moveindex_u = k_un + \ell_u$, $0\leq k_u \leq u$, $0\leq \ell_u \leq n-1$.
Denote $\tilde{k} := \permutation_{\ell}^{-1}(k)$, and
\[
r:= \rank_{D,k,\ell}^{[d-u-1]} = \rank_{D,\tilde{k},\ell}^{[d-u]} \geq 0.
\]

First of all, since $\permutation_{\bullet,u}$ is identity on $[u+1,d-1]$, we have
$\rank_{D,i,j}^{[d-u-1]} = \rank_{D,i,j}^{[d-u]}$, $\forall u+1 \leq i\leq d-1$.
In particular, for any $u+1 \leq v \leq d-1$, we have
\begin{equation}\label{eqn:invariance_of_eta-trivial_part}
\sum_{i=v}^{d-1}[\delta_{(-\infty,\rank_{D,in+\ell}^{[d-u-1]})}(r) - \delta_{(-\infty,\rank_{D,in+\ell+1}^{[d-u-1]}-m)}(r)] = \sum_{i=v}^{d-1}[\delta_{(-\infty,\rank_{D,in+\ell}^{[d-u]})}(r) - \delta_{(-\infty,\rank_{D,in+\ell+1}^{[d-u]}-m)}(r)].
\end{equation}
\textbf{Note} that if $\ell = n-1$, we further have $\delta_{(-\infty,\rank_{D,un+\ell+1}^{[d-u-1]}-m)}(r) = \delta_{(-\infty,\rank_{D,un+\ell+1}^{[d-u]}-m)}(r)$.
Thus, if $k\geq u+1$, then $\permutation_{\ell}^{-1}(k) = k$, and by Lemma \ref{lem:alternative_description_of_index_functions}.$(2)$, this already shows the result.

\textbf{From now on}, assume that $0\leq k\leq u$, hence so is $\tilde{k} = \permutation_{\ell}^{-1}(k)$. 
By Lemma \ref{lem:alternative_description_of_index_functions}.$(2)$, it remains to show:
\begin{equation}\label{eqn:invariance_of_eta-nontrivial_part}
\sum_{i=k}^u[\delta_{(-\infty,\rank_{D,in+\ell}^{[d-u-1]})}(r) - \delta_{(-\infty,\rank_{D,in+\ell+1}^{[d-u-1]}-m)}(r)] = \sum_{i=\tilde{k}}^u[\delta_{(-\infty,\rank_{D,in+\ell}^{[d-u]})}(r) - \delta_{(-\infty,\rank_{D,in+\ell+1}^{[d-u]}-m)}(r)].
\end{equation}
For simplicity, \textbf{denote} the left hand side by $\LHS_{\ell}$ and the right hand side by $\RHS_{\ell}$.

\begin{enumerate}[wide,labelwidth=!,labelindent=0pt,itemindent=!]
\item 
If $p = kn+\ell < \moveindex_u$, then by definition, $\permutationindex_{\ell,u}\geq k+1$, hence $\permutationindex_{\ell+1,u} \geq k$.
It follows that $\permutation_{\ell,u}$ and $\permutation_{\ell+1,u}$ restrict to bijections on $[k,u] := \{k,k+1,\cdots,u\}$, and $\permutation_{\ell,u}(k) = k = \permutation_{\ell,u}^{-1}(k) = \tilde{k}$.
\begin{itemize}[wide,labelwidth=!,labelindent=0pt,itemindent=!]
\item 
If $0 \leq \ell \leq n-2$, we have
\begin{eqnarray*}
\LHS_{\ell} &=& \sum_{i=k}^u\delta_{(-\infty,\rank_{D,\permutation_{\ell,u}^{-1}(i),\ell}^{[d-u]})}(r) - \sum_{i=k}^u \delta_{(-\infty,\rank_{D,\permutation_{\ell+1,u}^{-1}(i),\ell+1}^{[d-u]}-m)}(r)\\
&=& \sum_{\tilde{i}=k}^u \delta_{(-\infty,\rank_{D,\tilde{i},\ell}^{[d-u]})}(r) - \sum_{\tilde{i}=k}^u \delta_{(-\infty,\rank_{D,\tilde{i},\ell+1}^{[d-u]}-m)}(r) = \RHS_{\ell}.
\end{eqnarray*}

\item
If $\ell = n-1$, then $\permutationindex_{\ell+1,u} \geq \permutationindex_{\ell,u} \geq k+1$. So, $\permutation_{0,u}(k) = k = \permutation_{0,u}^{-1}(k)$. Now,
\begin{eqnarray*}
\LHS_{n-1} &=& \sum_{i=k}^u\delta_{(-\infty,\rank_{D,\permutation_{n-1,u}^{-1}(i),n-1}^{[d-u]})}(r) - \sum_{i=k}^u \delta_{(-\infty,\rank_{D,\permutation_{0,u}^{-1}(i+1),0}^{[d-u]}-m)}(r)\\
&=& \sum_{\tilde{i}=k}^u \delta_{(-\infty,\rank_{D,\tilde{i},\ell}^{[d-u]})}(r) - \sum_{\tilde{i}=k}^u \delta_{(-\infty,\rank_{D,\tilde{i}+1,0}^{[d-u]}-m)}(r) = \RHS_{n-1}.
\end{eqnarray*}
\end{itemize}

\item
It remains to consider the case that $\moveindex_u \leq p < (u+1)n$. Then by definition, $\permutationindex_{\ell,u} \leq k$.
\begin{enumerate}[wide,labelwidth=!,labelindent=0pt,itemindent=!,label=(2.\arabic*)]
\item 
Suppose that $0 \leq \ell \leq n-2$.
Then, for each $k \leq i\leq u$, by Proposition \ref{prop:from_Dyck_paths_to_admissible_invariant_subsets}.$(c)$ and Remark \ref{rem:intermediate_diagrams_from_Dyck_paths_to_admissible_invariant_subsets} ($0\leq \ell \leq n-2$), we have
\[
r = \rank_{D,kn+\ell}^{[d-u-1]} > \rank_{D,(k+1)n+\ell}^{[d-u-1]} > \cdots > \rank_{D,in+\ell}^{[d-u-1]} \geq \rank_{D,in+\ell+1}^{[d-u-1]} - m.
\]
So, $\delta_{(-\infty,\rank_{D,in+\ell}^{[d-u-1]})}(r) = \delta_{(-\infty,\rank_{D,in+\ell+1}^{[d-u-1]} - m)}(r) = 0$, and hence $\LHS_{\ell} = 0$.
It remains to show that $\RHS_{\ell} = 0$.

\begin{itemize}[wide,labelwidth=!,labelindent=0pt,itemindent=!]
\item 
If $k = \permutationindex_{\ell,u}$, then $\tilde{k} = \permutation_{\ell,u}^{-1}(k) = u$.
In this case, by Remark \ref{rem:intermediate_diagrams_from_Dyck_paths_to_admissible_invariant_subsets}, we have
$r = \rank_{D,u,\ell}^{[d-u]} \geq \rank_{D,u,\ell+1}^{[d-u]} - m$. So,
\[
\RHS_{\ell} = \delta_{(-\infty,\rank_{D,un+\ell}^{[d-u]})}(r) - \delta_{(-\infty,\rank_{D,un+\ell+1}^{[d-u]}-m)}(r) = 0.
\]

\item
If $\permutationindex_{\ell,u} + 1 \leq k \leq u$, then $ \permutationindex_{\ell,u} \leq \tilde{k} = \permutation_{\ell,u}^{-1}(k) = k - 1 \leq u -1$.
Notice that $\tilde{p} = \tilde{k}n + \ell \geq \permutationindex_{\ell,u}n + \ell \geq \moveindex_u > \moveindex_{u-1}$. Then for each $\tilde{k} \leq i\leq u-1$, again by Proposition \ref{prop:from_Dyck_paths_to_admissible_invariant_subsets}.$(c)$ and Remark \ref{rem:intermediate_diagrams_from_Dyck_paths_to_admissible_invariant_subsets} ($0\leq \ell \leq n-2$),
\[
\sum_{i=\tilde{k}}^{u-1}[\delta_{(-\infty,\rank_{D,in+\ell}^{[d-u]})}(r) - \delta_{(-\infty,\rank_{D,in+\ell+1}^{[d-u]}-m)}(r)] = 0.
\]
It remains to show that $\delta_{(-\infty,\rank_{D,un+\ell}^{[d-u]})}(r) - \delta_{(-\infty,\rank_{D,un+\ell+1}^{[d-u]}-m)}(r) = 0$.
Indeed, by Proposition \ref{prop:from_Dyck_paths_to_admissible_invariant_subsets}.$(c)$,
\[
\rank_{D,un+\ell}^{[d-u]} = \rank_{D,\permutationindex_{\ell,u},\ell}^{[d-1-u]} > \rank_{D,k,\ell}^{[d-1-u]} = r,~~\Rightarrow~~\delta_{(-\infty,\rank_{D,un+\ell}^{[d-u]})}(r) = 1.
\]
On the other hand, we have $\permutationindex_{\ell+1,u} \leq \permutationindex_{\ell,u} \leq k - 1$. In particular, $(k-1)n + \ell + 1 \geq \permutationindex_{\ell,u}n + \ell + 1 \geq \moveindex_u + 1$. So by definition, $\rank_{D,k-1,\ell+1}^{[d-u-1]} > \rank_{D,k,\ell}^{[d-u-1]} +m = r + m$. 
Thus by Proposition \ref{prop:from_Dyck_paths_to_admissible_invariant_subsets}.$(c)$, we have
\[
\rank_{D,un+\ell+1}^{[d-u]} = \rank_{D,\permutationindex_{\ell+1,u},\ell+1}^{[d-u-1]} > \rank_{D,\permutationindex_{\ell+1,u}+1,\ell+1}^{[d-u-1]} > \cdots > 
\rank_{D,k-1,\ell+1}^{[d-u-1]} > r + m,~~\Rightarrow~~\delta_{(-\infty,\rank_{D,un+\ell+1}^{[d-u]}-m)}(r) = 1.
\]
Altogether, we again have $\RHS_{\ell} = 0 = \LHS_{\ell}$.
\end{itemize}

\item
We're left with the case $\ell = n-1$. So, $\permutationindex_{n-1,u} + 1 \geq \permutationindex_{0,u} \geq \permutationindex_{n-1,u}$.
Now, for each $k \leq i\leq u-1$, by Proposition \ref{prop:from_Dyck_paths_to_admissible_invariant_subsets}.$(c)$ and Remark \ref{rem:intermediate_diagrams_from_Dyck_paths_to_admissible_invariant_subsets}, we have
\[
r = \rank_{D,kn+n-1}^{[d-u-1]} > \rank_{D,(k+1)n+n-1}^{[d-u-1]} > \cdots > \rank_{D,in+n-1}^{[d-u-1]} \geq \rank_{D,in+n}^{[d-u-1]} - m.
\]
So, $\delta_{(-\infty,\rank_{D,in+n-1}^{[d-u-1]})}(r) = \delta_{(-\infty,\rank_{D,in+n}^{[d-u-1]} - m)}(r) = 0$.
Moreover, as $\permutationindex_{n-1,u} \leq k \leq u$, by Proposition \ref{prop:from_Dyck_paths_to_admissible_invariant_subsets}.$(c)$, 
\[
\rank_{D,un+n-1}^{[d-u-1]}> \cdots > \rank_{D,k,n-1}^{[d-u-1]} = r,~~\Rightarrow~~\delta_{(-\infty,\rank_{D,un+n-1}^{[d-u-1]})}(r)  = 0.
\]
Altogether,
\begin{equation}
\LHS_{n-1} = - \delta_{(-\infty,\rank_{D,(u+1)n}^{[d-u-1]} - m)}(r).
\end{equation}
Recall that $\delta_{(-\infty,\rank_{D,(u+1)n}^{[d-u-1]} - m)}(r) = \delta_{(-\infty,\rank_{D,(u+1)n}^{[d-u]}-m)}(r)$.


\begin{itemize}[wide,labelwidth=!,labelindent=0pt,itemindent=!]
\item 
If $k = \permutationindex_{n-1,u}$, then $\tilde{k} = \permutation_{n-1,u}^{-1}(k) = u$.
In this case, 
\[
\RHS_{n-1} = \delta_{(-\infty,\rank_{D,un+n-1}^{[d-u]})}(r) - \delta_{(-\infty,\rank_{D,(u+1)n}^{[d-u]}-m)}(r).
\]
By above, it remains to show that $ \delta_{(-\infty,\rank_{D,un+n-1}^{[d-u]})}(r) = 0$.
Indeed,
\[
\rank_{D,un+n-1}^{[d-u]} = \rank_{D,\permutationindex_{n-1,u},n-1}^{[d-u-1]} = r,~~\Rightarrow~~\delta_{(-\infty,\rank_{D,un+n-1}^{[d-u]})}(r) = 0.
\]

\item
If $\permutationindex_{n-1,u} + 1 \leq k \leq u$, then $ \permutationindex_{n-1,u} \leq \tilde{k} = \permutation_{n-1,u}^{-1}(k) = k - 1 \leq u -1$.
Notice that $\tilde{p} = \tilde{k}n + n-1 \geq \permutationindex_{n-1,u}n + n-1 \geq \moveindex_u > \moveindex_{u-1}$. Then for each $\tilde{k} \leq i\leq u-2$, again by Proposition \ref{prop:from_Dyck_paths_to_admissible_invariant_subsets}.$(c)$ and Remark \ref{rem:intermediate_diagrams_from_Dyck_paths_to_admissible_invariant_subsets},
\[
\sum_{i=\tilde{k}}^{u-2}[\delta_{(-\infty,\rank_{D,in+n-1}^{[d-u]})}(r) - \delta_{(-\infty,\rank_{D,(i+1)n}^{[d-u]}-m)}(r)] = 0.
\]
Second, as $u-1\geq \tilde{k}$, we have $r = \rank_{D,\tilde{k},n-1}^{[d-u]} > \rank_{D,\tilde{k}+1,n-1}^{[d-u]} > \cdots > \rank_{D,u-1,n-1}^{[d-u]}$.
So, $\delta_{(-\infty,\rank_{D,(u-1)n+n-1}^{[d-u]})}(r) = 0$. 
Third, $\delta_{(-\infty,\rank_{D,un+n-1}^{[d-u]})}(r) =  \delta_{(-\infty,\rank_{D,\permutationindex_{n-1,u},n-1}^{[d-u-1]})}(r)$. As $\permutationindex_{n-1,u} \leq k-1 \leq u-1$, we have 
\[
\rank_{D,\permutationindex_{n-1,u},n-1}^{[d-u-1]} > \cdots > \rank_{D,k-1,n-1}^{[d-u-1]} > \rank_{D,k,n-1}^{[d-u-1]} = r,
~~\Rightarrow~~\delta_{(-\infty,\rank_{D,un+n-1}^{[d-u]})}(r) = 1.
\]
Finally, we have $\permutationindex_{n-1,u} \leq \permutationindex_{0,u} \leq \permutationindex_{n-1,u} + 1 \leq k$. So, $kn \geq (\permutationindex_{n-1,u} + 1)n > \permutationindex_{n-1,u}n + n - 1 \geq \moveindex_u$. Then by definition of $\moveindex_u$, we have
$\rank_{D,k,0}^{[d-u-1]} > \rank_{D,k,n-1}^{[d-u-1]} + m = r + m$. 
On the other hand, as $\permutationindex_{0,u} \leq k \leq u$, by Proposition \ref{prop:from_Dyck_paths_to_admissible_invariant_subsets}.$(c)$, we now have
\[
\rank_{D,un}^{[d-u]} = \rank_{D,\permutationindex_{0,u},0}^{[d-u-1]} > \cdots > \rank_{D,k,0}^{[d-u-1]}  > r + m,~~\Rightarrow~~\delta_{(-\infty,\rank_{D,un}^{[d-u]}-m)}(r) = 1.
\]
Altogether, we obtain that
\[
\RHS_{n-1} = - \rank_{D,(u+1)n}^{[d-u]} =  -\rank_{D,(u+1)n}^{[d-u-1]} = \LHS_{n-1}.
\]
\end{itemize}
\end{enumerate}
\end{enumerate}
This completes the proof.
\end{proof}

Finally, recall that $c_{in+j} = c_{i,j} = y_{i,j}^{[d-1]}$. 
Moreover, $c_{i,j} \leq c_{i,j+1}$, $\forall 0\leq j\leq n-1$, with $c_{i,n}:= c_{i,0} + m$.
As before, $c_{dn} = c_{d,0} := dm$.
Now, for each $x = in+j$, $0\leq i\leq d-1$, $0\leq j\leq n-1$, \textbf{denote}
\begin{eqnarray}\label{eqn:sweep_map-index_functions_for_admissible_invariant_subsets}
&&\tilde{\epsilon}((p,q);x) := \#\{y\in (c_{i,j},c_{i,j+1}]: \rank(x,y) > \rank(p,q)\};\nonumber\\
&&\tilde{\eta}((p,q);x) := \delta_{[p,dn)}(x)\cdot [\#\{y\in[c_{i,j},c_{i,j+1}): \rank(x,y) = \rank(p,q)\};\nonumber\\
&&\tilde{\epsilon}_j(p,q) := \sum_{i=0}^{d-1} \tilde{\epsilon}((p,q);in+j);\quad \tilde{\epsilon}(p,q) := \sum_{j=0}^{n-1} \tilde{\epsilon}_j(p,q);\\
&&\tilde{\eta}_j(p,q) := \sum_{i=0}^{d-1} \tilde{\eta}((p,q);in+j);\quad \tilde{\eta}(p,q) := \sum_{j=0}^{n-1} \tilde{\eta}_j(p,q).\nonumber
\end{eqnarray}
Then,
\[
\cG((c_{\bullet,\bullet});(p,q)) = \tilde{\epsilon}(p,q) + \tilde{\eta}(p,q).
\]
Moreover, observe that
\begin{equation}\label{eqn:compare_index_functions}
\tilde{\epsilon}((p,q);x) = \epsilon^{[d-1]}((p,q);x),\quad \tilde{\eta}((p,q);x) =  \eta^{[d-1]}((p,q);x),\quad \forall 0\leq j\leq n-2.
\end{equation}

\begin{lemma}\label{lem:sweep_map-compare_two_epsilon_functions_for_admissible_invariant_subsets}
$\forall 0\leq j\leq n-1$, we have
\begin{equation}
\epsilon_j^{[d-1]}(p,q) = \tilde{\epsilon}_j(p,q).
\end{equation}
In particular, $\epsilon^{[d-1]}(p,q) = \tilde{\epsilon}(p,q)$.
\end{lemma}

\begin{proof}
This is trivial if $j\leq n-2$. It remains to consider the case $j = n-1$.
Fix $\constant \geq (d-1)m$, then
\begin{eqnarray*}
&&\tilde{\epsilon}_{n-1}(p,q) -  \epsilon_{n-1}^{[d-1]}(p,q) = \sum_{i=0}^{d-1}(\tilde{\epsilon}((p,q);in+n-1) -  \epsilon^{[d-1]}((p,q);in+n-1))\\
&=&\sum_{i=0}^{d-1}(\#\{y\in (0,c_{i,n}]: \rank(in+n-1,y) > \rank(p,q)\} - \#\{y\in (0,c_{i+1,0}]: \rank(in+n-1,y) > \rank(p,q)\})\\
&=&\sum_{i=0}^{d-1}\#\{\tilde{y}\in (-im,c_{i,0}-(i-1)m]: \rank(n-1,\tilde{y}) > \rank(p,q)\}\\
&& - \sum_{i=0}^{d-1}\#\{\tilde{y}\in (-im,c_{i+1,0}-im]: \rank(n-1,\tilde{y}) > \rank(p,q)\}\\
&=&\sum_{i=0}^{d-1}\#\{\tilde{y}\in (-\constant,c_{i,0}-(i-1)m]: \rank(n-1,\tilde{y}) > \rank(p,q)\}\\
&& - \sum_{i=0}^{d-1}\#\{\tilde{y}\in (-\constant,c_{i+1,0}-im]: \rank(n-1,\tilde{y}) > \rank(p,q)\}\\
&=& 0.\quad (c_{d,0}-(d-1)m = m = c_{0,0} - (0-1)m)
\end{eqnarray*}
\end{proof}

\begin{lemma}\label{lem:compare_two_eta_functions}
$\forall 0\leq j\leq n-1$, we have
\begin{equation}
\eta_j^{[d-1]}(p,c_p) = \tilde{\eta}_j(p,c_p).
\end{equation}
In particular, $\eta^{[d-1]}(p,c_p) = \tilde{\eta}(p,c_p)$.
\end{lemma}

\begin{proof}
Say, $p = kn+\ell$, $0\leq k \leq d-1$, $0\leq \ell \leq n-1$.
By a similar argument as in Lemma \ref{lem:invariance_of_eta}, it's clear that $ \tilde{\eta}_j(p,c_p) = 0$ if $j\neq \ell$.
Thus, $\tilde{\eta}(p,c_p) = \tilde{\eta}_{\ell}(p,c_p)$. It remains to show that
$\eta_{\ell}^{[d-1]}(p,c_p) = \tilde{\eta}_{\ell}(p,c_p)$.
In fact, the case $\ell\neq n-1$ follows immediately from the observation (\ref{eqn:compare_index_functions}).
Thus, we're left with the case $\ell = n-1$.

Now, by definition, we have
\begin{eqnarray*}
&&\tilde{\eta}_{n-1}((p,q)) -  \eta^{[d-1]}_{n-1}((p,q)) = \sum_{i=0}^{d-1}(\tilde{\eta}((p,q);in+n-1) -  \eta^{[d-1]}((p,q));in+n-1)\\
&=&\sum_{i=k}^{d-1}(\#\{y\in (0,c_{i,n}]: \rank(in+n-1,y) = \rank(p,q)\} - \#\{y\in (0,c_{i+1,0}]: \rank(in+n-1,y) = \rank(p,q)\}).
\end{eqnarray*}
From now on, take $q = c_p = c_{k,n-1}$. 
Fix $\constant \geq (d-1)m$.
Then, 
\begin{eqnarray*}
&&\tilde{\eta}((p,q)) -  \eta^{[d-1]}((p,q))\\
&=&\sum_{i=k}^{d-1}(\#\{\tilde{y}\in (-\constant,c_{i,n}-im]: \rank(\ell,\tilde{y}) = \rank(p,q)\} - \#\{\tilde{y}\in (-\constant,c_{i+1,0}-im]: \rank(\ell,\tilde{y}) = \rank(p,q)\})\\
&=& -\#\{\tilde{y}\in (c_{k,0}-(k-1)m,m]: \rank(n-1,\tilde{y}) = \rank(p,q)\}.
\end{eqnarray*}
Here, $c_{i,n} - im = c_{i,0} - (i-1)m$, $c_{d,0} - (d-1)m = m$, and $c_{k.0}\leq a_{k,0} = km$.
Moreover, if $\rank(n-1,\tilde{y}) = \rank(p,c_p)$, then $\tilde{y} = c_{k,n-1} - km \leq c_{k,0} - (k-1)m$. 
Thus, $\tilde{\eta}((c_{a,b});(p,c_p)) -  \eta((c_{a,b});(p,c_p)) = 0$, as desired.
\end{proof}

Now, it's clear how to prove Proposition \ref{prop:sweep_map_vs_bijection}.
\begin{proof}[Proof of Proposition \ref{prop:sweep_map_vs_bijection}]
Fix $\forall D \isomorphic (y_{in+j} = y_{i,j})_{0\leq i\leq d-1,0\leq j\leq n-1} \in \Young(dn,dm)$, with $\Delta \isomorphic (c_{\bullet,\bullet}) := \Psi_d((y_{\bullet,\bullet})) \in \Adm(dn,dm)$. 
Recall that $c_{i,j} = y_{i,j}^{[d-1]} = y_{\permutation_j(i),j} - (\permutation_j(i) - i)m$, and $\rank_{\Delta,i,j} = \rank(in+j,c_{i,j}) = \rank_{D,i,j}^{[d-1]} = \rank_{D,\permutation_j(i),j} = \rank(\permutation_j(i)n+j,y_{\permutation_j(i),j})$. In particular, $e(\Delta) = |(c_{i,j})| = |D|$ is clear.

Now, fix any $p = kn + \ell$, $0\leq k\leq d-1$, $0 \leq \ell \leq n-1$.
By definition, we have $\epsilon_j^{[d-1]}(p,c_p) = \epsilon_j^{[d-1]}(\permutation_{\ell}(k)n+\ell,y_{\permutation_{\ell}(k),\ell})$. Then by Corollary \ref{cor:invariance_of_epsilon_j}, we have $\epsilon_j^{[d-1]}(\permutation_{\ell}(k)n+\ell,y_{\permutation_{\ell}(k),\ell}) = \epsilon_j(\permutation_{\ell}(k)n+\ell,y_{\permutation_{\ell}(k),\ell})$. 
Combined with Lemma \ref{lem:sweep_map-compare_two_epsilon_functions_for_admissible_invariant_subsets}, we then obtain:
\begin{equation}
\tilde{\epsilon}_j(p,c_p) = \epsilon_j^{[d-1]}(p,y_p^{[d-1]}) = \epsilon_j(\permutation_{\ell}(k)n+\ell,y_{\permutation_{\ell}(k),\ell}).
\end{equation}
In particular, $\tilde{\epsilon}(p,c_p) = \epsilon^{[d-1]}(p,y_p^{[d-1]}) = \epsilon(\permutation_{\ell}(k)n+\ell,y_{\permutation_{\ell}(k),\ell})$.

On the other hand, by Lemma \ref{lem:invariance_of_eta}, and Lemma \ref{lem:compare_two_eta_functions}, we have 
\[
\tilde{\eta}_j(p,c_p) = \eta_j^{[d-1]}(p,y_p^{[d-1]}) = \eta_j(\permutation_{\ell}(k)n+\ell,y_{\permutation_{\ell}(k),\ell}).
\]
In particular, $\tilde{\eta}(p,c_p) = \eta^{[d-1]}(p,y_p^{[d-1]}) = \eta(\permutation_{\ell}(k)n+\ell,y_{\permutation_{\ell}(k),\ell})$.

Altogether, we obtain that
\begin{eqnarray*}
\cG((c_{\bullet,\bullet});(p,c_p)) &=& \tilde{\epsilon}((p,c_p)) + \tilde{\eta}((p,c_p)) = \epsilon^{[d-1]}((p,y_p^{[d-1]})) + \eta^{[d-1]}((p,y_p^{[d-1]}))\\
&=& \epsilon(\permutation_{\ell}(k)n+\ell,y_{\permutation_{\ell}(k),\ell}) + \eta(\permutation_{\ell}(k)n+\ell,y_{\permutation_{\ell}(k),\ell})
=  \zeta((y_{\bullet,\bullet});(\permutation_{\ell}(k)n+\ell,y_{\permutation_{\ell}(k),\ell})).
\end{eqnarray*}
In particular, $\cG((c_{\bullet,\bullet})) = \zeta((c_{a,b}))$, and hence $\dim\Delta = |\cG(\Delta)| = |\zeta(D)|$.

Finally, $|\zeta(D)| = \codinv(D)$ follows from a standard argument.
For completeness, we include the details.
Observe that for $x = p$, we have $\rank(x,y) \geq \rank(p,y_p)$ $\Leftrightarrow$ $y \leq y_p = y_x$. So by definition, $|\zeta(D)| = \mathrm{I} + \mathrm{II}$, with
\begin{eqnarray*}
\mathrm{I} &:=& \#\{(p,x,y): 0\leq x < p\leq dn-1, y_x < y\leq y_{x+1}, \rank(x,y) > \rank(p,y_p)\};\\
\mathrm{II} &:=& \#\{(p,x,y): 0\leq p < x \leq dn-1, y_x < y\leq y_{x+1}, \rank(x,y) \geq \rank(p,y_p)\}.
\end{eqnarray*}
In addition, $\forall 0\leq p\leq dn-1$, $0\leq x\leq dn-1$, $y_x < y \leq y_{x+1}$, observe that the pair $(p,y_p)$, $(x,y)$ \textbf{determine} and are \textbf{uniquely determined by} a unique box $z = (p,y) \in \cR_{dn,dm}^+$ as in Figure \ref{fig:box_counting_for_co-dinv}. Moreover,
\begin{enumerate}[wide,labelwidth=!,labelindent=0pt,itemindent=!]
\item 
$x < p$ $\Leftrightarrow$ $z \in D$. In this case, $(\armlength,\leglength)(z) = (p-x-1,y_p-y)$. So,
\[
\rank(x,y) > \rank(p,y_p)~~\Leftrightarrow~~ m(\armlength(z)+1) < n\leglength(z) ~~\Leftrightarrow~~ \frac{\leglength(z)}{\armlength(z)+1} > \frac{m}{n}.
\]
It follows that $\mathrm{I} = \#\{z\in D: \frac{\leglength(z)}{\armlength(z)+1} > \frac{m}{n}\}$.

\item
$x > p$ $\Leftrightarrow$ $z \in \cR_{dn,dm}^+\setminus D$. In this case, $(\armlength,\leglength)(z) = (x - p,y - y_p - 1)$. So,
\[
\rank(x,y) \geq \rank(p,y_p)~~\Leftrightarrow~~ m\armlength(z) \geq n(\leglength(z)+1) ~~\Leftrightarrow~~ \frac{\leglength(z)+1}{\armlength(z)} \leq \frac{m}{n}.
\]
It follows that $\mathrm{II} = \#\{z\in \cR_{dn,dm}^+\setminus D: \frac{\leglength(z)+1}{\armlength(z)} \leq \frac{m}{n}\}$.
\end{enumerate}

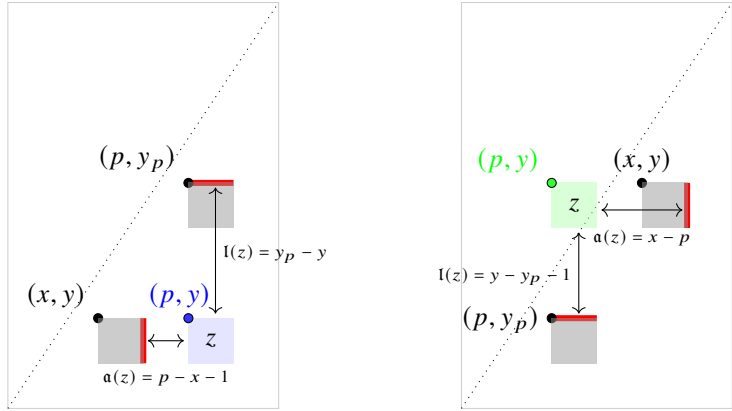
\begin{figure}[!htbp]
\vspace{-0.2in}
\begin{center}
\begin{tikzcd}[column sep=2pc,ampersand replacement=\&]

\begin{tikzpicture}[scale=0.6]

\draw[line width = 0.01mm,opacity=0.2] (0,0) rectangle (6,9);

\draw[thin,dotted] (0,0) -- (6,9);

\draw[red, line width=0.8mm] (3,1) -- (3,2);

\draw[red, line width=0.8mm] (4,5) -- (5,5);

\draw[fill=black] (2,2) circle (0.1);

\fill[black!40, opacity=0.5] (2,1) -- (2,2) -- (3,2) -- (3,1) -- (2,1);

\node[above left] at (2,2) {$(x,y)$};

\draw[fill=black] (4,5) circle (0.1);

\fill[black!40, opacity=0.5] (4,5) -- (5,5) -- (5,4) -- (4,4) -- (4,5);

\node[above left] at (4,5) {$(p,y_p)$};

\fill[blue!20, opacity=0.5] (4,2) -- (5,2) -- (5,1) -- (4,1) -- (4,2);

\node at (4.5,1.4) {$z$};

\draw[fill=blue!80] (4,2) circle (0.1);

\node[above] at (3.8,2) {{\color{blue}$(p,y)$}};

\draw[thin,<->] (3.1,1.5) -- (3.9,1.5);

\node[below] at (3.5,1.1) {\tiny$\armlength(z) = p - x - 1$};

\draw[thin,<->] (4.6,2.1) -- (4.6,4.9);

\node[right] at (4.5,3.5) {\tiny$\leglength(z) = y_p - y$};

\end{tikzpicture}

\&

\begin{tikzpicture}[scale=0.6]

\draw[line width = 0.01mm,opacity=0.2] (0,0) rectangle (6,9);

\draw[thin,dotted] (0,0) -- (6,9);

\draw[red, line width=0.8mm] (2,2) -- (3,2);

\draw[red, line width=0.8mm] (5,4) -- (5,5);

\draw[fill=black] (2,2) circle (0.1);

\fill[black!40, opacity=0.5] (2,1) -- (2,2) -- (3,2) -- (3,1) -- (2,1);

\node[left] at (2,2) {$(p,y_p)$};

\draw[fill=black] (4,5) circle (0.1);

\fill[black!40, opacity=0.5] (4,5) -- (5,5) -- (5,4) -- (4,4) -- (4,5);

\node[above] at (4,5) {$(x,y)$};

\fill[green!30, opacity=0.5] (2,5) -- (3,5) -- (3,4) -- (2,4) -- (2,5);

\node at (2.5,4.4) {$z$};

\draw[fill=green!80] (2,5) circle (0.1);

\node[above left] at (2,5) {{\color{green}$(p,y)$}};

\draw[thin,<->] (2.6,2.1) -- (2.6,3.9);

\node[left] at (2.7,3) {\tiny$\leglength(z) = y - y_p - 1$};

\draw[thin,<->] (3.1,4.4) -- (4.9,4.4);

\node[below] at (4,4.2) {\tiny$\armlength(z) = x - p$};

\end{tikzpicture}

\end{tikzcd}
\end{center}
\vspace{-0.2in}
\caption{The box $z = (p,y) \in \cR_{dn,dm}$ and the pair $(p,y_p)$, $(x,y)$ with $y_x < y \leq y_{x+1}$ determine each other. The red paths correspond to part of the Dyck path. Left: $x<p$, $\Leftrightarrow$ $z\in D$; Right: $x>p$, $\Leftrightarrow$ $z\in \cR_{dn,dm}\setminus D$.}
\label{fig:box_counting_for_co-dinv}
\end{figure}

Finally, by \cite[Cor.3.4]{Maz14}, we have 
\[
\mathrm{II} = |\cR_{dn,dm}^+\setminus D| + \#\{z\in D: \frac{\leglength(z)+1}{\armlength(z)} \leq \frac{m}{n}\} = \delta - |D| +  \#\{z\in D: \frac{\leglength(z)+1}{\armlength(z)} \leq \frac{m}{n}\}.
\]
Altogether, $|\zeta(D)| = \delta - \#\{z\in D: \frac{\leglength(z)}{\armlength(z)+1} \leq \frac{m}{n} < \frac{\leglength(z)+1}{\armlength(z)}\} = \delta - \dinv(D) = \codinv(D)$.
Done.
\end{proof}

\end{document}